\documentclass[preprint,12pt]{amsart}

\usepackage{latexsym,amscd,amsmath,verbatim,amssymb,calc,enumitem,amsxtra,multicol,mathrsfs,amsthm,ifpdf, xcolor, xspace,mathtools}
\usepackage[all]{xy}

\ifpdf
  \usepackage[breaklinks=true]{hyperref}
\else
  \usepackage[hypertex]{hyperref}
\fi

\DeclareMathAlphabet{\mathpzc}{OT1}{pzc}{m}{it}    
\theoremstyle{plain}
\newtheorem{theorem}{Theorem}[section]
\newtheorem{corollary}[theorem]{Corollary}
\newtheorem{proposition}[theorem]{Proposition}
\newtheorem{lemma}[theorem]{Lemma}
\newtheorem*{theorem*}{Theorem}
\newtheorem*{proposition*}{Proposition}

\theoremstyle{remark}
\newtheorem{remark}[theorem]{Remark}

\newtheorem{notation}[theorem]{Notation}

\newtheorem{discussion}[theorem]{Discussion}

\theoremstyle{definition}
\newtheorem{definition}[theorem]{Definition}

\numberwithin{equation}{theorem}

\DeclareFontFamily{U}{mathx}{}                                         
\DeclareFontShape{U}{mathx}{m}{n}{<-> mathx10}{}         
\DeclareSymbolFont{mathx}{U}{mathx}{m}{n}
\DeclareMathAccent{\widehat}{0}{mathx}{"70}
\DeclareMathAccent{\widecheck}{0}{mathx}{"71}

\newcommand{\new}{\newcommand}
\newcommand{\0}{^{\circ}}
\new{\oh}{^{\circ\circ}}
\new{\8}{^{\infty}}
\newcommand{\Ann}{\mathrm{Ann}}
\newcommand{\ann}{\operatorname{Ann}}
\newcommand\arwf[1]{\xrightarrow{#1}}
     
\newcommand{\Ass}{\operatorname{Ass}}
\newcommand{\ass}{\operatorname{Ass}}
\newcommand{\bE}{\mathbf{E}}
\newcommand\bem[1]{\bibitem[#1]{#1}}
\newcommand{\ben}{\begin{enumerate}}
\newcommand{\bena}{\ben[label=$({\rm \alph*})$]}
\newcommand{\beni}{\ben[label=$({\rm \roman*})$]}
\newcommand{\benn}{\ben[label=$({\rm \arabic*})$]}
\newcommand{\benr}{\ben[resume*]}
\newcommand\bensp[1]{\ben[label=#1]}
\newcommand{\blank}{\underline{\hbox{\ \ }}}

\newcommand{\bpm}{\begin{pmatrix}}
\newcommand{\bu}{{\bullet}}

\newcommand{\cC}{\mathcal{C}}
\new{\ccC}{\widecheck{\cC}}
\newcommand\Cech{\v{C}ech\ }

\newcommand{\cF}{\mathcal{F}}

\newcommand{\cH}{\mathcal{H}}
\newcommand{\cI}{\mathcal{I}}

\newcommand{\cM}{\mathcal{M}}
\newcommand{\CM}{Cohen-Macaulay\xspace}
\newcommand{\cP}{\mathcal{P}}

\newcommand{\cof}{\prec \kern -4 pt \prec_{\mathrm{cof}}}

\newcommand{\Coker}{\operatorname{Coker}}

\newcommand{\cR}{\mathcal{R}}

\newcommand{\cts}{^{\mathrm{cts}}}

\newcommand{\dagref}{\hyperref[dagger]{$(\dagger)$}\xspace}

\newcommand{\depth}{\operatorname{depth}}

\newcommand{\di}{\mathrm{dim}}
\newcommand{\disp}{\displaystyle}
\newcommand{\dlim}[1]{\varinjlim_{#1}}

\newcommand{\e}{{}^e\!}
\newcommand{\E}{\operatorname{E}} 
\newcommand{\edc}{\end{document}}
\newcommand{\een}{\end{enumerate}}
 
\newcommand{\epm}{\end{pmatrix}}

\newcommand{\ext}{\operatorname{Ext}}
\newcommand{\Ext}{\operatorname{Ext}}

\newcommand{\fa}{\mathfrak{a}}

\newcommand{\fe}{\mathfrak{e}}

\newcommand{\ff}{free filterable relative to $P$\xspace} 
 
\newcommand{\fffP}{free filterable relative to $\fP$\xspace}
\newcommand{\ffr}{free filterable relative to\xspace}
\newcommand{\fg}{finitely generated\xspace}

\newcommand{\fm}{\mathfrak{m}}
\newcommand{\fn}{\mathfrak{n}}

\newcommand{\fp}{\mathfrak p}
\newcommand{\fP}{\mathfrak{P}}

\newcommand{\fq}{\mathfrak q}

\newcommand{\fra}{\operatorname{frac}}

\renewcommand{\ge}{\geqslant} \renewcommand{\le}{\leqslant} 
\renewcommand{\geq}{\geqslant} \renewcommand{\leq}{\leqslant} 

\newcommand{\gr}{\operatorname{gr}}
\newcommand{\h}{^\mathrm{h}}
\newcommand{\height}{\operatorname{height}}

\renewcommand{\hom}{\operatorname{Hom}}
\newcommand{\Hom}{\operatorname{Hom}}

\newcommand{\id}{\operatorname{I}}
\newcommand{\injdim}{\mathrm{injdim}}

\newcommand{\im}{\operatorname{Im}}
\newcommand{\imp}{\Rightarrow}
\newcommand{\inc}{\subseteq}
\newcommand{\incn}{\subsetneq}
\newcommand{\inj}{\hookrightarrow}
\newcommand{\Ker}{\operatorname{Ker}}
\renewcommand{\ker}{\operatorname{Ker}}

\newcommand{\m}{\mathfrak m}
\newcommand{\Max}{\operatorname{Max}}

\newcommand{\Min}{\operatorname{Min}}

\newcommand{\N}{\mathbb{N}}

\newcommand{\noi}{\noindent}
\newcommand{\nw}[1]{{}^{#1} \!}
 
\new{\oM}{\ov{M}}

\renewcommand{\phi}{\varphi}

\newcommand{\ol}{\overline}
\newcommand{\om}{\omega}
\newcommand{\op}{\operatorname}

\newcommand{\ov}{\overline}

\new{\oWipr}{\ov{W}'\kern -6pt{}_I}
\newcommand{\page}{\mathbf E^{\bullet\bullet}}

\newcommand{\rmk}{(R, \, \fm, K)}
\newcommand\Rp[2]{$({#1}\, \backslash \!\! \backslash {#2})$-pseudoflat\xspace}
\newcommand{\Rpp}{$(R\, \backslash \!\! \backslash P)$-pseudoflat\xspace}
\newcommand{\Rpf}{$(R\, \backslash \!\! \backslash \fP)$-pseudoflat\xspace}

\newcommand{\sA}{\mathscr{A}}
 
\new{\ser}[1]{$\op{S}_{#1}$\xspace} \new{\Ser}{\ser} \new{\stwo}{\ser 2}\new{\Stwo}{\stwo} \new{\sone}{\ser 1}\new{\Sone}{\sone}
\new{\sk}{\hskip 2 pt}
\newcommand{\sm}{\setminus}
\newcommand{\smo}{\setminus \{0\}}
\newcommand{\sop}{system of parameters\xspace}
\newcommand{\spec}{\operatorname{Spec}}
\newcommand{\Spec}{\operatorname{Spec}}

\newcommand{\surj}{\twoheadrightarrow}

\newcommand{\tth}{{\widetilde{h}}}

\newcommand{\tg}{{\widetilde{g}}}
\newcommand{\tE}{{\widetilde{E}}}

\newcommand{\tfp}{{\widetilde{\fp}}}

\newcommand{\tor}{\operatorname{Tor}}
\newcommand{\Tor}{\operatorname{Tor}}
\newcommand{\tP}{{\widetilde{P}}}
\newcommand{\tom}{\widetilde{\omega}}
\newcommand{\tQ}{{\widetilde{Q}}}
\newcommand{\tR}{{\widetilde{R}}}
\newcommand{\tW}{\widetilde{W}}
\newcommand{\twg}{$\tW\!$-generically\xspace} 
\new{\tX}{\wt{X}}

\new{\tZ}{\wt{Z}}
\newcommand{\ul}{\underline}
\newcommand{\uu}{\ul{u}}
\newcommand{\uU}{\ul{U}}
\newcommand{\ux}{\underline{x}}
\newcommand{\uy}{\underline{y}}
\newcommand{\uz}{\underline{z}}
\newcommand{\uZ}{\ul{Z}}

\newcommand{\vect}[2]{{#1}_1, \, \ldots, \, {#1}_{#2}}
\newcommand{\vct}[3]{{#1}^{#2}_1, \,  \ldots, \, {#1}^{#2}_{#3}}

\newcommand{\vP}{{\breve{P}}}
\newcommand{\vQ}{{\breve{Q}}}
\newcommand{\vR}{{\breve{R}}}

\newcommand{\vtR}{{\breve{\tR}}}
\newcommand{\vtom}{{\breve{\tom}}}
\newcommand{\wg}{$W\!$-generically\xspace} 
\newcommand{\wh}{\widehat}

\newcommand{\wcH}{\widecheck{H}}
\newcommand{\wt}{\widetilde}
 
\newcommand{\xra}{\xrightarrow} 
\newcommand{\Z}{\mathbb{Z}}

\newcommand{\disappear}[1]{}

\definecolor{grn}{rgb}{0.0, 0.5, 0.0}
\definecolor{gre}{rgb}{0.0, 0.6, 0.0}
\definecolor{prp}{rgb}{0.4, 0.0, 0.3}
\definecolor{ora}{rgb}{0.8, 0.4, 0.0}
\new{\gre}[1]{{\color{gre} #1}}

\def\todo#1
{\par\vskip 7 pt \noindent {\textcolor{yellow}\footnotesize \color{yellow} \fbox{\parbox{4.8in}{\color{black}
				\textbf{To do: } {\color{blue}#1}}}} \vskip 3 pt \par}
\def\forth#1
{\par\vskip 7 pt \noindent {\textcolor{yellow}\footnotesize \color{yellow} \fbox{\parbox{4.8in}{\color{black}
				\textbf{For thought: } {\color{blue}#1}}}} \vskip 3 pt \par}
			
\allowdisplaybreaks[1]

\begin{document}

\title[Generic local duality and purity exponents] 
      {Generic local duality and\\ purity exponents} 

\author{Mel Hochster}
\address{Department of Mathematics, University of Michigan, 
         Ann Arbor, MI 48109}
\curraddr{}
\email{hochster@umich.edu}
\author{Yongwei Yao}
\address{Department of Mathematics and Statistics, Georgia State University, 
         Atlanta, GA 30303}
\curraddr{}
\email{yyao@gsu.edu}
\thanks{The first author was partially supported by National Science Foundation grants DMS-1902116 and DMS-2200501.}



\date{\today} 

\maketitle

\begin{abstract}
We prove a form of generic local duality that generalizes a result of Karen E.~Smith.
Specifically, let $R$ be a Noetherian ring, let $P$ be a prime ideal of $R$ of height $h$, let $A:=R/P$,
and $W$  be a subset  of $R$ that maps onto $A\setminus \{0\}$.  Suppose that $R_P$ is Cohen-Macaulay, and
that $\omega$ is a finitely generated $R$-module such that $\omega_P$ is a canonical module for $R_P$.
Let $E:=H^h_P(\omega)$.  We show that for every finitely generated $R$-module $M$ there exists $g \in W$
such that for all $j\geq 0$,  $H_P^j(M)_g \cong \operatorname{Hom}_R(\Ext_R^{h-j}(M,\, \omega),\, E)_g$, and that, 
moreover, every $H_P^j(M)_g$ has an ascending filtration by a countable sequence of finitely generated 
submodules such that the factors are finitely generated free $A_g$-modules. In fact,  this sequence may 
be taken to be $\{\operatorname{Ann}_{H_P^j(M)_g}P^n\}_n$. We use this result to obtain a strengthened form of the 
base change results proved by Smith.  We also use this result to study the purity exponent for a 
nonzerodivisor $c$  in a reduced excellent Noetherian ring $R$ of prime characteristic $p$, 
which is the least  $e \in \mathbb{N}$ such that the map $R \to R^{1/p^e}$ with $1 \mapsto c^{1/p^e}$ 
 is pure. In particular, in the case where $R$ is a homomorphic image of an excellent Cohen-Macaulay ring and is 
 $\operatorname{S}_2$,  we establish an upper semicontinuity result for the function 
 $\mathfrak{e}_c:\operatorname{Spec}(R) \to \mathbb{N} \cup \{\infty\}$, where 
 $\mathfrak{e}_c(P)$ is the purity exponent for the image of $c$ in $R_P$.  This result enables us to 
 prove that excellent strongly F-regular rings are very strongly F-regular (also called F-pure regular). 
Other consequences are that the strongly F-regular locus is open in every excellent Noetherian 
ring of prime characteristic $p > 0$ and that
 the F-pure locus is open in every $\operatorname{S}_2$ ring that is a homomorphic image of an excellent 
 Cohen-Macaulay ring.
\end{abstract} 

\bigskip
{\leftskip 71pt \rightskip 71pt \baselineskip 11pt
{\footnotesize \noi \it This paper celebrates Karen E.~Smith's remarkable and fundamental contributions
to commutative algebra and algebraic geometry, which have been an inspiration for the research of many mathematicians.}\par} 

\section{Introduction}

Throughout this paper, unless otherwise specified, all rings are assumed to be commutative, associative with multiplicative 
identity and modules are unital.  Given rings are usually Noetherian. 
We let $\N_+ \subsetneq \N \subsetneq \Z$ be the set of positive integers, the set of nonnegative integers, and the set 
of integers, respectively.

We prove a form of generic local duality that generalizes a result of Karen E.~Smith \cite{Sm18}.  Our result
may be described as follows. 
Let $R$ be a Noetherian ring, let $P$ be a prime ideal of $R$ of height $h$, let $A:=R/P$,
and let $W$  be a subset of $R$ that maps onto $A\smo$ under the natural homomorphism $R \to R/P$.  
Suppose that $R_P$ is Cohen-Macaulay, and
that $\om$ is a finitely generated $R$-module such that $\om_P$ is a canonical module for $R_P$.
Let $E:=H^h_P(\om)$.  We show that for every finitely generated $R$-module $M$ there exists $g \in W$
such that for all $j \in \N$,
\footnote{Note that in the statement above, if either $R$, or $R_{g'}$ for some $g' \in R\setminus P$, is 
excellent and has finite Krull dimension, then $g \in W$ can be chosen so that the isomorphism 
$H_P^j(M)_g \cong \hom_R(\Ext_R^{h-j}(M,\, \om), \,E)_g$ holds {\it for all} $j \in \Z$. See Theorem~\ref{main}, part~(e).}
$H_P^j(M)_g \cong \hom_R(\Ext_R^{h-j}(M,\, \om), \,E)_g$, and 
$H := H_P^j(M)_g$ satisfies the following condition:
\bensp{{($\dagger$)}}\phantomsection \label{dagger}
\item The module $H$ has an ascending filtration by a countable sequence of finitely generated $R_g$-submodules 
such that the factors are  free $A_g$-modules and the union of the submodules in the sequence is the whole module.
More precisely, we show that  $\{\Ann_{H_P^j(M)_g}P^n\}_n$ gives such a filtration, and
the factors are $A_g$-free of finite rank.
\een

We apply these results to obtain a strengthened form of base change results proved in \cite{Sm18}.
We also note the related papers \cite{ChCidSim22, Cid23, CidSmi24}.
We point the reader to Theorems \ref{main}, \ref{bc-reduced}, \ref{bc0-reduced} and \ref{bc-geom}
for statements of our main results on filtering local cohomology and on base change. 

We call a module satisfying \dagref \emph{free filterable relative to} $P_g$.  The formal definition of
this notion is given in Definition~\ref{def:ff}. 

In the situation studied by Smith in \cite{Sm18}, $A$ is a subring of $R$ such that 
$A$ is isomorphic to $R/P$ under the natural map $R \to R/P$, and $W = A\smo$. 
We compare the results of this manuscript and those of \cite{Sm18} further in \S\ref{Smith}.  
In the applications here, there is usually no copy of $R/P$ in $R$ (i.e., $R \surj R/P$ does not split as a map of rings).  
Most frequently, $W$  is simply $R \sm P$. 
A key property of $R$-modules $M$ that are \ff  is that every sequence of elements in $R$ that maps to a possibly
improper regular sequence in $A$ is a possibly improper regular sequence\footnote{A sequence $\vect f s$ of 
elements of a ring 
$R$ is called a {\it possibly improper regular sequence} on an $R$-module $M$ (which is often the ring) 
if for all $0 \leq i \leq s-1$, the element $f_{i+1}$ is a nonzerodivisor on $M/(\vect fi)M$.  Note that we do 
not require $(\vect f s)M \not= M$.}  on $M$;   moreover, this property is preserved if
we make a flat base change $R \to R'$.
In fact, this  property  also holds under substantially weaker assumptions than in
the definition of \ff.  We want to study this property even when $P$ is replaced by an ideal $\fP$ that
is not necessarily prime. 

We are therefore led to make the following definition:

\begin{definition} \label{def:Rpf} If $\fP  \subsetneq R$ is any proper ideal, 
we define an $R$-module $M$ to be  {\it \Rpf} if, for every flat $R$-algebra $R'$, a  sequence in $R'$ that maps to
a possibly improper regular sequence in $R'/\fP R'$ is a possibly improper regular sequence on 
$R' \otimes _RM$. 
\end{definition}

If $M$ is killed by $P$ and $A:= R/P$ is regular, then $M$ is \Rpp if and only if $M$ is $A$-flat 
(see Proposition~\ref{reRpf}(i)). 
In this paper we are often interested in the case where $M$ is an $R$-module, not necessarily finitely
generated, $P$ is a prime ideal in $R$ such that $A:=R/P$ becomes regular after suitable localization, 
and such that every element of $M$ is killed by a power of $P$.  We briefly develop some facts about \Rpf modules 
in \S\ref{Rpf}.

In \S\ref{pureexp}, we use our results to study the purity exponent for a nonzerodivisor $c$ in a ring $R$ of 
 prime characteristic $p >0$ that is a homomorphic
 image of an excellent \CM ring and is \stwo.  In the case where $R$ is reduced, the purity exponent for $c$
 is the least integer $e \in \N$ such that the $R$-linear map $R \to R^{1/p^n}$ with $1 \mapsto c^{1/p^n}$
 is pure for all $n \ge e$. In particular, we establish an upper semicontinuity result for the function $\fe_c:\spec(R) \to \N \cup \{\infty\}$, where 
 $\fe_c(P)$ is the purity exponent for the image of $c$ in $R_P$.  This result enables us to prove that excellent strongly 
 F-regular rings are very strongly F-regular (also called F-pure regular). The terminology is explained in \S\ref{vstfreg}.   
 This result was previously known when $R$ is F-finite, when $R$ is essentially of finite type over an 
 excellent semilocal ring, and in a handful of other cases.  We refer the reader to 
 \cite{Hash10}, \cite{DaMuSm20}, \cite{DaSm16},  \cite{HoY23}, and \cite{DEST25}  for these earlier results.  
The semicontinuity result also implies that  the strongly F-regular locus is open in every excellent Noetherian ring, and also that
 the F-pure locus is open in every ring that is an \stwo image of an excellent \CM ring.

 A technical problem that arises is that an excellent Cohen-Macaulay local ring $R_P$ may fail to have
 a canonical module.  This difficulty can be overcome because there is always a local \'etale extension
 of $R_P$ with the same residue class field that has a canonical module. See \S\ref{canon}.
 
There are no restrictions on the characteristic of the ring in \S\S\ref{genld}--\ref{canon}.  The results in \S\ref{pureexp} typically need the assumption that the ring has prime characteristic $p > 0$.  

For background results in commutative algebra, see \cite{BH93, Mat80, Mat87, Na62}.

We note that using the results of an earlier version of this paper, S.~Lyu \cite[Appendix~A]{Lyu25} has generalized our Theorems~\ref{main}, \ref{mainsmct}, \ref{strong-locus}, and Corollary~\ref{pure-locus}
by substantially weakening the hypotheses. See Remarks~\ref{Lyu-Ch2} and \ref{Lyu-Ch8} for details.

We would like to thank Hailong Dao, Rankeya Datta, and Thomas Polstra for their comments on an 
earlier version of this manuscript.

\section{Generic local duality}\label{genld} 

In this section and the next, in Theorems~\ref{main}, \ref{bc-reduced}, \ref{bc0-reduced} and~\ref{bc-geom}, we state and 
prove our main results on  generic local duality and base change. We begin with a brief recap of the basic local 
cohomology results that we are generalizing.

\begin{discussion}[Local duality]\label{recap} We give a brief treatment of the basic facts about local duality over 
a Cohen-Macaulay  local ring $R$.  We refer the reader to \cite{Gro67}  and \cite[Ch.~3]{BH93} for a detailed 
discussion. Throughout this discussion,  
$(R,\,P,\,\kappa)$ is a Cohen-Macaulay local ring of Krull dimension $h$, and $\om$ is a canonical module 
for $R$. The following statements are part of the standard theory. Let $M$ be a finitely generated 
$R$-module (although some statements hold without finite generation).

\benn 

\item The module $\omega$ is a small (i.e., finitely generated) Cohen-Macaulay module over $R$ 
of Krull dimension $h$ such that  $H^h_P(\omega)$ is  an injective hull for $\kappa$. This statement 
characterizes $\om$ up to  noncanonical isomorphism. 

\item Hence, $\ext_R^i((M,\, H^h_P(\omega)\bigr) = 0$ for all $i \geq 1$, and 
$\Hom_R\bigl(\kappa,\, H^h_P(\omega)\big) \cong \kappa$. 

\item If $\fp$ is any prime ideal of $R$,  $\omega_{\fp}$ is a canonical module for $R_{\fp}$. 

\item If $\vect x k$ is part of a \sop\ for $R$, $\omega/(\vect x k)\omega$ is a canonical module for $R/(\vect x k)$. 

\item The homothety map $R \to \hom_R(\om,\,\om)$, determined by $r \mapsto (u \mapsto ru)$, is an isomorphism. 

\item (Local duality)  For every $i$, there is a natural isomorphism  of covariant functors $H^i_P(\blank)$ and 
$\hom_R\big((\ext_R^{h-i}(\blank, \omega),\, H^h_P(\om)\bigr )$, 
so that $H^i_P(M)$ is the Matlis dual of $\ext_R^{h-i}(M, \om)$.   In consequence, for all $i$, 
the modules $H^i_P(M)$ have DCC. 
 
\item If $M$ has finite length, i.e., if $M$ is killed by a power of $P$, then $\Ext_R^h(M,\, \om)$ has a natural 
identification with the Matlis dual of $M$, which by (1) above may be thought of as 
$\Hom_R\big(M,\, H^h_P(\om)\big)$. 
\item For all $i \in \N$,  $\di\bigl(\Ext_R^i(M,\, \om)\bigr) \le h-i$. 
\item If $M\neq 0$ is a finitely generated Cohen-Macaulay module of dimension $\delta$,  then $\ext^i_R(M, \om) = 0$ for all $i$
 except when $i = h-\delta$,
which is also the height of the annihilator of $M$. Moreover, $\ext_R^{h-\delta}(M,\, \om)$ is always nonzero.
In consequence, $\ext^{h - \delta}_R(\blank, \, \om)$ is exact on finitely generated Cohen-Macaulay $R$-modules of dimension $\delta$.
In particular, $\hom_R(\blank, \, \om)$ is an exact contravariant functor from finitely generated Cohen-Macaulay 
modules of dimension $h$ to finitely generated Cohen-Macaulay modules of dimension $h$. In addition,
since $\om$ is Cohen-Macaulay of dimension $h$,  $\ext^i_R(\om, \, \om) = 0$ for $i> 0$.  
\een

All are well known, but we give short proofs of (7), (8) and (9), assuming the preceding items (1)--(6).
To prove (7),  we note that if $M$ is killed by a power of $P$, then we have, using (6), that  
$M = H_P^0(M) \cong \hom_R\big(\ext_R^h(M,\, \om), \,H^h_P(\om)\big)$, and then each of $M$ and 
$\ext_R^h(M,\, \om)$ is the Matlis dual of the other \qed

For (8), let $N := \ext_R^i(M,\, \om)$. We may assume $0 \le i \le h$.  Observe that $N$ is not supported at
any prime $\fp$ with  $\height(\fp)<i$, since $N_{\fp}\cong \ext_{R_{\fp}}^i(M_{\fp},\, \om_{\fp}) = 0$. Thus,
$\di(N) \le h-i$. \qed

For (9), we also remark that for any finitely generated module $M$ over a Noetherian ring $R$ and 
any ideal $I$ of $R$, the modules
$H^i_I(M)$ all vanish if $IM = M$, while if $IM \not= M$ the first nonvanishing local cohomology module
occurs when $i = \depth_IM$, and all local cohomology modules vanish if $i$ exceeds either $\dim(M)$
or the number of generators of an ideal with the same radical as $I$.  It follows that when $M$ is a finitely
generated Cohen-Macaulay module over a local ring $(R,P)$, $H^i_P(M) = 0$ for all $i$ except when
$i = \di(M)$. By locality duality, we establish (9). \qed

Before concluding the discussion, we note also that these results imply 
that a canonical module $\omega$ has injective dimension precisely $h$,
that $R$ is a canonical module for $R$ if and only if $R$ is a Gorenstein local 
ring, and that $R$ has a canonical module if and only if it is a homomorphic 
image of a Gorenstein local ring $S$, in which case, if $R \cong S/I$, we may take 
$\omega = \Ext^{c}_S(R, \, S)$,  where  $c := {\di(S)-\di(R)} = \height I$.
\end{discussion}

\subsection{Modules free filterable with respect to a prime} \label{subsec:ff} We first define modules 
that are free filterable with respect to an ideal (especially, with respect to a prime ideal).

\begin{definition}[Compare with \dagref]\label{def:ff} 
Let $R$ be a ring, $\fP$ an ideal of $R$, and $M$ an $R$-module. We shall say that an $R$-module $M$ is {\it \fffP} if it has an ascending
filtration by a sequence of finitely generated $R$-submodules $\{M_n\}_{n \in \N}$ such that all of the factors 
$M_{n+1}/M_n$ are free over $R/\fP$ 
and such that the union $\bigcup_{n \in \N}M_n$ of the submodules is equal to $M$. 
\end{definition} 

In particular, we are interested in modules that are free filterable with respect to a prime ideal.  Next, we begin the development of several auxiliary results that will be needed in the proof of the main result,
Theorem~\ref{main}.

\begin{notation}\label{not1} From this point on in this section, let $R$ be a Noetherian ring,   
let $P \inc R$ be a prime ideal of $R$, let  $A := R/P$, and let $M$ denote an $R$-module that is not 
necessarily finitely  generated until Notation~\ref{not}, where the restriction that $M$ be finitely generated 
will be imposed.  
\end{notation}

The following fact, although obvious, is very useful if there happens to be a ``copy" of $A$ in $R$.

\begin{proposition}\label{ffAfree} Suppose that $A \inj R$  and that $P$ is a prime ideal of $R$ such that
the composite map $A \inj R \surj R/P$ is an isomorphism.  Then an $R$-module $M$ is \ff
if and only if it is free of finite or countably infinite rank. 
\end{proposition}

\begin{proposition}\label{ff}  Let $R$, $M$, $P$, and  $A$ be as in Notation~\ref{not1}. 
\bena
\item If $M$ is \ff, then $\Ass_R(M) \subseteq \{P\}$, and every finitely generated submodule of $M$ is killed by 
a power of $P$.
\item If $M$ is any finitely generated module killed by a power of $P$, there exists $g \in R\sm P$ such
that $M_g$ is free filterable relative to $PR_g$ over $R_g$. 
\item If $M$ is \ff and $N$ is a finitely generated submodule, then there exists  $g \in R\sm P$ such
that $(M/N)_g$ and $N_g$ are free filterable relative to $PR_g$ over $R_g$.  
\een
\end{proposition}
\begin{proof} Part~(a) follows from the fact that every finitely generated submodule $N$ of $M$ is contained
in a submodule $N'$ with a finite filtration by finitely many factors that are direct sums of $A$, with
the additional property that $M/N'$ is \ff.   
To prove part~(b), take a finite prime cyclic filtration of $M$ and localize at $g$ that is not in $P$ but is in all other
primes $Q_i$ such that $R/Q_i$ occurs in the filtration.  
For part~(c),  choose $N'$ as in the proof of part~(a) and then, by part~(b), there exists $g \notin P$ such that 
$(N'/N)_g$ and $N_g$ become free filterable with respect to $PR_g$ over $R_g$. It follows that $(M/N)_g$ is 
free filterable relative to $PR_g$ over $R_g$. 
\end{proof}

\begin{remark}\label{rem:fff} Clearly, if $M$ has an ascending filtration by a sequence of finitely generated modules
such that each factor is \ff, then we may refine that filtration to get one in which the factors are all
$A$-free (simply inserting finitely many modules between each pair of consecutive terms of the original filtration to make
all factors $A$-free).  Another useful fact is recorded in Lemma~\ref{lem:ff}.  We need a preliminary discussion. \end{remark}

\begin{discussion}[\textbf{Double filtrations and compatible submodules}] \label{disc:compat}
Suppose that we are given an ascending sequence filtration (these may be finite or infinite) of a module $M$, say 
$$0 = M_0 \inc M_1 \inc \cdots \inc M_n \inc \cdots \inc M,$$ where $\bigcup_{n \in \N} M_n = M$,  and 
for each $n \in \N$ a filtration 
$$M_n = M_{n,0} \inc M_{n,1} \inc \cdots \inc M_{n,s} \inc  \cdots \inc M_{n+1}$$  
such that $\bigcup_s M_{n,s} = M_{n+1}$ and every $M_{n,s+1}/M_{n,s}$ is finitely generated. This, equivalently, 
gives a filtration of each factor $M_{n+1}/M_n$ using the modules $M_{n,s}/M_n$.  We refer
to the $M_{n,s}$ as a {\it double filtration} of $M$. We refer to the modules $M_{n+1}/M_n$ as the {\it factors}
of the double  filtration and the modules $M_{n,s+1}/M_{n,s}$ as the {\it double factors} of the double filtration. 

We say that a submodule module $N \inc M$ is {\it compatible} with this filtration if 
for all $n$, the submodule $(M_{n+1} \cap N) + M_n$ (which is the same as $M_{n+1} \cap (N + M_n)$) 
equals $M_{n,s(n,N)}$ for some $s(n,N) \in \N \cup\{\infty\}$ with the understanding that $M_{n,\infty} = M_{n+1}$.

\smallskip\noi (1) Given a compatible submodule $N$ of $M$, we have an induced filtration 
$\{M_n \cap N\}_n$ of $N$, with the factors $\disp \frac {M_{n+1}\cap N} {M_n\cap N}$. 

\noi (2) Denote $N_n:=M_n \cap N$. We have a double filtration of $N$ as follows. By the compatibility hypothesis
on $N$, we have that $(M_{n+1} \cap N) + M_n = M_{n,s(n,N)}$ for some $s(n,N) \in \N\cup\{\infty\}$. 
Thus, the factors of $N$ are
$$ \frac{N_{n+1}}{N_n} \cong \frac{M_{n+1} \cap N} {M_n \cap N} 
\cong \frac{(M_{n+1} \cap N) + M_n}{M_n} = \frac{M_{n+1} \cap (N + M_n)}{M_n} = \frac{M_{n, s(n, N)}}{M_n}.$$ 
Now we filter $N_{n+1}/N_n$ using the images of the modules $N_{n,s} := M_{n,s}\cap N$. 
It is straightforward to see that 
\[
\frac{N_{n,s+1}}{N_{n,s}} = \frac{M_{n,s+1}\cap N}{M_{n,s}\cap N} 
\cong \frac{(M_{n,s+1}\cap N) + M_{n,s}}{M_{n,s}} 
=\begin{cases}
{M_{n,s+1}}/{M_{n,s}} & \text{if } s < s(n,N)\\
0 & \text{if } s \geq s(n,N)
\end{cases}.
\]
This $\{N_{n,s}\}_s$ is the induced double filtration of $N$.
 
\noi(3) The module $M/N$ has the induced filtration $\{(M_n+N)/N\}_n$ with factors 
$$
\frac{(M_{n+1}+N)/N}{(M_n+N)/N} \cong \frac{M_{n+1}+N}{M_n+N} 
\cong \frac{M_{n+1}}{M_{n+1} \cap (N + M_n)} \cong
\frac {M_{n+1}}{M_{n,s(n,N)}}.
$$
The induced double filtration, by $\{(M_{n,s}+N)/N\}_s$, has double factors
\[
\frac{(M_{n,s+1}+N)/N}{(M_{n,s}+N)/N} 
\cong \frac{M_{n,s+1}}{M_{n,s+1} \cap (N + M_{n,s})} \cong
\begin{cases}
0 & \text{if } s < s(n,N)\\
{M_{n,s+1}}/{M_{n,s}} & \text{if } s \geq s(n,N)
\end{cases}.
\]

\noi (4) Suppose that $Q$ is another compatible submodule of $M$ such that $N \inc Q$.  
We claim that $N$ is a compatible submodule of $Q$.  The induced filtration on $Q$ from part~(1) has
submodules $Q_n = M_n \cap Q$.   The induced filtration from these on $N$ is the same as the one
induced by the original filtration on $M$,  since $(M_n \cap Q) \cap N = M_n \cap N$.  
It is also true that the double filtrations on $N$ induced by $M$ and by $Q$,
as described in part~(2), are the same, since the double filtration that $Q$ induces on $N$ 
consists of the modules $Q_{n,s} \cap N = (M_{n,s} \cap Q) \cap N = M_{n,s} \cap N$. 
  
\noi (5) Hence, if $N \inc Q \inc M$ where $N$ and $Q$ are both compatible with double filtration on
 $M$,  then $Q/N$ has a filtration whose double factors are all double factors of the filtration of $M$.
 
\noi (6) We also observe the following:  if $G$ is any finitely generated submodule of the doubly filtered
module $M$, then there is a finitely generated compatible submodule $N$ of $M$ such that $G \inc N$.  It follows that
$M$ is the directed union of its finitely generated compatible submodules.  To find $N$, 
note that there exists $n \in \N$ such that $G \inc M_{n+1}$. Then there exists $s \in \N$ such that 
$G + M_n \subseteq M_{n,s}$. We shall prove by induction on $n$ that we can find a finitely generated submodule $N$ 
with $G \inc N \inc M_{n+1}$ such that $N$ is a compatible with the double filtration on $M$. 
(Note that the modules $M_t$ in the filtration of $M$ with $t > n+1$ will play no further role in the argument.)
If $n =0$, then $N:= M_{0,s}$ can be used as the choice of $N$. Assume $n \geq 1$. As $M_{n,s}/M_n$ and $G$
are both finitely generated, there exists a finitely generated $R$-submodule $H$ such that $G \inc H \inc M_{n,s}$ and 
$H + M_n = M_{n,s}$. Clearly, $M_n \cap H$ is finitely generated and contained in $M_n$. By the induction hypothesis,
we can choose a finitely generated submodule $L$ with $M_n \cap H \inc L \inc M_n$ such that $L$ is 
compatible with the double filtration on $M$.  Let $N:= L + H$, so that $N \inc M_{n,s} \inc M_{n+1}$ and $N$ is finitely generated. 
As $L \inc M_n$, we have $M_{n+1} \cap (N + M_n) = M_{n+1} \cap (H + M_n) = M_{n,s}$. 
For all $t \ge n+1$, it is clear that $M_{t+1} \cap (N + M_t) = M_t = M_{t,0}$.
Also note that $M_n \cap N = M_n \cap (L + H) = L + (M_n \cap H) = L$, 
since $M_n \cap H \inc L \inc M_n$ by the choice of $L$. Thus, for all $t = 0,\,\dotsc,\, n-1$, we have 
$M_{t+1} \cap N = M_{t+1} \cap L$ and, hence, 
\[
(M_{t+1} \cap N) + M_{t} = (M_{t+1} \cap L) + M_{t} = M_{t,s(t,L)} \text{ for some $s(t,L) \in \N \cup\{\infty\}$},
\] 
since $L$ is a compatible submodule of $M$.
This proves that $N$ is compatible with the double filtration on $M$. 
\end{discussion}

\begin{lemma}\label{lem:ff} If $M$ has a filtration by an ascending sequence of submodules whose union is $M$ and
whose factors are \ff, then $M$ is \ff. 
\end{lemma}

\begin{proof}  The hypothesis says that $M$ has a filtration by submodules $M_n$ such that
every $M_{n+1}/M_n$ is \ff. This gives rise to a double filtration as in Discussion~\ref{disc:compat} such that
every $M_{n,s+1}/M_{n,s}$ is finitely generated and free over $A$. 
Choose a finite set of elements $S_{n, t} \inc M_{n,t+1}$ whose images in $M_{n,t+1}/M_{n,t}$
generate the double factor $M_{n,t+1}/M_{n,t}$.  It is easy to see that the union of the sets $S_{n,t}$, 
namely $\bigcup_{n \ge 0,\,t \ge 0}S_{n,t}$, generates $M$.  Take $N_0 := 0$ and recursively 
construct a sequence of finitely generated compatible submodules $N_n$ of $M$ such that
for each $n$,  $N_{n+1}$ contains $N_n$ and 
$\bigcup_{i \le n,\,t \le n}S_{i,t}$, which is possible by part~(6) of Discussion~\ref{disc:compat}.  
By part~(5) of Discussion~\ref{disc:compat}, the modules $N_{n+1}/N_n$ have finite filtrations 
with factors that are finitely generated and $A$-free, by Remark~\ref{rem:fff}.  
\end{proof}

\subsection{Graded generic freeness}  In the sequel, we need a slightly strengthened version of 
generic freeness \cite[\S6.9]{EGAIV65} in graded situations, with a quite short proof.

\begin{lemma}[Graded generic freeness]\label{grgenfree}  Let $R = \bigoplus_{i \in \N}[R]_i$ be 
an $\N$-graded ring that is finitely generated over a Noetherian domain $A \to [R]_0$,  and let
$M$ be a finitely generated $\Z$-graded $R$-module. Then there exists $g \in A \smo$  such that
$[M_g]_d$ is free over $A_g$ for all $d$. When $A \to [R]_0$ is module-finite, the modules
$[M_g]_d$ are free of finite rank over $A_g$ for all $d$. \end{lemma}

\begin{proof} Since $M$ has a finite filtration by graded cyclic modules, 
we can reduce at once to the
case where $M  = R/I$,  where $I$ is homogeneous, and we change notation and write $R$ for 
$R/I$. That is, we only need to do the case where the module is a finitely generated graded $A$-algebra.  We use
induction on the number $n$ of homogeneous generators of $R$ over $A$ (the case where $n =0$ is
obvious).  For $n \ge 1$, we may write $R = B[u]$ where $B$ needs strictly fewer than $n$ homogeneous generators over $A$
and $u$ is homogeneous.  We have a filtration of $R$ by $B$-submodules $C_t := B + Bu + \cdots + Bu^t$, and
a typical factor is $D_t := C_t/C_{t-1} \cong B/J_t$,  where $J_t := \{b \in B: bu^t \in C_{t-1}\}$.  The ideals
$J_t$ form an ascending sequence, and so stabilize.  The filtration therefore has only finitely many
distinct factors, each of which is an $\N$-graded quotient of $B$,  and so has strictly fewer than $n$
homogeneous generators over $A$.  Thus, we can localize at one element of $A\smo$ so that every graded component 
of every factor $D_t$ is $A$-free.  It follows that each $[R]_d$ has a countable ascending  filtration by $A$-modules 
such that every factor is $A$-free, which implies that $[R]_d$ is $A$-free. 
\end{proof}

Note that Lemma~\ref{grgenfree} does not follow from the assertion that one can localize at 
$g \in A \setminus \{0\}$ so that $M_g$ becomes $A_g$-free: that
only implies that the $[M_g]_d$ are projective modules over $A_g$.  
In Lemma~\ref{grgenfree}, if $A \to [R]_0$ is not injective the result is true, since localizing at one nonzero element 
of its kernel kills $R$ and $M$. One does not always know whether the map one has is injective.  
\medskip

The following result will be needed in \S\ref{sec-basechange}. It is close to material in 
\cite[Ch.~8, \S20]{Mat80}, 
but we give a very brief argument adapted to our situation.

\begin{proposition}\label{B-flat} Let $R$ be a Noetherian algebra over a Noetherian ring $B$ and suppose 
that $I$ is an ideal of $R$.  Suppose that $M$ is a  finitely generated $R$-module such that $\gr_I(M)$
is $B$-flat. Then $M/I^tM$ is $B$-flat for all $t \in \N$. Suppose also that either
\benn
\item for every prime ideal $\fq$ of $B$,  $\fq \otimes_B M$ is $I$-adically separated, or 
\item $I$ is in the Jacobson radical of $R$, or 
\item $R$ is $I$-adically complete.  
\een
Then $M$ is $B$-flat.  
\end{proposition}

\begin{proof} The first statement follows because $M/I^tM$ has a finite filtration in which the factors 
$I^jM/I^{j+1}M$ are all $B$-flat. For the rest, note that $(3) \imp (2) \imp (1)$, and it suffices
to show that (1) $\imp$ $\Tor^B_1(B/\fq,\, M) = 0$
for all $\fq \in \Spec(B)$, i.e., all the natural maps $\theta: \fq \otimes_B M \to M$ are injective.
But if $u \in \Ker(\theta)$,  then for all
$t \in \N$ the image of $u$ is in $\Ker(\theta_t)$ with $\theta_t:\fq \otimes_A (M/I^tM) \to M/I^tM$. 
Since every $M/I^tM$ is $B$-flat, $\theta_t$ is injective, and so 
$u \in \bigcap_{t \in \N} I^t(\fq \otimes_B M) = 0$, by (1).
 \end{proof}

\subsection{The main result on generic local duality}
For a general reference on local cohomology, see \cite{Gro67}.

\begin{remark} We shall be using freely throughout that Ext and, in particular, Hom, commutes with
flat base change when the ring is Noetherian and the first (left) input module is finitely
generated.  This will most frequently be applied in the case of localization, but will also be needed to
pass to the $P$-adic completion of $R$. \end{remark}

\begin{remark}\label{natisom} 
We recall that, for \emph{any} three $R$-modules $G,\, \om, \,\cC$, we have the following natural maps
\ben
\item $\phi :\hom_R(G, \, \om )\otimes_R\cC \xra{\quad} \hom_R(G, \,\om \otimes_R \cC)$.
\item $\psi :G \otimes_R\hom_R(\om,\,\cC) \xra{\quad} \hom_R\big(\hom_R(G, \,\om),\,\cC\big)$.
\een
Both commute with finite direct sums of choices of $G$. The map $\phi$ is induced by the bilinear map 
$(h, c) \mapsto (u \mapsto h(u)\otimes c)$ and the map $\psi$ is induced by the bilinear map 
$(u, f) \mapsto (h \mapsto f(h(u)))$. Both $\phi$ and $\psi$ provide very useful natural isomorphisms
when $G = R$ and, hence, when $G$ is any finitely generated projective module. Moreover, 
when $G$ is finitely presented, the map $\phi$ is an isomorphism when $\cC$ is flat and the map $\psi$ is
an isomorphism when $\cC$ is injective. See \cite[Prop.~1.4.6(c), p.~39,  Prop.~1.4.9(b), p.~41]{CFH24}.
\end{remark}  

\begin{remark}\label{open}  We shall also be using freely throughout that the regular locus, the 
Cohen-Macaulay locus, and the Gorenstein locus in an excellent ring are Zariski open sets.  
We refer the reader to \cite[\S\S6.11--6.13, \S\S7.6--7.8]{EGAIV65} for a detailed treatment of 
openness of loci.   There is also a particularly readable discussion of some of this 
material in \cite[Ch.~13]{Mat80}. Note that when $R$ is excellent, 
if $\om$ is a finitely generated $R$-module and $P$ a
prime such that $\om_P$ is a canonical module for $R_P$, there exists $g \in R \sm P$ such that
$R_g$ is \CM and $\om_g$ is a global canonical module.  To see this, first note that since $\om_P$ is
faithful over $R_P$, the same holds for $\om$ and $R$ after localizing at one element of $R\sm P$. 
Next, consider the Nagata idealization $S:= R \oplus \om$ (see, for example,  the proof 
of \cite[Thm.~(3.36)]{BH93}, where $S$ is denoted $R*M$) of $\om$, where $\om^2 = 0$.  
Note that $S$ is \fg as an $R$-module, and the algebra maps $R \inj R \oplus \om \surj R$,
where the second map is the quotient map that kills the ideal $\om$, induce isomorphisms of the spectra. 
In fact, the composite map is the identity on $\Spec(R)$. Under this identification of spectra,
 $P$ corresponds to $P\oplus\om$. Then 
 $(R \oplus \om)_{P\oplus \om} \cong R_P \oplus \om_P = S_P$,
 and since $\om_P$ is a canonical module for $R_P$,  we have that 
 $S_P$ is Gorenstein, as in \cite[Thm.~(3.36)]{BH93}.
 Since $S$ is also excellent, we can choose $g \in R\sm P$ such that 
 $S_g \cong R_g \oplus\om_g$ is Gorenstein. Thus, if $Q \in D(g)$, $S_Q \cong R_Q \oplus \om_Q$ 
 is Gorenstein, and $\Hom_{S_Q}(S_Q/\om_Q, S_Q) \cong
\Ann_{S_Q} \om_Q = \om_Q$ (since $\Ann_{R_Q} \om_Q = 0$) is a canonical module for 
$S_Q/\om_Q = R_Q$, as claimed. 
\end{remark}

\begin{definition} Let $W$ be a subset of $R$.  We shall say that a statement about
$R$-modules, ideals of $R$, and/or maps of $R$-modules holds \emph{\wg} if there exists an element $g \in W$
such that the statement  holds after we make a base change from $R$ to $R_g$. \end{definition}

\begin{remark}\label{inv} 
Assume that $I$ is any ideal of $R$, $M$ is an $R$-module in which every element is killed by a power
of $I$, and $g,\, g' \in R$ satisfy $g' \equiv g \mod I$. Then $M_{g} \cong M_{g'}$. That is,
inverting one element with a given residue mod $I$ inverts every element with the same residue.
To see this, note that $g' = g - (g-g') = g[1 -(g-g')/g] \in R_g$ and $(g-g')/g$ acts nilpotently on every 
cyclic submodule of $M_g$, which shows that $g'$ acts bijectively on every cyclic submodule of $M_g$. 
Consequently, $g'$ acts bijectively on $M_g$. By symmetry, $g$ acts bijectively on $M_{g'}$.
Thus $M_{g} \cong M_{g'}$.\footnote{Note also that, if $g' \equiv g \mod I$, then $g'$ is invertible in the $IR_g$-adic completion 
$S$ of $R_g$, and so acts invertibly on every $S$-module:  in this case, $g-g'$ is in the Jacobson radical of $S$.}

Suppose that $H(\blank)$ is a functor that commutes with localization and whose values are modules 
in which every element is killed by a power of $I$, for example $H^i_I(\blank)$, $\Tor_i^R(V,\,H^j_I(\blank))$,
$\Ext^i_R(L,\,H^j_I(\blank))$ or $\Ext^i_R(N,\,\blank)$, 
where $V$ is any $R$-module, $L$ and $N$ are finitely generated modules over a Noetherian
ring $R$ such that $N$ killed by a power of $I$. Suppose that $Q$ is any module, {\it not necessarily a module such that
every element is killed by a power of} $I$. Then we again have $H(Q_g) \cong H(Q_{g'})$
for all $g,\, g' \in R$ such that $g' \equiv g \mod I$, since we may apply the paragraph above to $M = H(Q)$.  
We shall frequently make use of these observations without comment. 
\end{remark}

\begin{remark}\label{avoid} Given any ideal $I$ of $R$ and any $g \in R$, along with finitely many prime ideals $\vect P n \in \Spec(R) \sm V(I)$, there exists $g' \in R$ such that $g' \equiv g$ mod $I$ and  $g' \notin \bigcup_{i=1}^n P_i$. 
This follows at once from the form of prime avoidance given in \cite[Thm.~124, p.~90]{Kap74}.
\end{remark}

\begin{notation}\label{not} Let $R$ be a Noetherian ring, $P$ be a prime ideal of $R$, $A := R/P$, and
$W \inc R \sm P$ a subset that maps onto $A \sm \{0_A\}$ under the natural map $R \to R/P$. 
Let $M$ be  a finitely generated $R$-module.  Assume also that $R_P$ is Cohen-Macaulay of Krull dimension $h$ 
and that $\om$ is a finitely generated $R$-module  such that $\om_P$ is a canonical module for $R_P$. 
Let $E:=H^h_P(\om)$. For every prime $\fp$ of $R$, let $\kappa_\fp$ denote $R_\fp/\fp R_\fp$ 
which is naturally isomorphic with $\fra(R/\fp)$.
\end{notation}

In the sequel, given an $R$-module $M$, we write $\E_R(M)$ for the injective hull of $M$ over $R$, 
which is unique up to non-unique isomorphism.

\begin{remark}\label{rmk:not} With Notation~\ref{not} above, note that the
map  $R \to \Hom_R(\om,\,\om)$ via the
homothety map $r \mapsto (u \mapsto ru)$ becomes an isomorphism when we localize at one 
$g \in R\setminus P$, since that is true once we localize at $P$, and the modules are finitely generated.
Moreover, since $\om_P$ is a canonical module
for the Cohen-Macaulay ring $R_P$,  $\E_{R_P}(R_P/PR_P)$, which is canonically
isomorphic with $\E_R(R/P)$,
is isomorphic with $E_P$. The next theorem, one of the main results of this paper, asserts that $E$ 
behaves very much like $\E_R(R/P)$ on a Zariski neighborhood of $P$.  The size of the neighborhood is
typically adjusted, depending on a specified set of finitely many finitely generated $R$-modules.  
We want to emphasize parts~(e) and (f) of Theorem~\ref{main}. 
\end{remark}

\begin{theorem}[Main theorem on generic local duality]\label{main} Let notation be as in \ref{not}. 
\bena
\item  For every  $i \geq 1$, there exists $g \in W$, where $g$ depends
on $M$ and $i$, such that $\Ext^i_R\big(M, E\big)_g = 0$.  That is, for every $i \geq 1$, the module
$\Ext^i_R\big(M, H^h_P(\om)\big)$ is \wg $0$. Hence, if $0 \to M' \to M \to M'' \to 0$ is a sequence of 
finitely generated $R$-modules that becomes exact after localization at $P$, then the induced sequence
$$0 \xra{\quad} \Hom_R(M'',\, E) \xra{\quad} \Hom_R(M,\, E) \xra{\quad} \Hom_R(M',\, E) \xra{\quad} 0$$ 
is \wg exact.  Therefore, the functor $\hom_R(\blank, \, E)$ is \wg exact on 
any given finite set of short exact sequences of finitely generated $R$-modules, and, hence, on
any given finite set of finite long exact sequences of finitely generated $R$-modules. The choice of $g \in W$, at which we localize, 
depends on which finite set of finite exact sequences one chooses.  

\item 
If $M_P$ is killed by a power of $P$ then we have, $W$-generically,  an isomorphism 
$$\Hom_R(M, \, E) \cong \ext^h_R(M,\,\omega).$$
Consequently, \wg,  $\hom_R(M, \, E)$  is a finitely generated $R$-module. 

Moreover, the unique largest $A$-module contained in $E$, i.e., $\Ann_E P \cong \Hom_R(R/P, \, E)$, is \wg  
isomorphic with $A = R/P$. That is, \wg, we have $R/P \cong \Ann_E P$.  Thus, after localization at $R\setminus P$,
we obtain a map $\kappa_P \mapsto E_P$ such that the image of $\kappa_P$ is the socle
in $E_P \cong \E_{R}(\kappa_P)$.

\item \wg, $H^h_P(R) \cong \hom_R(\om, E)$.   

\item Let $N := \Hom_R(M, E)$, where $M$ is a finitely generated $R$-module.  Then  after 
localizing at one  element of $W$ \emph{that is independent of $n \in \N$}, the filtration of $N$ by the modules 
$\Ann_N P^n$,  which ascends as $n$ increases with $\bigcup_{n \in \N}\Ann_N P^n = N$, has 
factors that are finitely generated free modules over $A$.  In particular, if $\fa$ is any ideal of $R$,
the conclusion holds for $N = \Ann_E\fa \cong \hom_R(R/\fa, \, E)$. 

\item Assuming the isomorphism in part (c), we have a natural transformation of functors 
from the category of finitely generated $R$-modules to the category of $R$-modules, namely
$$\hom_R\big(\Ext_R^{h-i}(\blank, \om), \, H^h_P(\om)\big) \arwf{\qquad} H^i_P(\blank),$$  
such that for every finitely generated $R$-module $M$  there exists $g \in W$ such that for all $i \geq 0$,
\[
(*)\phantomsection\label{main*} 
\qquad\qquad \hom_R\big(\Ext_R^{h-i}(M, \om),\, H^h_P(\om)\big)_g \arwf{\ \ \cong\ \ } H^i_P(M)_g \qquad\qquad
\]
is an isomorphism.  Note that $g$ depends on $M$ but \emph{not} on $i \in \N$.  
If $R$ (or some localization of $R$ at an element of $R \setminus P$) is excellent and has finite Krull dimension, 
then we can choose $g \in W$ that depends on $M$ but \emph{not} on $i \in \Z$ such that \hyperref[main*]{$(*)$} 
is an isomorphism for all $i \in \Z$.

\item Hence, if $M$ is any finitely generated $R$-module, after localizing at one element $g \in W$, all of 
the local cohomology modules $H^i_P(M)$ are \ff.  That is, the local cohomology modules  
$H^i_P(M)$ are \wg \ff. More precisely,  after localization at one element of $W$,  the $R$-modules
$\Ann_{H^i_P(M)}P^n$, $n \in \N$, give an ascending filtration of $H^i_P(M)$ such that all of the factors are 
$A$-free of finite rank with $\bigcup_{n \in \N}\Ann_{H^i_P(M)} P^n = H^i_P(M)$.
\een
\end{theorem}

\begin{proof}
By Remark~\ref{inv}, it suffices to assume $W = R \sm P$. 
In the course of the proof we may repeatedly, but finitely many times,  localize at one element  
$g \in W$. Each time, we make a change of terminology, and continue to use $R$ to denote
the resulting ring  $R_g$.  Likewise, we  use $P$ for its extension to $R_g$, and for every module under
consideration we use the same letter for the module after base change from $R$ to $R_g$ (finitely
generated modules under consideration  are replaced by their localizations at $g$: this is the same as base change from $R$ to $R_g$).

There exist $\ux:=\vect x h \in R$ whose images are a \sop\ for $R_P$, where they form a regular
sequence on $R_P$ and $\om_P$. After inverting an element of $W = R \sm P$ (and using $R$ to denote the resulting ring 
$R_g$, as agreed above), 
we may assume that $\sqrt{(\ux)} = P$ in $R$. Therefore, we have $E = H^h_P(\om) = H^h_{(\ux)}(\om)$.
For the same reason, we may assume throughout that $\vect x h$ is a regular sequence on $R$ and on $\om$ 
such that $\sqrt{(\ux)} = P$ in $R$,  and that $E = H^h_P(\om) = H^h_{(\ux)}(\om)$.

The proofs of (a), (b), and (c) depend on some results about the spectral sequences of
a double complex, which are collected for reference in ($\ddagger$) below. \medskip

{ {\leftskip 13pt 
\noindent {\it Discussion} ($\ddagger$). {\bf The spectral sequences of a double complex.} \phantomsection\label{ddref}
We only need the four standard facts listed below from the theory of spectral sequences for a double complex with 
bounded diagonals that relate the iterated cohomology (or second page, $\page_2$) of the double complex with associated 
graded modules of the cohomology $\cH^{\bullet}$ of the total complex. Let $k,\, v \in \Z$ be fixed integers.
\ben
\item The spectral sequence calculation commutes with flat base change, including localization. 

\item If the terms $\bE^{ij}_2=0$ for all $(i,j) \in \Z^2$ such that $i+j = k$, then $\cH^k = 0$. 

\item If the terms $\bE^{ij}_2 = 0$ for all $(i,j) \in \Z^2$ such that $j \not= \nu$ (respectively, for all $(i,j) \in \Z^2$ such that $i \not = \nu$),
then $\cH^{i+\nu} \cong \bE_2^{i,\nu}$ for all $i$ (respectively, $\cH^{\nu + j} \cong \bE_2^{\nu, j}$ for all $j$). 

\item If $\bE^{ij}_2 = 0$ for all $(i,j) \in \Z^2$ such that $k-1 \leq i+j \leq k+1$  except when $i =\nu$ or $j = k - \nu$,  
then $\cH^k \cong \bE^{\nu,\, k-\nu}_2$.
\footnote{We comment only on (4).
We have for $r \geq 2$ that  all of the spots  on the diagonal 
where $i+j = k$ are stable as the page index $r$ increases:  at spots where $i \not = \nu$, this holds
because $\bE^{i, k - i}_r = 0$ .   At the $(\nu, k-\nu)$ spot this holds because for $r \geq 2$ the 
graded component of  the differential $d^r$ mapping to (resp., from) $\bE^{\nu,\, k-\nu}_r$  
has domain (resp., target) on the diagonal of degree $k-1$ (resp., degree $k+1$) at a spot
that is {\it not} in the row or column of $(\nu, k-\nu)$, and so is 0.
Hence, there is only one possibly nonzero factor, namely $\bE^{\nu, k-\nu}_2$, in the filtration of $\cH^k$ 
one gets from the $\page_{\infty}$ page.} 
\een} }

In the proofs for (a), (b) and (c) below, using one spectral sequence we show that, \wg, the cohomology 
of the total complex  in degree  $i+h$ is $\ext_R^i(M, \, E)$.
Using the other spectral sequence we prove, in each of cases (a), (b) and (c), that $W$-generically 
there are zeros at a certain  spots of the second page, enabling us to reach the conclusion we need from 
one of the facts just above.

Let $\cC^{\bu}$ denote the usual modified \Cech complex used to calculate 
$H^\bu_P(R) \cong H^\bu_{(\ux)}(R)$, namely 
$$
0 \xra{\quad} R \xra{\quad} \cdots \xra{\quad} \bigoplus_{1 \le i_1 < i_2 < \cdots < i_j \le h} R_{x_{i_1} \cdots x_{i_j}} 
\xra{\quad} \cdots \xra{\quad} R_{x_1\cdots x_h} \xra{\quad} 0,
$$ 
where the typical term $\cC^j$ is flat over $R$.  Let $G_{\bu}$ be a free  resolution 
of $M$ by free  $R$-modules of finite rank. 
Consider the isomorphic (see Remark~\ref{natisom}) double complexes 
$$\hom_R(G_\bu, \om \otimes_R \cC^\bu) \cong \hom_R(G_\bu, \om) \otimes_R \cC^\bu.$$
We compute the cohomology of the total complex by using the two spectral sequences for iterated 
cohomology obtained from this double complex.  The typical module in the double complex is
$\hom_R(G_i, \om \otimes_R \cC^j) \cong \hom_R(G_i,\, \om) \otimes \cC^j$.  
If we fix $i$ and let $j$ vary, then because $\ux$ is a regular sequence on 
any finite direct sum of copies of $\om$, the cohomology vanishes except when $j = h$.  This shows that 
the spectral sequence degenerates, and that, \wg,  the cohomology of the  total complex in degree $h+i$ is  
$\cH^{h+i} \cong \ext_R^i\big(M,\, H^h_P(\om)\big)  = \ext_R^i(M, \, E)$. 
If we fix $j$ and let $i$ vary (noting that $\cC^j$ is flat), then we first get cohomology  
$\ext^i_R(M,\, \om) \otimes \cC^\bu$, 
while the iterated cohomology is $H^j_P\big(\Ext_R^i(M,\, \om)\bigr)$ on the second page. We denote 
pages of this second spectral sequence by $\page_r$, $r \geq 2$ throughout the rest of the proof
of parts (a), (b) and (c). In particular, explicitly, $\bE_2^{ij} \cong H^j_P\big(\Ext_R^i(M,\, \om)\bigr)$. 
Quite generally, this spectral sequence is $0$ for all the spots 
with  $j< 0$ or $j > h$ from the second page onward. We now prove the statements (a), (b) and (c), by 
proving that, in each case, certain of the modules $H^j_P\big(\Ext_R^i(M,\, \om)\bigr)$ vanish generically 
for a relevant set of pairs $(i,\,j)$.

(a) Replacing $M$ by a suitable module of syzygies,\footnote{This is not necessary, but makes the argument
a bit easier to follow.} we see that it suffices to prove the case $i = 1$.
By Discussion~\hyperref[ddref]{$(\ddagger)$(2)}, to show that $\Ext_R^1(M,\, E)$ vanishes after 
localization at one element  $g \in W$, it suffices to prove that there exists $g \in W$ such that
$H^j_P\big(\ext_R^i(M, \om)\bigr)_g = 0$ when 
$i + j = h+1$ and $0 \leq j \leq h$. Thus, we want to show that 
$H^j_P\big(\ext_R^{h+1-j}(M, \om)\bigr)$ vanishes \wg for $0 \leq j \leq h$,
which proves that the spectral sequence stabilizes at $0$ on 
those spots from the second page onward. When $j = 0$, this holds because 
$\injdim_{R_P}(\om_P) = \di(R_P) = h$ and so $\ext_R^{h+1-0}(M,\, \om)$ becomes $0$ after localizing 
at one  element of $W = R \setminus P$.  
Assume $1 \leq j \leq h$. Let $I = \Ann_R \big(\ext_R^{h+1-j}(M, \om)\big)$.  By localizing at
one element of $W$, we may assume that $I$ is contained in $P$. Next note that $\height(P/I) \le  j-1$, 
because this statement is unaffected by localization at $P$, and, from Discussion~\ref{recap}(8), we see that
$\height(P_P/I_P) = \di(R_P/IR_P)  = \di\bigl(\ext_{R_P}^{h+1-j}(M_P, \,\om_P) \bigr) \le j-1$.
Thus,  $PR_P/IR_P$ has a  system of parameters consisting of  images of elements of $R$ 
with at most $j-1$ elements. Consequently, after  localizing at one element of $W$, we have that $P/I$ is 
the radical of an ideal $(\uz) = (\vect z {j-1})$, with at most $j-1$ generators. But then, for 
modules killed by $I$,  $H^j_P(\blank) \cong  H^j_{P/I}(\blank) \cong H^j_{\uz}(\blank) \cong 0$. 

Thus, the $\page_2$ page degenerates to $0$ for all $(i,\,j)$ spots with $i+j = h+1$. 
It follows at once from \hyperref[ddref]{$(\ddagger)$(2)} that 
$\Ext^1_R\big(M, E\big)_g = 0$ as required and, hence, that for fixed $i \geq 1$,
 $\Ext^i_R\big(M, E\big)_g = 0$ for suitable $g$.
Now it is straightforward to prove the remaining claims on the $W$-generic 
exactness of  $\hom_R(\blank,\,E)$ on short exact sequences of finitely generated $R$-modules. 
This concludes the proof of (a).

(b) First note that if $M_P=0$, we may localize at one element of $W= R \setminus P$ and assume
that $M =0$, in which case the result is obvious. Otherwise, after we localize at one element of $W$, we may
assume that the radical of the annihilator of $M$ is $P$. Hence $H^0_P\bigl(\Ext_R^{i}(M,\, \om)\bigr) 
\cong \Ext_R^{i}(M,\, \om)$ and $H^j_P\bigl(\Ext_R^{i}(M,\, \om)\bigr) = 0$ for all $i \in \N$ and all $j \not=0$, 
since a power of $P$ kills $\Ext_R^{i}(M,\, \om)$.

In light of this, we see that  $\page_2$ degenerates:  all the nonzero terms can only occur for 
$j=0$ and are simply the modules $\Ext^i_R(M, \, \om)$. By \hyperref[ddref]{$(\ddagger)$(3)}, 
the module $\Ext^i_R(M, \, \om)$ 
is the cohomology $\cH^i$ of the total complex in degree $i$. In particular, when $i =h$,  we have that, \wg, 
$\Ext^h_R(M, \, \om) \cong \cH^h \cong \Hom_R(M,\,E)$. 
 
Next, let $N:=\hom_R(R/P, \, E) \cong \Ann_EP$. By the first statement in part~(b), proved above, 
we know that, $W$-generically, $N \cong \ext_R^h(R/P, \, \om)$ is 
finitely generated.  After localization at $P$, we have $N_P \cong \kappa_P =
\fra{A}$.  We  may choose an element $u \in \hom_R(R/P, \, E)$ that becomes a generator this module
once we localize at $P$. Hence, the linear map $\theta:A \to  \hom_R(R/P, \, E)$ such that $1 \mapsto u$ becomes
an isomorphism from $\kappa_P$ to $\kappa_P$ when we localize at $P$. Since the kernel and cokernel of
$\theta$ are finitely generated $R$-modules, they both vanish after localization at some $g \in W$, which 
then makes $\theta$ an isomorphism. This concludes the proof of part~(b).

(c) Note that, \wg, $\ext_R^0(\om,\,\om) = \hom_R(\om,\,\om) \cong R$ (cf.~Remark~\ref{rmk:not}) 
and $\ext_R^i(\om,\,\om) = 0$ for all $1 \le i \le h+1$ (cf.~Discussion~\ref{recap}(9)). 
We apply \hyperref[ddref]{$(\ddagger)$(4)} to $\page_2$ with $M = \om$, $k = h$ and $\nu = 0$ to conclude that \wg 
we have $\hom_R(\om,\, E) \cong \cH^h \cong H^{h-0}_P\big(\ext^{0}(\om, \, \om)\bigr) \cong H_P^h(R)$, as the 
nonzero terms $\bE_2^{ij}$ on the diagonals corresponding to total degrees $h-1$, $h$ and $h+1$ can only 
occur when $i = 0$. 

(d) First note that, for all $s \in \N$, we have following series of isomorphisms 
(in which the second isomorphism is by the adjointness of $\otimes$ and $\Hom$) 
\begin{align*}
\Ann_{\Hom_R(M,\,E)} P^s &\cong \hom_R\big(R/P^s,\,\Hom_R(M,\,E)\big)\\
&\cong \hom_R\big((R/P^s) \otimes_R M, \, E\big) \cong \hom_R(M/P^sM, \, E).
\end{align*}
Let $N_s := \hom_R(M/P^sM, \, E)$, for $s \in \N$.
It suffices to show that  there exists $g \in W$ such that for all $s \in \N$,  $(N_{s+1}/N_s)_g$
is $A_g$-free of finite rank.  Given what have been proved, we can localize at $g \in W$ so that
\ben[label=(\roman*)]
\item $\hom_R(A,\, E) \cong A$ and $\Ext_R^1(A,\, E) = 0$, by parts~(a) and (b) above. 
Therefore, $\hom_R(A^{\oplus k},\, E) \cong A^{\oplus k}$ and $\Ext^1_R(A^{\oplus k}, \,E) = 0$ for all $k \in \N$. 
\item For all $s \in \N$, the graded component $[\gr_P M]_s = P^sM/P^{s+1}M$ is $A$-free of finite rank,
by Lemma~\ref{grgenfree}. Consequently, for each $s \in \N$, $M/P^{s}M$ admits a finite filtration whose factors are free 
$A$-modules of finite rank.
\item In light of (i) and (ii) above, we see that $\Ext_R^1(P^sM/P^{s+1}M,\, E) = 0$ for all $s \in \N$. Moreover, 
a straightforward induction on the length of the filtration, as well as the long exact sequence for $\Ext$, shows that 
$\Ext_R^1(M/P^{s}M,\, E) = 0$ for all $s \in \N$.  
\item Similarly, by (i) and (ii) above, we see that $\hom_R(P^sM/P^{s+1}M,\, E)$ is $A$-free of finite rank for 
all $s \in \N$.
\een
We show that, under the conditions (i)--(iv), all of the modules $N_s/N_{s+1}$ are $A$-free of finite rank.
For all $s \in \N$, we have an exact sequence
$$ 0 \xra{\quad} P^sM/P^{s+1}M \xra{\quad} M/P^{s+1}M \xra{\quad} M/P^sM \xra{\quad} 0.$$
Additionally, as $\Ext_R^1(M/P^{s}M,\, E) = 0$ by (iii) above, we apply $\hom_R(\blank, \, E)$ to obtain the 
following exact sequence:
$$0 \xra{\quad} N_s \xra{\quad} N_{s+1} \xra{\quad} \hom_R(P^sM/P^{s+1}M, E) \xra{\quad} 0.$$
To complete the proof of (d), we observe that $\hom_R(P^sM/P^{s+1}M,\, E)$ is $A$-free of finite rank, by (iv) 
above.  Also note that $\bigcup_{s \in \N}\Ann_{\Hom_R(M,\,E)} P^s = \Hom_R(M,\,E)$, 
as every element of $\Hom_R(M,\,E)$ is annihilated by a power of $P$.

(e) Let $G_\bu$ denote a projective resolution of $M$ by finitely generated modules and 
let $\cC^\bu$ be the modified \Cech complex for $\ux$.
Then $H^i\big(\hom(G_\bu, \om)\big) \cong \ext_R^i(M, \om) $ and 
$H^{i}_P(M) \cong H^{i}_{(\ux)}(M) \cong H^{i}(M \otimes \cC^\bu)$.
Since, \wg, $\vect x h$ is a regular sequence with radical $P$ in $R$, we may take $\cC^\bu$ 
as a flat resolution of $H^h_{(\ux)}(R)$ and use it to calculate 
the Tor. Note, however, that $\cC^\bu$ is numbered for calculating cohomology. 
Therefore, for fixed $M$ and any finite set of choices for $i \in \Z$,  \wg\ we have
\begin{align*}
H^{i}_P(M) &\cong H^{i}_{(\ux)}(M) \cong H^{i}(M \otimes \cC^\bu) \\
&\cong \tor^{R}_{h-i}\big(M, \, H^h_P(R)\big) \cong H_{h-i}\big(G_{\bu} \otimes_R H^h_P(R)\big) \\
&\cong H_{h-i}\Big(G_{\bu} \otimes_R \hom_R\big(\om,\, H^h_P(\om)\big)\Big) &&\text{(by part~(d))}\\
&\cong H_{h-i}\Big(\hom_{R}\big(\hom_{R}(G_{\bu}, \om),\, H^h_P(\om)\big)\Big) &&
\text{(by Remark~\ref{natisom})}\\
&\overset{\alpha}{\cong} \hom_R\Big(H^{h-i}\big(\hom_R(G_\bu, \om)\big),\, H^h_P(\om)\Big) &&
\text{(by part~(a))}\\
&\cong \hom_R\big(\ext_R^{h-i}(M,\om),\, H^h_P(\om)\big).
\end{align*}
Here we would like to point out that, 
in order to see the isomorphism $\overset{\alpha}{\cong}$, we need 
$\hom_R\big(\blank,\,H^h_P(\om)\big)$ to preserve the exactness of several (but finitely many) short exact sequences, 
which can be achieved by repeated application of part~(a).

Next, we observe that the usual homotopy arguments show that 
the identification $$\hom_R\big(\ext_R^i(M,\om), H^h_P(\om)\big) \cong \tor^{R}_i\big(M, \, H^h_P(R)\big)$$
is independent of the choice of projective resolution $G_\bu$ for $M$, which is needed to see that
this is a natural isomorphism of functors on the variable module $M$.  

It remains to explain the statements in which $i$ is allowed to take on infinitely many values.
We already know that we may allow $0 \leq i \leq h$ or any larger \emph{finite} set of choices for $i$.
But since we have, \wg, that $P = \sqrt{(\vect x h)R}$, the local cohomology modules $H^{i}_P(M)$ vanish when
$i > h$, as do the modules involving $\Ext^{h-i}_R(M, \, \om)$ since $h - i < 0$. Thus, we have proved
that for fixed $M$, the result holds \wg for all integers $i \geq 0$.

For $i < 0$, the local cohomology modules $H^{i}_P(M)$ vanish.  
To complete the proof of (e), it remains only to prove that under the hypothesis that  $R$ or
$R_{g}$, for some $g \in R\setminus P$, is excellent of  finite Krull dimension, say $k$, 
we have for some $g' \in W$ and for all $j > k$ that
$$\Hom_{R_{g'}}\bigl(\Ext_{R_{g'}}^j(M_{g'}, \om_{g'}), \, E_{g'}\bigr) = 0,$$
since there are only finitely many integers $i$ such that $h+1 \leq h-i \leq k$.  
Consequently, it suffices to show that for some choice of $g' \in W$, we have
$\Ext_{R_{g'}}^j(M_{g'}, \om_{g'}) = 0$ for all $j > k$. For this it is enough to show that
the injective dimension of $\om_{g'}$ is at most $k$ for some $g' \in W$. This is true by 
Remark~\ref{open}, because $\om_Q$ is a canonical module for $R_Q$ for all primes 
$Q$ in a sufficiently small Zariski neighborhood of $P$. 

(f) This is immediate from (c) and (e). The proof is complete.
\end{proof}

\begin{remark}\label{quot}  Just as the usual form of local duality may be applied to homomorphic 
images of Gorenstein rings or to homomorphic images $\ov{R}$ of Cohen-Macaulay rings $R$ such 
that $R$ has a canonical module, it should be clear that the results of Theorem~\ref{main} can be applied 
to modules over a homomorphic image $\ov{R}$ of a ring $R$ satisfying the hypothesis of Theorem~\ref{main}. 
One can consider the modules over $\ov{R}$ as $R$-modules. Both the results in the next subsection and 
Theorem~\ref{hulls} provide an illustration of this technique.  \end{remark}

\begin{remark}\label{vanG}  There are many finitely generated modules $M$ for which the conclusion 
in the last sentence of part (e) of Theorem~\ref{main}, where $i$ is allowed to take on all values in $\Z$, 
holds on a sufficiently small Zariski neighborhood $D(g)$ of $P$ without any additional hypothesis on $R$.
This holds, for example, if $M_g$ has finite projective dimension over $R_g$ for some $g \in R \sm P$, 
which is equivalent to the statement that $M_P$ has finite projective dimension over $R_P$.
\end{remark}

\begin{remark}\label{avoid2} In results such as Theorem~\ref{main}, by Remark~\ref{avoid} 
the choice of $g$ can be made in $R\sm P$ such that $g$ also avoids finitely many primes not containing $P$.
\end{remark}

\begin{remark}\label{Lyu-Ch2} Using the results of an earlier version of this paper, S.~Lyu \cite[A.1.7, A.1.8]{Lyu25} has generalized Theorem~\ref{main} by assuming $\om$ to be a complex with \fg cohomology modules such that $\om_P$ is a dualizing complex of $R_P$.
\end{remark}

\section{Base change}\label{sec-basechange}  

\subsection{Base change for local cohomology and taking sections over an open set}\label{subsec-bc}  

We prove here several results on base change for local cohomology and for taking sections over an open set 
over a Noetherian ring $B$ that maps to  $R$, nearly all of which are, in essence, corollaries of Theorem~\ref{main}(f). 
However, some additional work is needed.

\begin{discussion}\label{base} We introduce some convenient notations and conventions, and
record some remarks about them.
\ben

\item If $S$ is a ring, we let $S\0$ denote the elements not in any minimal prime of $S$. \smallskip 

\item Fix an ideal $I \inc R$. Let $\Min(I)$ be the set of minimal primes of $I \inc R$ (also 
denoted by $\Min_R(R/I)$ in the literature). 
Let $W_I \subseteq  R \sm \big(\bigcup_{P \in \Min(I)}P\big)$ be a subset such that $W_I$ maps onto
$(R/I)\0$ under the natural map $R \to R/I$. 
This corresponds to our previous $W$ in \eqref{not} when $I = P$ is prime.
To study $H_I^i(M)$, we may assume that $\sqrt I = I$, and we explicitly do so as needed.  
\smallskip 

\item When we say {\it after localizing at one $w \in W_I$},
we typically make a statement in which $R$, $I$, $P$ and some given modules, e.g.,~$M$, have been replaced
by their localizations at the element $w$, but we continue to denote them by $R$, $I$, $P$, $M$, and
so forth. \smallskip  

\item By prime avoidance, we see $\left({\bigcap}_{\substack {P \neq P'\\ P,\,P' \in \Min(I)}} \big(P+P'\big)\right)  \sm 
\left(\bigcup_{P \in \Min(I)} P\right) \neq \emptyset$.
Consequently, $W_I \supseteq W_I':= W_I \cap  \left({\bigcap}_{\substack {P \neq P'\\ P,\,P' \in \Min(I)}} 
\big(P+P'\big)\right) \neq \emptyset$.  
Assume, as we may, that  $\sqrt I = I$. After localizing at one element $w \in W_I'$,  we have that, for all 
$n \in \N$,  $R/I^n \cong \prod_{P \in \Min(I)} R/P^n$ by the Chinese remainder theorem, which then 
implies $H^i_I(M) \cong \bigoplus_{P \in \Min(I)} H^i_P(M)$ for all $i \in\Z$.  \smallskip 

\item  If $H$ is any $R$-module in which every element is killed by a power of $I$, and if $g,\, g' \in R$ such that $g \equiv g' \mod I$,
then, by Remark~\ref{inv}, $H_g \cong H_{g'}$.  Consequently, in working with 
statements about the desired behavior of such modules $H$ after localizing 
at an element $g$ of $W_I$,   we are in the same position as if all elements of 
$R \sm \left(\bigcup_{P \in \Min(I)} P\right)$ were available as choices for $g$.     

\item By part~(5) above and Remark~\ref{avoid}, if we take $W_I = R \sm \big(\bigcup_{P \in \Min(I)}P\big)$ 
in Theorems~\ref{bc-reduced} and \ref{bc0-reduced}, then the element $g$ can be chosen in $W_I \cap R\0$. 
\een
\end{discussion} 
 
From Theorem~\ref{main}  and Discussion~\ref{base} part~(4) we have at once:

\begin{lemma}\label{support-I} Assume that \eqref{not} applies to all minimal primes of $I$. Then, we may 
localize at one element $w \in W_I$, such that for all $i \in \Z$, $H_I^i(M)$ is a finite direct sum of modules each 
of which is \ff for some $P \in \Min(I)$.
\end{lemma}

\begin{notation}\label{not.B} In \eqref{bc}, \eqref{bc-reduced} and 
\eqref{bc0-reduced}, we use the following notation: $B \to R$ is a homomorphism of 
Noetherian rings,  $I$ is an ideal of $R$,  $W_I$ is a subset of $R$ as described in (\hyperref[base]{\ref*{base} part (2)})
and $M$ is a finitely generated $R$-module. 
\end{notation} 

We note the following.  

\begin{remark} \label{flat-MandH} Let $B$ be a ring and let $(G^\bu, d^\bu)$ be any complex of $B$-modules 
\[
\tag{$G^\bu,\, d^\bu$}  \cdots \xra{\mathmakebox[20pt][c]{d^{i-1}}} G^i \xra{\mathmakebox[20pt][c]{d^{i}}} \ \cdots \  
\xra{\mathmakebox[20pt][c]{d^{n-1}}} G^{n} \xra{\mathmakebox[20pt][c]{d^{n}}} 0 \xra{\mathmakebox[20pt][c]{}} \cdots
\]
such that $G^i = 0$ for all $i > n$, for some given $n \in \Z$.  
Let $Z^i := \Ker(d^i)$ and $N^i := \im(d^{i-1})$, so that $H^i := H^i(G^\bu) =
Z^i/N^i$,  for all $i \in \Z$. Assume that the modules $G^i$ and $H^i$ are $B$-flat for all $i \in \Z$. Then:
\benn 
\item The modules $Z^i$ and $N^i$ are $B$-flat, for all $i \in \Z$.   
\item For any $B$-module $L$ and for all $i \in \Z$,  we have $H^i(G^\bu \otimes_B L) \cong H^i(G^\bu) \otimes_B L$,  i.e.,
the calculation of cohomology of $G^\bu$ commutes with base change to $L$.  
\een
Part (1) follows easily via reverse induction on $i$, starting with $i = n$, and the short exact sequences
$(\dagger) \ 0 \to N^i \to Z^i \to H^i \to 0$ and $(\ddagger)\ 0 \to Z^{i-1} \to G^{i-1} \to N^{i} \to 0$,  for $i \in \Z$.
In (2), since all of the modules $G^i,\, H^i, \, Z^i, \, N^i$ are $B$-flat from (1), the base change result is immediate from the short exact
sequences $(\dagger)$ and $(\ddagger)$.
\end{remark} 

\begin{discussion}\label{Cech}
Let $B \to R$, $M$ and $I$ be as in \eqref{not.B}. Let $\ux :=\vect x h$ generate $I$.  Let 
$X = \Spec(R)$, and $U := X \sm V(I)$. 
Consider the \Cech complex $\ccC^{\bu}(\ux^\infty; \, M)$ and the modified \Cech complex 
$\cC^{\bu}(\ux^\infty; \, M)$,  which are, explicitly,
\begin{align*}
\ccC^{\bu}(\ux^\infty; \, M): &&&0 \to \bigoplus_{1 \leq i \leq h} M_{x_{i}}\to \cdots \to \bigoplus_{1 \le i_1 < i_2 < \cdots < i_j \le h} M_{x_{i_1} \cdots x_{i_j}}  \to \cdots \to M_{x_1\cdots x_h} \to 0,\\
\cC^{\bu}(\ux^\infty; \, M): &&&0 \to M \to \cdots \to \bigoplus_{1 \le i_1 < i_2 < \cdots < i_j \le h} M_{x_{i_1} \cdots x_{i_j}} 
\to \cdots \to M_{x_1\cdots x_h} \to 0
\end{align*}
in which the typical module $\displaystyle\bigoplus_{1 \le i_1 < i_2 < \cdots < i_j \le h} M_{x_{i_1} \cdots x_{i_j}}$ is 
numbered as $\ccC^{j-1}(\ux\8; \, M)$ or as $\cC^{j}(\ux\8; \, M)$. In the notation of 
\cite[Notation, p.\ 649]{Rot09} or \cite[1.2.8, p.\ 9]{Weib94} 
for shifting degree in a complex, $\ccC^{\bu}(\ux\8; \, M)$ is $\cC^{\bu}(\ux\8;\,M)[1]$  brutally truncated, 
in the language of \cite[1.2.7,$\,$p.$\,$9]{Weib94}, by killing the degree $-1$ term, which is $M$.
   
Of course, $H^i_I(M)$ is the $i\,$th cohomology module of $\cC^\bu(\ux\8;\,M)$. 
In the sequel, let $\wcH^i_I(M)$ denote the $i\,$th cohomology module of $\ccC^\bu(\ux\8;\,M)$. 
Note that $\wcH^0_I(M)$ is the module of global sections of the sheaf $M^\sim$ on the open set $U$, 
often denoted $\Gamma(U,\, M^\sim)$. If $\cF$  is the restriction of $M^\sim$ to $U$, then
$$\wcH^i_I(M) = H^i\big(\ccC^\bu(\ux\8;\,M)\big) \cong H^i(U,\, \cF).$$  
Hence, for $i \geq 1$,  $\wcH^i_I(M) = H^i(U, \cF) \cong H^{i+1}_I(M)$. 

Scheme-theoretically, if $j:U \inj X$ is the open immersion
and $f: X \to \Spec(B)$ is $\Spec(B \to R)$,  then we may also write  $\wcH^i_I(M) = H^i\big(\ccC^\bu(\ux;\,M)\big) = 
f_*(R^ij_*(\cF))$.  We will need the latter notation when we extend the results on base change to schemes.
\end{discussion}

The following is important in proving our main results on base change, 
Theorem~\ref{bc-reduced} and Theorem~\ref{bc0-reduced}.

\begin{theorem}[Base change via flattening]\label{bc} Let notation be as in \eqref{not.B}. 
Suppose that  that after localizing at one element $b \in B$ and one element $g \in R$, 
the modules $M_{bg}$ and $H^i_I(M)_{bg}$, for all $i \in \Z$, are $B_b$-flat.
Then we have that:\\
$(\natural)$\phantomsection\label{bc-natural} For all $B_b$-modules or algebras $L$,  for all $i \in \Z$, 
$H^i_I(M_{bg} \otimes_{B_b} L) \cong H^i_I(M)_{bg} \otimes_{B_b} L$.%
\footnote{If $L$ is a $B$-module or a $B$-algebra,  
we may apply $\blank \otimes_B L$ and the base change result still holds, because the other module 
represented by $\blank$ is already a $B_b$-module.}\\
$(\natural_0)$\phantomsection\label{bc-natural-0} For all $B_b$-modules or algebras $L$, 
$\wcH^0_I\big(M_{bg} \otimes_{B_b} L\big)
\cong \wcH^0_I\big(M_{bg}\big) \otimes_{B_b} L$. 
\end{theorem}

\begin{proof} By replacing the $B$-algebra $R$ with the $B_b$-algebra $R_{bg}$, 
we may assume that the modules $M$ and $H^i_I(M)$, for all $i \in \Z$, are all $B$-flat. 

Consider ${\cC}^\bu(\ux\8; \, M)$ and ${\ccC}^\bu(\ux\8; \, M)$ as in \eqref{Cech},
so that $H^j_I(M) \cong H^{j}\big({\cC}^\bu(\ux\8; \, M)\big)$ and 
$\wcH^j_I(M) \cong H^{j}\big({\ccC}^\bu(\ux\8; \, M)\big)$ for all $j \in \Z$. 
Since $M$ and $H^j_I(M)$ are $B$-flat for all $j \in \Z$, \hyperref[bc-natural]{$(\natural)$} follows from an application of 
Remark~\ref{flat-MandH} to the complex ${\cC}^\bu(\ux\8; \, M)$. Moreover, by Remark~\ref{flat-MandH}(1), all the modules 
and cohomology modules of ${\ccC}^\bu(\ux\8; \, M)$ are $B$-flat. Thus, Remark~\ref{flat-MandH}(2), applied to the 
complex ${\ccC}^\bu(\ux\8; \, M)$, verifies \hyperref[bc-natural-0]{$(\natural_0)$}.
\end{proof} 

Note that in the following result we do not need to make any assumptions about existence of a canonical
module and so forth. Rather, we reduce to a case where the needed assumptions hold in order for us to apply Theorem~\ref{bc}. 

\begin{theorem}[Base change for local cohomology]\label{bc-reduced}
Let $B \to R$, $M$ and $R$ be as in Notation~\ref{not.B}.  Further assume that $B$ is reduced and that, 
after localization at one element $b' \in B\0$ and at one element $g' \in W_I$, 
the ring $(R/I)_{b'g'}$ is \fg as a $B_{b'}$-algebra.
Then after we localize at one element $b \in B\0$ and at one element $g \in W_I$,  
the base change result \hyperref[bc-natural]{$(\natural)$} holds.

Moreover, if $(R/I)_{b'}$ is module-finite over $B_{b'}$ for some $b' \in B \0$, 
then after we localize at one element $b \in B\0$, for all $B_b$-modules or algebras $L$, for all $i \in \Z$, 
we have $H^i_I(M_{b} \otimes_{B_b} L) \cong H^i_I(M)_{b} \otimes_{B_b} L$. 
\end{theorem}

\begin{remark} In this result, if $J:= \Ann_R M$, we may replace the condition that $(R/I)_{b'g'}$ be
\fg as a $B_{b'}$-algebra (respectively, in the second paragraph, that $(R/I)_{b'}$ be module-finite over $B_{b'}$)
by imposing the appropriate finite generation condition on $R/(I+J)$ instead, which is weaker.  
One simply replaces $R$ by $R/J$ in the proof. \end{remark}

\begin{proof}[Proof of Theorem~\ref{bc-reduced}] 
After changing notation, we may assume that $R/I$ is \fg as a $B$-algebra. 
By localizing at an element of $B\0$, and by change of notation, $B$ is a product of domains. 
Hence, we may assume that $B$ is a domain.
Also, after localizing at an element of $B\0$, we may assume that $B \inj R/P$ for all $P \in \Min(I)$. 
By Lemma~\ref{grgenfree}, we may localize at one element of $B\0$ and assume that all homogeneous
components of $\gr_I(M)$, as well as $R/P$ for all $P \in \Min(I)$, are $B$-free. 
(The conclusion that $R/P$ is generically $B$-free follows even from the classic Grothendieck version of 
generic freeness \cite[\S6.9]{EGAIV65}.) Next, replace $R$ and $M$ by their 
respective $I$-adic completions and replace $W_I$ by its natural image in the $I$-adic completion of $R$.  
This does not change $\gr_I(M)$, $R/P$ for all $P \in \Min(I)$, nor any of 
the local cohomology modules of $M$ with support in $I$.  Change notation and assume that
$R$ and $M$ are $I$-adically complete. By Proposition~\ref{B-flat}, $M$ is $B$-flat.  

Suppose that $R/I$ requires $s$ generators, say $\uu:= \vect us$, over $B$ and that
the ideal $I \inc R$ requires $k$ generators, say $\uz:=\vect zk$, over $R$. Let $\uU := \vect U s$ and $\uZ := \vect Z k$ be 
indeterminates over $B$.   Then we have a $B$-algebra map $\cR: = B[\uU][\kern-.1em[\uZ]\kern-.1em] \to R$ 
such that  $\uU \mapsto \uu$ and $\uZ \mapsto \uz$, which is surjective.\footnote{Every $\wh{v} \in \wh{R}^I$ is  
a formal sum $\sum_{t=0}^\infty v_t$ with $v_t \in I^t$,  
and since each $(\uZ)^t/(\uZ)^{t+1} \to I^t/I^{t+1}$  is onto, the partial sums of the series lift, 
recursively, to a power series in $\cR$ that maps to $\wh{v}$.}  
Let $\cI := \Ker\big(\cR \surj R \surj R/I\big)$ and $\cP := \Ker\big(\cR \surj R \surj R/P\big)$ for each $P \in \Min(I)$.  
Clearly, we have $\cR/\cI \cong R/I$, $\Min(\cI) = \{\cP \mid P \in \Min(I)\}$, $\cR/\cP \cong R/P$ and $\cR_{\cP} \surj R_P$.

Note that $\cR_{\cP}$ is regular (hence admits a canonical module) for each $\cP \in \Min(\cI)$, 
since we may test this after killing the regular sequence $\uZ$, and the result is a local ring of the regular ring 
$\fra(B)[\uU]$.  
As noted above, $M$ and $\cR/\cP \cong R/P$, for $\cP \in \Min(\cI)$, are all flat over $B$. 
By Lemma~\ref{support-I}, there exists $g \in W_I$ such that the modules 
$H^i_\cI(M)_{g}$ are finite direct sum of modules each of which is free filterable relative to $\cP\cR_g$, for some 
$\cP \in \Min(\cI)$, over $\cR_g$.
Thus, for all $i \in \Z$, $H^i_\cI(M)_g$ are flat over $B$.  We change notation and assume that we have localized 
so that the modules $M$ and $H^i_I(M) \cong H^i_\cI(M)$, for all $i \in \Z$, are $B$-flat.
Therefore, we can apply Theorem~\ref{bc} above to obtain \hyperref[bc-natural]{$(\natural)$}. 

In the final statement, after a change of notation, we may assume that $R/I$ is module-finite over $B$. 
Since $R/I$ is algebraic over the image of $B$, every $w \in W_I$ has a 
multiple in $R/I$ that is in the image of $B\0$. Thus, instead of localization at an element 
$w \in W_I$, it suffices to localize at an element in $B\0$. 
\end{proof}

Next, we prove base change results for the module of sections of the sheaf $M^\sim$ over
an open set $U := \Spec(R) \sm V(I)$, which is $\Gamma(U, M^\sim)$.  This is the $0\,$th cohomology 
of a \Cech complex, such as $\wcH^0_I(M)$ in Discussion~\ref{Cech}.  
Note that the final statement of the next theorem has been observed in \cite[the paragraph after Remark~1.2]{Sm18}.

\begin{theorem}[Base change for sections over an open set]\label{bc0-reduced} Let $B \to R$, $M$ and $R$ be as in 
Notation~\ref{not.B}.  Further assume that $B$ is reduced and that, after localizing at an element $b' \in B\0$ and 
an element $g' \in W_I$, the ring $R_{b'g'}$ is \fg as a $B_{b'}$-algebra. Then after we localize at one element $b \in B\0$ 
and one element $g \in W_I$, for every $B_b$-module or algebra $L$, the base change 
result \hyperref[bc-natural-0]{$(\natural_0)$} holds.

Moreover, if the ring $R_{b'}$ is \fg over $B_{b'}$ and $(R/I)_{b'}$ is module-finite over $B_{b'}$ for some $b' \in B\0$, 
then after we localize at one element $b \in B\0$, for every $B_b$-module or algebra $L$, 
$\wcH^0_I\big(M_{b} \otimes_{B_b} L\big) \cong \wcH^0_I\big(M_{b}\big) \otimes_{B_b} L$. See \cite{Sm18}.
\end{theorem}

\begin{remark}\label{bc0-reduced-rmk}  As usual, we may replace $R$ by $R/\Ann_R(M)$ in utilizing this theorem.  
\end{remark}

\begin{proof}[Proof of Theorem~\ref{bc0-reduced}] As in the proof of Theorem~\ref{bc-reduced}, by a change of notation, 
we may assume that $R$ is \fg as a $B$-algebra and that $B$ is a domain.
By Grothendieck's theorem on generic freeness \cite[\S6.9]{EGAIV65} or 
Lemma~\ref{grgenfree}, and as in the proof of Theorem~\ref{bc-reduced}, after localizing at one element $b \in B\0$ 
and one element $g \in W_I$,   
the modules $M_{bg}$ and $H^i_I(M)_{bg}$, for all $i \in \Z$, are $B_b$-flat. 
Now the base change result \hyperref[bc-natural-0]{$(\natural_0)$} follows immediately from Theorem~\ref{bc},
while the final statement follows exactly as in the last paragraph of the proof of Theorem~\ref{bc-reduced}.   
\end{proof}

We note the example from \cite[Remark 1.2]{Sm18} is such that the module of global sections cannot
be made $A$-free, where $A = \Z$.  However, it is $A$-flat.  

\subsection{Base change results for schemes} 

We conclude \S\ref{sec-basechange} with Theorem~\ref{bc-geom},
which is a geometric form of Theorems~\ref{bc-reduced} and \ref{bc0-reduced}. 
Here we adopt, insofar as possible, the same notation as in \cite[Theorem~1.1]{Sm18}. 
Theorem~\ref{bc-geom} generalizes the main result of \cite{Sm18}, in which $Z \to S$ is required to be finite. 
We require instead that $Z \to S$ be an affine morphism of finite type. In the result
of \cite{Sm18},  $S$ is replaced by base change to a dense open subscheme $S\0$. Here, in addition to this,
we also need to restrict to an open subscheme of $X$. 

\begin{notation}\label{2b-c}  Let $X \arwf{f} S$ be a morphism of Noetherian schemes, with $S$ reduced. 
Let $Z \inc X$ be a closed subscheme, and let $\cF$ be a coherent sheaf on $X\sm Z$. Assume as
well that $Z \to S$ is an affine morphism of finite type. 

By a {\it doubly generic base change} for this data we mean \emph{both} of the following steps:
\ben
\item A base change from $S$ to a dense open subscheme  $S\0$, replacing  $X, \, Z$, and $\cF$  by 
$X\0 := S\0 \times_S X$ and $Z\0 := S\0 \times_S Z$, and $\cF$ by $\cF\0 := S\0 \times_S \cF$. 

\item The replacement of  $X\0$ by a dense open subscheme ${X\oh}$ that meets every irreducible component $C$  of 
$Z\0$ in a dense open subset of $C$, along with the replacement of
$Z\0$ by $Z\oh := X\oh \cap Z\0$ and of $\cF\0$ by its restriction to $X\oh$.  
\een
\end{notation}

\begin{theorem}\label{bc-geom} Let $X \arwf{f} S$, $Z$, and $\cF$ be as in Notation~\ref{2b-c}. 
\bena
\item Then there is a doubly generic base change such that, if we keep our original notation after this base change,
with $U := X\sm Z$,   if $j:U \inj X$ denotes the open immersion, then for all $r \geq 1$ the sheaves  $f_*(R^rj_*\cF)$ 
are locally free over $S$, and their calculation commutes with base change. 

\item If $X$ itself is of finite type over $S$, then the conclusion of part (a) holds for all $r \geq 0$. 

\item If, in addition to (a) (respectively to (b)), the morphism  $Z \to S$ is, locally on a dense open set in $S$, a finite morphism, 
one may take $S\0$ such that $X\oh = X\0 = S\0 \times_S X$,  so that only the base change on $S$ described
in Notation~\ref{2b-c}(1) is needed. 

\een
\end{theorem}

\begin{proof} (a) One can evidently make an initial choice of  $S\0$ such that it is the disjoint union of its irreducible 
components,  which are reduced.  Thus, one comes down to the case where $S$ is reduced and irreducible. 
The reduction to the case where  $S = \Spec(B)$ is affine with $B$  a domain, and $Z = \Spec(R/I)$, with 
$R/I$ of finite type over $B$, is identical to the argument in \cite[pp.~13,\,14]{Sm18}, 
except that here we use the notation $B$ instead of $A$,
and in the argument here we have that $R/I$ is affine of finite type over $B$ because 
we assume that the morphism $Z \to S$ is an affine morphism of
finite type. (In \cite{Sm18}, the hypothesis instead is that $Z \to S$ is finite; of course, 
this implies affine of finite type.) The construction
of $X\oh \to S\0$ can be done locally: it corresponds to passing to $B_b \to R_{bg}$ as in Theorems~\ref{bc-reduced} 
and \ref{bc0-reduced}, 
with $b \in B\0$ and $g \in R \sm \big(\bigcup_{P \in \Min(I)}P\big) =:W_I$, where $g$ can also be chosen in $R\0$ by Discussion~\ref{base}(6), 
for suitable open affines in $S$ and $X$ and taking a union in both schemes. Since we can use $b \in B\0$
 and $g \in W_I \cap R\0$  we have  that $S\0 \inj S$ is {\it dense} open in $S$ and $X\oh$ is {\it dense} open in $X\0$.
 
Exactly as in \cite{Sm18},
we may assume that $\cF$ is the restriction of a coherent sheaf, also denoted $\cF$, on $X\oh$, which corresponds to an $R$-module $M$ since 
we are in the affine case. Then, in our local set-up, $f_*(R^rj_* \cF)$  may
be identified with  $H^{r+1}_I(M)_{bg}$, and the stated result now follows from Theorem~\ref{bc-reduced}, which gives 
condition \hyperref[bc-natural]{$(\natural)$}. 

(b) This follows from the fact that in the situation described in Theorem~\ref{bc0-reduced}, where $R$ is \fg as a $B$-algebra, the 
base change result also holds for $r = 0$.

(c) For this statement, we may reduce to the affine case with $B$ a domain and with $R/I$ integral over $B$,   
so that instead of localizing $B$ at  $b$ and $R$ at $bg$ we may localize both $B$ and $R$ at $bb'$, where $b' \in B\0$ 
is an element  whose image in $R/I$ is a multiple of the image 
of $g$ in $R/I$ (cf.~the last statements of Theorems~\ref{bc-reduced} and \ref{bc0-reduced}).
\end{proof}

\section{Comparison with the results of Karen Smith}\label{Smith}  

The results of Karen Smith \cite{Sm18} were a great inspiration in developing the theory of \S2.  
Many of the results of \cite{Sm18} are connected with cohomology of sheaves on schemes and base 
change, which we have already discussed in \S\ref{sec-basechange}.  
Much of \S\ref{sec-basechange},  including the proof of  Theorem~\ref{bc-geom}, 
reflects the methods in \cite{Sm18}.   We focus in this brief section on the underlying results in 
commutative algebra that are used in proving the base change results. 
In the process, we discuss their connections with Theorem~\ref{main}. 
We first note \cite[Theorem 2.1]{Sm18}:

\begin{theorem}[K.~E.~Smith]  Let $A$ be a Noetherian domain, let 
$R = A[\kern-.1em[\vect xh]\kern-.1em]$ be a power series ring over $A$, and let $M$ be 
a finitely generated $R$-module. Denote by $I$ the ideal $(\vect xh) \incn R$. 
Then the local cohomology modules $H^i_I(M)$ are generically free over $A$ and 
commute with base change for all $i$. 
\end{theorem}

We second note \cite[Corollary 1.3]{Sm18}:

\begin{corollary}[K.~E.~Smith]\label{Smith-bc}  
Let $A$ be a Noetherian reduced ring such that $A \to R$, where $R$ is Noetherian, 
and let $I = (\vect xh)$ be an ideal of $R$ such that the composite map $A \to R \to R/I$ is 
module-finite. Let $M$ be an $R$-module such that $M/IM$ is finitely generated over $A$.  
Then there exists $g \in A \sm \bigcup_{\fp \in \Min(A)}\fp$ such that $H^j_I(M_g)$ is $A_g$-free for all $j$.  
Moreover,  $H^j_I(M \otimes_A L) \cong H^j_I(M)\otimes_AL$ for every $A_g$-module $L$.  
In particular, this holds when $L$ is an $A$-algebra such that $A \to L$ factors through $A_g$. 
\end{corollary}

Another key result in \cite{Sm18} is that when $R = A[\kern-.1em[\vect X h]\kern-.1em]$, where $A$ is a 
domain without loss of generality,  and $M$ is finitely generated over $R$ as above, one has, after localizing 
at one element of $A\smo$, that 
$H^j_I(M) \cong \hom_R\big(\Ext^{h-j}_R(M,R), H^h_I(R)\big)$ for all $j \ge 0$.

\begin{discussion}\label{A-base-change}
Evidently, to obtain the base change result of \cite[Corollary 1.3]{Sm18}, we may  apply the last statement of 
Theorem~\ref{bc-reduced}.

Let $R = A[\kern-.1em[\vect X h]\kern-.1em]$, where $A$ is a Noetherian domain, and let $M$ be a finitely 
generated $R$-module, as above. Let $P = (\vect X h)R$, $\om = R$, and $W = A\smo$. Note that $R_P$ 
is regular and, hence, $\om_P$ is an canonical module for $R_P$. 
By Theorem~\ref{main}, we see that, \wg, $H^j_P(M)$ is \ff and 
$H^j_P(M) \cong \hom_R\big(\Ext^{h-j}_R(M,\om), H^h_P(\omega)\big)$ for all $j \ge 0$.
Now, as the composition $A \inc R \surj R/P$ is an isomorphism, Proposition~\ref{ffAfree} tells 
us that, \wg, $H^j_P(M)$ is $A$-free for all $j \ge 0$.

Of course, in the more general situation of this paper, as in Theorem~\ref{main}, one gets local 
cohomology  with  filtrations that have $A$-free factors \wg.  One cannot hope for more in the general
case, since the local cohomology modules, typically, are not $A$-modules.
\end{discussion}  

\begin{discussion}\label{comp}
Note that \cite{Sm18} states many results in terms of $\Hom\cts_A(M, A)$, which consists of the 
$A$-linear maps from the $R$-module $M$ to $A$ that vanish on $I^tM$ for some $t$. 
Here $A$, $R$, $I$ and $M$ are as in Corollary~\ref{Smith-bc}, but the situation reduces to the case 
where $R$ is a power series over a domain $A$. 
With $R = A[\kern-.1em[\vect x h]\kern-.1em]$ and $ I = (\vect xh)R$,  we have that $H^h_I(R) \cong \Hom\cts_A(R,\, A)$,
which is the same as the module $E$ in our Theorem~\ref{main}  if one takes $\om := R$ and $P: = I$.
Moreover,  as noted in \cite[Proposition~3.1]{Sm18}, the functor $M \mapsto \Hom_R(M, E)$ is naturally isomorphic to
the functor $M \mapsto \Hom\cts_A(M, A)$ on $R$-modules $M$: in essence, this is just the usual adjointness of tensor
and Hom, restricted to maps that kill $I^tM$ for some $t$.  It is then easy to see that the results of Theorem~\ref{main}, 
when applied to the case where $R = A[\kern-.1em[\vect xh]\kern-.1em]$, $P = (\vect x h)R$, $W = A\smo$ and 
$\om = R$,  are the same as the results of \cite{Sm18}, in light of the fact that, by Proposition~\ref{ffAfree}, \ff 
implies $A$-free.   
\end{discussion}

Finally, we note that the result on generic local duality in \cite[Theorem~5.1]{Sm18} is stated 
there as follows: 

\begin{theorem}[K.~E.~Smith]Let $R$ be a power series ring $A[\kern-.1em[\vect x h]\kern-.1em]$ over a Noetherian domain 
$A$, let $I = (\vect x h)R$, and let $M$ be a finitely generated $R$-module. Then, after replacing $A$ by its localization at one 
element of $A \smo$, for all $i \ge 0$, there is a functorial isomorphism 
$$H^i_I(M) \cong  \Hom\cts_A\big(\Ext^{h-i}_R(M, R),\, A \big).$$\end{theorem}

As noted in Discussion~\ref{comp} just above, the $R$-module $\Hom\cts_A\big(\Ext^{h-i}_R(M, R),\, A \big)$ is naturally 
isomorphic with $\Hom_R\big(\Ext^{h-i}_R(M, R),\, H^h_I(R)\big)$. 
Thus, our result on  generic local duality in part~(e) of Theorem~\ref{main}, applied to 
$R = A[\kern-.1em[\vect xh]\kern-.1em]$, agrees with the result on generic local duality in \cite[Theorem~5.1]{Sm18}.

\section{\texorpdfstring{\Rpf}{(R || P)-pseudoflat} modules}\label{Rpf}  

In this section, we discuss \Rpf $R$-modules. See Definition~\ref{def:Rpf}.

\begin{proposition}\label{reRpf} Let $R$ be a ring, let $\fP$ be a proper ideal in $R$, and let $A:=R/\fP$.
\bena
\item A direct limit of \Rpf $R$-modules is \Rpf. 
\item A flat $A$-module is \Rpf. 
\item Let $0 \to M' \to M \to M'' \to 0$ be an exact sequence of $R$-modules.  
\beni
\item If $M'$ and $M''$ are \Rpf then so is $M$. 
\item If $M$ and $M''$ are \Rpf then so is $M'$.
\een
\item An $R$-module with a countable ascending filtration whose factors are
\Rpf is \Rpf. 
\item An arbitrary direct sum of \Rpf $R$-modules is \Rpf. 
\item An $R$-module with a countable ascending filtration whose factors are
flat $A$-modules is \Rpf. 
\item If an $R$-module is \ffr $\fP$ then it is \Rpf. 
\item A module with a finite right resolution by \Rpf $R$-modules is \Rpf. 
\item If $A$ is regular, an $A$-module $M$ is \Rpf if and only if it is $A$-flat. 
\item If $R \to R'$ is flat over $R$ and $M$ is \Rpf then $M \otimes_R R'$ is \Rp{R'}{\fP R'}. 
\een
\end{proposition}
\begin{proof} Parts~(a), (b), (c) and (j) follow from the definition of \Rpf and basic facts
about flat base change and the behavior of (possibly improper) regular sequences. 
Parts~(d) and (e) follow from (c) and (a),  while
part~(f) follows from (b) and (d).  Part~(f) implies part~(g), while  part~(h) follows from
part~(c)(ii)  by induction on the length of the resolution. To prove (i), note that if $M$ is $A$-flat 
then it is \Rpf by part~(b).  Now suppose $M$ is an $A$-module and \Rpf.  This is preserved
when make a base change to any local ring of $A$, and so we may assume that $A$ is regular local.
The result now follows from the assertion that if $A$ is regular local and every regular sequence on
$A$ is a possibly improper regular sequence on $M$, then $M$ is $A$-flat. The argument is given
in \cite[6.7, p.~77]{HH92},  where it is not ever used that the regular sequences $\vect x k$ on $A$ are proper regular
sequences on $M$:  one still has that $\Tor^A_i\big(A/(\vect x k), \, M\big) = 0$ for $i \geq 1$ when $\vect x k$
are possibly improper regular sequences on $M$.
\end{proof}
 
We also note:
  
\begin{proposition}\label{HdRpf} Let $\fP$ be a proper ideal in $R$ and let $A:=R/\fP$.  
Also, let $R'$ be a flat extension of $R$, let $\uy:= \vect y d  \in R'$ 
be a sequence of elements whose images form a possibly
improper regular sequence on $R'/\fP R'$, and let $\uy^t = y_1^t, \, \ldots, \, y_d^t$. If
\[
\tag{$*$}  0 \xra{\quad} M' \xra{\quad} M \xra{\quad} M'' \xra{\quad} 0
\]
is an exact sequence of $R$-modules such that $M''$ is \Rpf,  then for all $t \in \N$, the sequences
\[
\tag{$*_t$}  0 \xra{\quad} \frac{M' \otimes_R R'}{(\uy^t)(M' \otimes_R R')} 
\xra{\quad}  \frac{M \otimes_R R'}{(\uy^t)(M \otimes_R R')} 
\xra{\quad} \frac{M'' \otimes_R R'}{(\uy^t)(M'' \otimes_R R')} \xra{\quad} 0
\]
and the sequence 
\[
\tag{$**$}   0 \xra{\quad} H^d_{(\uy)}(M'\otimes_R R') \xra{\quad} H^d_{(\uy)}(M\otimes_R R') \xra{\quad} H^d_{(\uy)}(M''\otimes_R R') \xra{\quad} 0
\] 
are exact. 
Hence, for any exact sequence $0 \to M_1 \to M_2 \to \cdots \to M_n \to 0$, with $n$ finite, 
and given any $i_0 \in \{1,\,2\}$, if $M_i$ is \Rpf for all $i \neq i_{0}$, then the module $M_{i_0}$ is \Rpf and the whole sequence remains exact when we apply either $\blank \otimes_R R'/(\uy^t)$ or $H^d_{(\uy)}(\blank\otimes_R R')$. 
\end{proposition}

\begin{proof} Since the pseudoflatness conditions are preserved by the flat base change $R \to R'$ 
(cf.~Proposition~\ref{reRpf}(j)), we may assume that $R' = R$.
The sequences $(*_t)$ are standard and reduce by induction to the case $d=1$.
The sequence $(**)$ is the direct limit of the sequences $(*_t)$. 

The final statement concerning the exact sequence $0 \to M_1 \to M_2 \to \cdots \to M_n \to 0$ 
follows at  once by induction on $n$, the length of the exact sequence, 
with $n = 3$ proved just now in $(*_t)$ and $(**)$ above. 
The conclusion that $M_{i_0}$ is \Rpf can be obtained by repeated applications of Proposition~\ref{reRpf}(c).
\end{proof}

\begin{corollary}\label{rseqfilt}  
Let $\fP$ be a proper ideal in $R$ and let $A:=R/\fP$.  Let $R'$ be a flat extension of $R$, and let
$\uy = \vect yd \in R'$ be a sequence of elements whose images in $R'/\fP R'$ form a possibly
improper regular sequence. Suppose that 
$$
0 = M_0 \inc M_1 \inc \cdots \inc M_n \inc \cdots \inc M
$$ 
is a finite or countably infinite filtration of $M$ over $R$, so that $M = \bigcup_{n \geq 1} M_n$, 
and denote the factors as $N_i := M_i/M_{i-1}$,  $i \geq 1$.   
If the modules $N_i$ are \Rpf for all $i \geq 1$, then $M$ is \Rpf and, for all $t$, the modules 
\[
\frac{M_i \otimes_R R'}{(\uy^t)(M_i \otimes_R R')} 
\quad \text{ (respectively, $H^d_{(\uy)}(M_i \otimes_R R')$),} \quad i \ge 0,
\]
give a corresponding filtration of $\dfrac{M \otimes_R R'}{(\uy^t)(M \otimes_R R')}$ (respectively, 
$H^d_{(\uy)}(M \otimes_R R')$) with factors $\dfrac{N_i \otimes_R R'}{(\uy^t)(N_i \otimes_R R')}$ (respectively,
$H^d_{(\uy)}(N_i \otimes_R R')$). 
\end{corollary} 

Before stating the next corollary, we set up the following notation, which is convenient in 
describing elements of top local cohomology:

\begin{notation}\label{lcnot} Let $\uy = \vect y d \in R$ and let $M$ be an $R$-module, so that we may view 
$H^d_{(\uy)}(M)$ as $\dlim t M/(\uy^t)M$ as usual, where $(\uy^t) := (y_1^t, \, \ldots, y_d^t)R$.
For $u \in M$, we use the notation $[u;\, \uy^t]$ for the natural image of $u+(\uy^t)M \in M/(\uy^t)M$ in 
$H^d_{(\uy)}(M)$.
Note that when $\uy$ is a regular sequence on $M$,  we may think of $M/(\uy^t)M$ as a natural 
submodule of $H^d_{(\uy)}(M)$.  As long as there is no confusion, we also use the same notation 
$[u;\, \uy^t]$ to denote the corresponding element in $H^d_{(\uy)}(M)_Q$ after localization, where $Q$ 
is a prime ideal of $R$.
\end{notation}

The following result plays a critical role in the proof of Key Lemma~\ref{mainloclemma},  and consequently in the proofs
of Theorems~\ref{mainsmct} and~\ref{maine}.  
 
\begin{corollary}\label{Annihilator}  Let $R$ and $\fP$ be as in Corollary~\ref{rseqfilt}. 
Let $M$ be an $R$-module, and let $u \in M$ be an element such that both $M$ and $M/Ru$ are \Rpf. 
Also, let $R'$ be a flat extension of $R$ and let $\uy:= \vect y d  \in R'$ whose images form a possibly
improper regular sequence on $R'/\fP R'$. Consider $u \otimes 1_{R'} \in M \otimes_R R'$ and,
using Notation~\ref{lcnot} above,  consider
$[u \otimes 1_{R'};\uy] \in H^d_{(\uy)}(M \otimes_R R') \cong M \otimes_R H^d_{(\uy)}(R')$.
Then $\Ann_{R'}([u \otimes 1_{R'};\uy]) = \Ann_R(u)R' + (\uy)R'$.
\end{corollary}

\begin{proof}  Since $\Ann_R u$  commutes with the flat base change $R \to R'$ and since the 
pseudoflatness conditions are preserved
by that base change (cf. Proposition~\ref{reRpf}(j)), we may assume that $R' = R$.  
Note that the module $Ru$ is \Rpf as well by Proposition~\ref{reRpf}(c). Thus, by Corollary~\ref{rseqfilt}, we have an injection 
$Ru \otimes_R R/(\uy) \inj H^d_{(\uy)}(Ru) \inj H^d_{(\uy)}(M)$
that takes $u \otimes \ol{1}$ to $[u;\uy] \in H^d_{(\uy)}(M)$, where we agree that 
$u \otimes \ol{1} \in Ru \otimes_R R/(\uy)$. 
Thus $\Ann_{R}([u;\uy]) = \Ann_{R}(u \otimes \ol 1)$, which is $\Ann_{R}\big(Ru \otimes_R R/(\uy)\big)$
since $u \otimes \ol 1$ generates $Ru \otimes_R R/(\uy)$.
But $\Ann_R\big(Ru \otimes_R R/(\uy)\big)$
is immediate from the fact that $Ru \otimes_R R/(\uy) \cong R/\Ann_R(u) \otimes_R R/(\uy) \cong R/\big(\Ann_R(u)+(\uy)\big)$. 
\end{proof}

\begin{remark}\label{annrmk} Of course, in Corollary~\ref{Annihilator}, we may also apply the 
result with $\uy$ replaced by $\uy^t = y_1^t, \, \ldots, y_d^t$, which is also a regular sequence, 
and obtain $\Ann_{R'}([u \otimes 1_{R'};\uy^t]) = \Ann_R(u)R' + (\uy^t)R'$. 
\end{remark}

The next corollary combines results in this section with the main result of \S\ref{genld} (Theorem~\ref{main}):

\begin{corollary}\label{HHom-exact}
Let $R,\,P,\,E,\,W$ be as in Theorem~\ref{main}. Let $\Delta$ be a finite set and 
consider a family of complexes $\{M_{i1} \to M_{i2} \to M_{i3}\}_{i \in \Delta}$, where for all $i \in \Delta$ and
$1 \leq j \leq 3$  the $M_{ij}$ are  finitely generated $R$-modules. After we localize at one element
of $W$ and change notation to call the localized ring $R$, for all flat $R$-algebras $R'$ and $\uy = \vect yd \in R'$ 
whose images form a possibly improper regular sequence on $R'/PR'$, the functor 
$H^d_{(\uy)}\big(\hom(\blank \otimes_R R',\, E \otimes_R R')\big)$ commutes with taking homology at the $M_{i2}$ 
spot for all $i \in \Delta$.  Hence, $H^d_{(\uy)}\big(\hom(\blank \otimes_R R', E \otimes_R R')\big)$ commutes with 
taking homology of any finite collection of finite complexes of finitely generated $R$-modules 
after localizing at one element of $W$ that is dependent on the set of complexes.
\end{corollary}

\begin{proof}
The information about the complexes  and their homology is determined by a finite family of short exact
sequences. Thus, it suffices to prove that exactness is preserved when we apply
the functor to one short exact sequence.  By parts~(a) and (d) of Theorem~\ref{main}, after we localize at one 
element of $W$,  when we apply $\Hom_{R'}(\blank \otimes_R R', \, E \otimes_R R')$ we get a short exact 
sequence of $R'$-modules that are all \ffr $PR'$ and hence are all \Rp{R'}{PR'}. The result now follows from 
part~(g)  of Proposition~\ref{reRpf} and Proposition~\ref{HdRpf}, applied over $R'$.
\end{proof}

\begin{discussion}[A non-Noetherian acyclicty criterion] 
We first observe that the acyclicity criterion in  \cite{BuEi73} is valid without Noetherian assumptions.  
This is shown in \cite[\S6]{Nor76}, and a much more concise treatment is available in \cite{Ho02}.  
One needs to use a notion of depth based on vanishing of Koszul homology:  this agrees with the usual notion in the 
Noetherian case. Given a finitely generated ideal $I$ of an arbitrary ring $R$ and an $R$-module $M$,
define the \emph{Koszul depth} of $M$ on $I = (\vect r n)$ to be $d$,  where $d$ is the least value of $i \in \N$
such that  $H_{n-i}(\vect r n;\, M) \not=0$.  This is independent of the choice of generators of $I$. \end{discussion}

\begin{theorem}[Acyclicity Criterion]\label{acyc}  Let $R$ be an arbitrary ring, not necessarily Noetherian.
Let $G_\bu$ denote a finite complex of finite rank nonzero free modules over $R$,  say 
$$
G_\bullet: \qquad 0 \xra{\quad} R^{b_n} \xra{\quad} \cdots \xra{\quad} R^{b_1} \xra{\quad} R^{b_0} \xra{\quad} 0.
$$
Let $N \not=0$ be an arbitrary $R$-module.  Let $\alpha_i$ denote the
matrix of the map $G_i \to G_{i-1}$,  $1 \leq i \leq n$.  Let $r_0 = 0$ and for $i \geq 1$, 
let $r_i$ denote the determinantal rank of $\alpha_i$ modulo $\Ann_R N$. 
Then $G_\bu \otimes_R N$ is acyclic if and only if the following
two conditions hold:
\benn
\item  $b_i = r_i + r_{i-1}$,  for all $1 \leq i \leq n$. 
\item The Koszul depth of $N$ on $I_{r_i}(\alpha_i) \geq i$, for all $1 \leq i \leq n$. 
\een
\end{theorem}

\begin{remark}\label{rank} Note that condition (1) is equivalent to the assumption that 
$r_n = b_n$, $r_{n-1}  = b_n - b_{n-1}$, 
$r_{n-2} = b_n-b_{n-1}+ b_{n-2}$, $\ldots$, $r_{n-i} = \sum_{t=0}^i (-1)^t b_{n-t}$, $\ldots$, 
and $r_1 = b_n - b_{n-1} + \cdots + (-1)^{n-1}b_1$.  \end{remark}

\begin{remark}\label{nzdrank} Consider a complex $G_\bu$ as in Theorem~\ref{acyc}, and let the $r_i$ be as in
Remark~\ref{rank}.  We claim that if $R$ has Koszul depth at least one  on each of the ideals $I_{r_i}(\alpha_i)$, 
where $1 \leq i \leq n$, then it is automatic that the ideals $I_{r_i+1}(\alpha_i)$ are all $0$, $1 \leq i \leq n$.  
To see this, note that after a faithfully flat extension
of $R$, each of the ideals $I_{r_i}(\alpha_i)$, $1 \leq i \leq n$, will contain a nonzerodivisor.  Inverting the 
product of these does not affect the issue, so that we may assume that each $I_{r_i}(\alpha_i)$ is the 
unit ideal.  We use induction on $n$. The case $n=1$ is obvious. Assume $n \geq 2$.
The issue is local on $R$, so that we may assume that $R$ is qasilocal.  After changes of basis 
in $G_{n-1}$, which do not affect the issue, we may assume that $\alpha_n$ has the block form 
$\bpm \id_{b_n} \\ 0 \epm$,  where the $0$ block has size $(b_{n-1} - b_n) \times b_n$. 
Now it is clear that $I_{r_n+1}(\alpha_n) = I_{b_n+1}(\alpha_n) = 0$.
The fact that  $G_\bullet$ is a complex implies that $\alpha_{n-1}$ has the block form 
$\bpm 0_{b_{n-2} \times b_n} & \alpha_{n-1}' \epm$, where $\alpha_{n-1}'$ is $b_{n-2} \times (b_{n-1}-b_n)$.   
That is, $G_\bullet$ is the direct sum of a complex of length $n-1$, namely 
$$
G'_\bullet: \qquad 0 \xra{\quad} R^{b_{n-1} -b_n} \arwf{\alpha_{n-1}'} R^{b_{n-2}} \arwf{\alpha_{n-2}} \dotsb  
\arwf{\ \alpha_2\ } R^{b_1} \arwf{\ \alpha_1\ } R^{b_0} \xra{\quad} 0$$
and a complex with nonzero terms only at the spots indexed by $n$ and $n-1$, namely 
$$0 \xra{\quad} R^{b_n} \arwf{\ \id_{b_n}\ } R^{b_n} \xra{\quad} 0.$$ 
Then it is routine to see that $I_{j}(\alpha_{n-1}) = I_{j}(\alpha_{n-1}')$ for all $j$. 
In particular, we have $I_{r_{n-1}}(\alpha_{n-1}) = I_{b_{n-1}-b_n}(\alpha_{n-1}')$.
Now the remaining claims, i.e., $I_{r_i+1}(\alpha_i)=0$ for $i = 1, \dotsc,n-1$, follow from the induction 
hypothesis on $G'_\bullet$.
\end{remark}

We are now ready to state the result mentioned earlier.  

\begin{theorem} \label{Thm:acyclicity}Let $R$ be a Noetherian ring, $P$ a prime ideal of $R$,  and $A:=R/P$.
Let $G_{\bullet}$ be  complex of finitely generated free $R$-modules of finite length $n$ such
that $A \otimes_R G_{\bullet}$ is acyclic.  Let $\cH$ be an $R$-module and assume at least one of the following
\benn
\item $\cH$ is \ff. 
\item $\cH$ is a direct limit of modules each of which has a finite filtration with $A$-flat factors.  
\item $\cH$ is \Rpp and $P$ contains all primes in $\Ass_R(R)$ (the latter holds, for example, if $R$ is a domain). 
\een
Then $\cH \otimes_R G_\bullet$ is acyclic.
\smallskip

In particular, under the hypotheses of Notation~\ref{not}, if $\fra(A) \otimes_R G_\bu$ is acyclic, then 
given a finitely generated $R$-module $M$, 
after localizing at one element of $W$, we have that 
$H^j_P(M) \otimes G_\bu$ is acyclic. \end{theorem}

\begin{proof}  Note that condition~(1) implies condition~(2).  Assume that (2) holds, so that
$\cH$ is direct limit of modules $\cH_\lambda$ each of which has a finite filtration by flat 
$A$-modules.  Since  $A\otimes_R G_\bu$ is acyclic, so is the tensor product with any flat $A$-module,
and then one has that the tensor product with each $\cH_\lambda$ is acyclic by iterated use of the snake lemma.  
By taking a direct limit, one sees that the tensor product with $\cH$ is acyclic.  

We now prove that condition (3) is sufficient for acyclicity.  Assume that $\cH$ is \Rpp.  
We adopt the detailed description of $G_\bullet$ given in 
the statement of Theorem~\ref{acyc}.  Let the $r_i$ be determined from the $b_i$ as in Remark~\ref{rank}.  
The hypothesis that $A \otimes_R G_{\bullet}$ is acyclic shows that for each $i$, $1 \leq i \leq n$,
the ideal of $r_i$ minors of $\alpha_i$, when passed to $A$, is either the unit ideal $A$ or contains an 
$A$-regular sequence oflength at least $i$. The ideals $I_{r_i}(\alpha_i)$ over $R$ therefore 
each contain an element  $f_i \in R \setminus P$,  which is a nonzerodivisor on $R$. By Remark~\ref{nzdrank}
and Remark~\ref{rank},  these numbers $r_i$ are indeed the determinantal rank of $\alpha_i$ on $\cH$ for 
$1 \leq i \leq n$, 
which allows us to conclude that condition~(1) of Theorem~\ref{acyc} holds for $\cH \otimes_R G_\bullet$.  
Moreover, it follows that each $I_{r_i}(\alpha_i)$ contains $i$ elements of $R$
that map to a possibly improper regular sequence in $A$.   
Hence, these elements form a possibly improper regular sequence on $\cH$, 
which verifies that condition~(2) of Theorem~\ref{acyc} holds for $\cH \otimes_R G_\bullet$.
The acyclicity of $\cH \otimes_R G_\bullet$ now follows from Theorem~\ref{acyc}.

For the final statement on $H^j_P(M) \otimes G_\bu$, first note that because $H^j_P(M)$ is supported
only at $P$,  localizing at an element  $g \in W$ yields the same result for $H^j_P(M) \otimes G_\bu$  
if $g$ is replaced by an element
of $R \sm P$ with the same image as $g$ in $R/P$.  Therefore we can carry out the proof with $W = R \sm P$.
The result  follows because if $\fra(A) \otimes_R G_\bu$ is acyclic, then after replacing
$R$ by its localization at one element of $W$, the complex $A \otimes_R G_\bu$ becomes acyclic, and
by Theorem~\ref{main}, $H^j_ P(M)$ is \ff after localizing at one element
of $W$.   Now we may apply part~(1) of this theorem.
\end{proof}

\section{Generic behavior of injective hulls}\label{hull}  

In this section we study generic behavior for homomorphic images of rings that satisfy 
Notation~\ref{not}.  Since our focus is on the homomorphic image, we change notations as indicated 
just below. We start with a ring that satisfies Notation~\ref{not}, but we denote this ring by $\tR$, so that 
Theorem~\ref{main} holds for $\tR$ and $\tP \in \spec(\tR)$. The focus of this section is $R:=\tR/\fa$, 
where $\fa$ is an ideal of $\tR$ such that $\fa \subseteq \tP$. Denote $P := \tP/\fa \in \spec(R)$. 
Because we want $R_P$ to have Krull dimension $h$,  we will use $\tth$ 
for the Krull dimension of $\tR_{\tP}$. In the first subsection we discuss what happens without taking 
a quotient.  The second subsection deals with 
the case where we work with $R = \tR/\fa$.  The third subsection 
is focused on the case when $R_P$ is \stwo. Note that we do not need to change the notation for $A$,
since $\disp A = \tR/\tP \cong \frac {\tR/\fa}{\tP/\fa}$.

\begin{notation}\label{notat} Let $\tR$, $\tP$, $A:= \tR/\tP$, $\tth: = \height(\tP)$, and $\tom$ be as in 
Notation~\ref{not}. That is, let $\tR$ be a Noetherian ring, $\tP$ a prime ideal of $\tR$ such that $\tR_\tP$ 
is Cohen-Macaulay of Krull dimension $\tth$.  Assume also that $\tom$ is a finitely generated $\tR$-module 
such that $\tom_\tP$ is a canonical module for $\tR_\tP$. Let $\tE := H^\tth_{\tP}(\tom)$ and 
$\tW := \tR \sm \tP$. For every prime $\fp$ of $\tR$ let $\kappa_\fp$ denote $\tR_\fp/\fp \tR_\fp$, which is 
naturally isomorphic with $\fra(\tR/\fp)$. Also, denote $P :=\tP/\fa$, which is a prime of $R$. However, we 
do \emph{not} assume that $R_P$ is \stwo in this section until subsection~\ref{s2}. 
\end{notation}

Note that, for simplicity,  we are taking $\tW$ to be $\tR \sm \tP$ itself rather than a multiplicative system 
in $\tR$ or a subset of $\tR$ that naturally maps onto $A\setminus \{0\}$. 

Consider the following conditions:
\ben[label=(E\arabic*)]
\item \label{E1} The regular locus in $\Spec(A)$  has non-empty interior, i.e., there exists $g_1 \in \tW$ such that $A_{g_1}$ is
regular.
\item \label{E2} There exists $g_2 \in \tW$ such that  $\tR_{g_2}$ is \CM and $\tom_{g_2}$ is a global canonical 
module for $\tR_{g_2}$. 
\item \label{E3} If $R_P$ is \stwo then $P$ is in the interior of the \stwo locus of $R$; that is, if $(\tR/\fa)_\tP$ is \stwo then there exists $g_3 \in \tW$ such that  $(\tR/\fa)_{g_3}$ is \stwo. 
\een

\begin{remark}\label{S2} By \cite[Cor. 5.10.9]{EGAIV67},  every local ring that is catenary and \stwo is equidimensional.
By \cite[Prop. 6.11.8]{EGAIV67},  the \stwo locus is open in an excellent ring. 
\end{remark}

\begin{proposition}\label{e12} If $\tR_g$ is excellent for some $g \in \tW$, then \ref{E1}, \ref{E2} and \ref{E3} all hold. 
\end{proposition} 
\begin{proof}  It suffices to assume that $\tR$ is excellent. Then (E1) is clear, since $A = \tR/\tP$ 
is an excellent domain. For (E2), we can localize at an element of $\tW$
so that $\tR$ is \CM, and so that $\tom$ is a global canonical module.   See Remark~\ref{open}. 
Finally, (E3) follows from Remark~\ref{S2}. 
\end{proof}

\subsection{Generic behavior of injective hulls near \texorpdfstring{$\tP$ if $\tR_\tP$ is \CM}{a prime ideal in \CM locus}}\label{gebE1}
We shall show how to obtain injective hulls for quotients $\tR/\tQ$ over $\tR$ for all prime ideals $\tQ$ 
in an open neighborhood of $\tP$ in $V(\tP)$.  

\begin{theorem}\label{Ecm} Let $\tR$, $\tP$, $\tth$, $A:=\tR/\tP$, $\tW : = \tR \sm \tP$, $\tom$, and $\tE$ be 
as in Notation~\ref{notat}.  Assume that $\tR$ is excellent {\it or} that conditions \ref{E1} and \ref{E2} above hold. 
  Then we can localize at one element $g \in \tW$ such that after replacing
$\tR$, $\tP$, $\tom$, and $\tE$ by their localizations at $g$, all the following statements hold:
\bena  
\item The ring $A$ is regular, the ring $\tR$ is \CM, and the module $\tom$ is a canonical module for $\tR$, 
i.e., for all $\tQ \in \Spec(\tR)$,  $\tom_\tQ$ is a canonical module for $\tR_\tQ$.   
\item $\Ann_\tE \tP \cong A$, so that we have an injection $0 \xra{\ \ } A \inj \tE$. 
\een

In all of the remaining parts, assume that we have localized at an element of $\tW$ so that conditions
{\rm(a)} and {\rm(b)} hold. 
Moreover, in the remaining parts, we place the following condition \eqref{Ecm-sharp} on $\tQ$ and $\uy= \vect yd$:
\[\tag{$\natural$}\label{Ecm-sharp}
\begin{split}
\qquad &\text{The ideal $\tQ \in \Spec(\tR)$  is in $V(\tP)$ and the sequence $\uy$  in $\tQ$ maps}\\
&\text{to a \sop  for  $A_\tQ \cong A_Q$.}
\end{split}
\]
Then we have the following:
\benr
\item After localization at one element of $\tW$, for all $\tQ$ and $\uy$ satisfying \eqref{Ecm-sharp}, 
the module $H^d_{(\uy)}(\tE)_\tQ$ is an injective hull for $\tR_\tQ/\tQ R_\tQ$ over $\tR_\tQ$, i.e., $H^d_{(\uy)}(\tE)_\tQ \cong \E_\tR(\tR/\tQ)$. 
\item For any given finitely generated $\tR$-module $M$,  after localizing at one element of $\tW$, for
all  $\tQ$ and $\uy$ satisfying \eqref{Ecm-sharp}, the natural map 
$$\Phi_M:H^d_{(\uy)}\big(\hom_\tR(M, \tE)\big)_\tQ \xra{\qquad} \hom_\tR\big(M, H^d_{(\uy)}(\tE)\big)_\tQ$$ 
is an isomorphism. 
\item After localizing at one element of $\tW$, for all $\tQ$ and $\uy$ satisfying \eqref{Ecm-sharp}, the natural map  $H^d_{(\uy)}(\Ann_\tE \tP)_\tQ \inj H^d_{(\uy)}(\tE)_\tQ$ is an injection, 
and a socle generator for $H^d_{(\uy)}(\Ann_\tE \tP)_\tQ \cong H^d_{(\uy)}(A)_\tQ$ over $A_\tQ$ maps to 
a socle generator for $H^d_{(\uy)}(\tE)_\tQ$ over $\tR_\tQ$.  
\een
\end{theorem}

\begin{proof} Part~(a) is immediate from the hypotheses and Proposition~\ref{e12}.  
Part~(b) follows from part~(b) of Theorem~\ref{main} applied over $\tR$. 
 
(c) Let $\tQ$ and $\uy = \vect y d$ satisfy \eqref{Ecm-sharp}, so that $\dim(R_Q) = \tth + d$.
Note that $\tom_Q$ is assumed to be a canonical module for $\tR_\tQ$ by part~(b) above,  and the maximal
ideal of $\tR_\tQ$ is the radical of $\tP+(\uy)$. 
Let $\ux = \vect x \tth$ be a sequence of $\tth$ elements of $\tR$  
whose images form a \sop for $\tR_\tP$. Localize at one element of $\tW$ so that the radical of $(\ux)$ is $\tP$ and, hence, the radical of $\big((\ux) + (\uy)\big)R_Q$ is $QR_Q$. 
Then  $\E_\tR(\tR/\tQ) \cong \E_{\tR_\tQ}(\tR_\tQ/\tQ\tR_\tQ)$ can be realized as follows:
\begin{align*} 
\E_{\tR_\tQ}(\tR_\tQ/\tQ \tR_\tQ) \cong H^{\dim(R_Q)}_{Q}(\tom_\tQ) &\cong
 H^{\tth+d}_{(\ux,\, \uy)}(\tom_\tQ) \cong H^d_{(\uy)}\big(H^\tth_{(\ux)}(\tom_\tQ)\big)\\
&\cong H^d_{(\uy)}\big(H^\tth_\tP(\tom)_\tQ\big) = H^d_{(\uy)}(\tE_\tQ) \cong  H^d_{(\uy)}(\tE)_\tQ.
\end{align*} 

(d) Take a presentation $G_1 \to G_0 \to M \to 0$ where $G_0$ and $G_1$ are finitely generated free $\tR$-modules. 
By Corollary~\ref{HHom-exact} applied over $\tR$, we may localize at one element of $\tW$ such that exactness is preserved at $G_0$ and $M$ when
we apply $H^d_{(\uy)}\big(\hom(\blank, \tE)\big)_\tQ$ for all $\tQ$ and $\uy$ satisfying \eqref{Ecm-sharp}.
 
Note that $H^d_{(\uy)}(\blank)_\tQ$ is isomorphic with $H^d_{(\uy)}(\tR)_\tQ\otimes_\tR \blank$. With this
identification, the natural map $\Phi_M$ is induced by the bilinear map  
$$H^d_{(\uy)}(\tR)_\tQ \times \Hom_\tR(M,\tE) \xra{\qquad} \Hom_\tR\big(M, H^d_{(\uy)}(\tR)_\tQ \otimes_{\tR}\tE\big)$$
whose value on $(h, f)$  is the map $u \mapsto h \otimes f(u)$.  
It is straightforward to check that $\Phi_R$ is an isomorphism, and that $\Phi_{M \oplus M'}$
may be identified with $\Phi_M \oplus \Phi_{M'}$, so that $\Phi_G$ is an isomorphism for any finitely
generated free $\tR$-module $G$. 
To verify that $\Phi_M$ is an isomorphism, we make use of the presentation $G_1 \to G_0 \to M \to 0$ of $M$ over $\tR$.  
For all $Q$ and $\uy$ satisfying \eqref{Ecm-sharp},  we have the commutative diagram below
in which the rows are exact: 
$$\xymatrix@C=20pt{    
0 \ar[r]  &H^d_{(\uy)}\big(\hom_R(M, \tE)\big)_\tQ  \ar[r]\ar[d]_{\Phi_M}  &H^d_{(\uy)}\big(\hom_\tR(G_0, \tE)\big)_\tQ \ar[r]\ar[d]_{\Phi_{G_0}}^{\cong} & H^d_{(\uy)}\big(\hom_\tR(G_1, \tE)\big)_\tQ \ar[d]_{\Phi_{G_1}}^{\cong} \\
0  \ar[r] & \hom_\tR\big(M, H^d_{(\uy)}(\tE)_\tQ\big) \ar[r] &  \hom_\tR\big(G_0, \,H^d_{(\uy)}(\tE)_\tQ\big) \ar[r] &
\hom_\tR\big(G_1, H^d_{(\uy)}(\tE)_\tQ\big)
}
$$
The top row is exact because of the paragraph right above the diagram, 
while the bottom row is exact thanks to the left exactness of 
$\hom_\tR\big(\blank,H^d_{(\uy)}(\tE)_\tQ\big)$. 
Since we have already shown that both $\Phi_{G_0}$ and $\Phi_{G_1}$ are isomorphisms, it follows from 
the five lemma that the map $\Phi_M$ is an isomorphism.

(e) Applying Corollary~\ref{HHom-exact} to the exact sequence $\tR \to \tR/\tP \to 0$ over $\tR$, we see that, 
after localization at one element of $\tW$ as needed, there is an injection 
$H^d_{(\uy)}(\Ann_\tE \tP)_\tQ \inj H^d_{(\uy)}(\tE)_\tQ$ for all $Q$ and $\uy$ 
satisfying \eqref{Ecm-sharp}, in which all the maps are natural. Hence, a socle generator 
for $H^d_{(\uy)}(\Ann_\tE \tP)_\tQ \cong H^d_{(\uy)}(A)_\tQ$ over $A_\tQ$ maps to a socle generator 
for $H^d_{(\uy)}(\tE)_\tQ \cong \E_{\tR_\tQ}(\tR_\tQ/\tQ \tR_\tQ)$ over $\tR_\tQ$. 
\end{proof} 

\subsection{Generic behavior of injective hulls for images of \CM rings}\label{genquo} 
In this subsection, the focus is the quotient ring $R :=\tR/\fa$, with $P:= \tP/\fa$, where $\fa \inc \tP$. 
We shall show how to obtain injective hulls for $R/Q$ over $R$ for all prime ideals $Q: = \tQ/\fa$ in 
 open neighborhood of $P$, where $\fa \inc \tQ$. 

\begin{remark}\label{homcomm} Let $\tR$ be a Noetherian ring and $N$ an $\tR$-module.  
Throughout the next result, we freely use the natural identification 
$\Hom_\tR(M, N_\tQ) \cong \Hom_\tR(M,N)_\tQ$ when $M$ is finitely generated
(and so finitely presented).  Note that $N$ need not be finitely generated.  We likewise use the identification 
$\Hom_\tR(\tR/\fa, N) \cong \Ann_N \fa$,  so that we have $\Ann_{N_\tQ} \fa \cong (\Ann_N \fa)_\tQ$.  
\end{remark}

\begin{theorem}\label{hulls} Let $\tR$, $\tP$, $\tth$, $A:=\tR/\tP$, $\tW : = \tR \sm \tP$, $\tom$, and $\tE$ be as above.  
Assume that $\tR$ is excellent or that conditions \ref{E1} and \ref{E2} above hold. 
Let $\fa \inc \tP$ be an ideal of $\tR$. Denote $R = \tR/\fa$, $P = \tP/\fa$ and $E = \Ann_\tE \fa$. 
For any prime $\tQ \supseteq \tP$, denote $Q = \tQ/\fa \inc R$.
Then we can localize at one element $\tg \in \tW$ such that, after replacing
$\tR$, $\tP$, $\tom$, $R$  and $\tE$ by their localizations, the following statements hold:
\bena 
\item The ring $A$ is regular, the ring $\tR$ is \CM, and the module $\tom$ is a global canonical module for $\tR$.
\item $A \cong \Ann_\tE \tP = \Ann_E \tP = \Ann_EP$, so that we have an injection $A \inj  E \inc \tE$. 
Write $\Ann_EP = Au$ for some $u \in E$. \label{hulls-Au}
\een

In all of the remaining parts, assume that we have localized at an element of $\tW$ so that conditions
{\rm(a)} and {\rm(b)} hold. Note that localization of an $R$-module at an element $\tg \in \tR \sm \tP$ is
the same as localization of that $R$-module at $g \in R \sm P$, where $g$ is the natural image of $\tg$.
In particular, the localization of an $R$-module at a prime $\tQ \supseteq \tP$ is the same as its localization
at $Q := \tQ/\fa$. 
Moreover, in the remaining parts, we place the following condition \eqref{hulls-pound} on $\tQ$ and $\uy= \vect yd$:
\[\tag{\#}\label{hulls-pound}
\begin{split}
\qquad &\text{The ideal $\tQ \in \Spec(\tR)$  is in $V(\tP)$ and the sequence $\uy$  in $\tQ$ maps}\\
&\text{to a regular \sop  for  $A_\tQ \cong A_Q$.}
\end{split}
\]
Then we have the following:
\benr
\item After localizing at one element of $\tW$, for all $\tQ$ and $\uy$ satisfying \eqref{hulls-pound}, 
the inclusion $E = \Ann_\tE \fa \inj \tE$ induces the following injective homomorphism 
$$
H^d_{(\uy)}(E)_Q  \cong H^d_{(\uy)}(E)_\tQ = H^d_{(\uy)}(\Ann_\tE \fa)_\tQ \cong  
\Ann_{H^d_{(\uy)}(\tE)_\tQ}\fa \inj H^d_{(\uy)}(\tE)_\tQ. 
$$ 
Therefore, $H^d_{(\uy)}(E)_Q$ is an injective hull for $R_Q/QR_Q$ over $R_Q$, i.e., 
$H^d_{(\uy)}(E)_Q  \cong \E_{R_Q}\big((R/Q)_Q\big)$.
\item After localizing at one element of $\tW$, for all $\tQ$ and $\uy$ satisfying \eqref{hulls-pound}, 
the natural map 
$$
H^d_{(\uy)}(\ann_E P)_Q \inj H^d_{(\uy)}(E)_Q 
$$
is an injection, under which a socle generator for $H^d_{(\uy)}(\ann_E P)_Q \cong H^d_{(\uy)}(A)_Q$ 
over the regular local ring $A_Q$ maps to a socle generator for 
$H^d_{(\uy)}(E)_Q \cong  \Ann_{H^d_{(\uy)}(\tE)_\tQ}\fa \cong \E_{R_Q}\big((R/Q)_Q\big)$ over $R_Q$. 
In particular, with $u \in E = \Ann_\tE \fa$ as in part~\ref{hulls-Au} above, the element $[u;\uy] \in H^d_{(\uy)}(E)_Q$ 
is a socle generator. 
\een
\end{theorem}

\begin{proof} Parts~(a) and (b) hold for the same reason as in Theorem~\ref{Ecm}.

(c) The induced isomorphism $H^d_{(\uy)}(E)_\tQ = H^d_{(\uy)}(\Ann_\tE \fa)_\tQ 
\cong  \Ann_{H^d_{(\uy)}(\tE)_\tQ}\fa$ follows from part~(d) of Theorem~\ref{Ecm} 
with $M:= \tR/\fa$ and Remark~\ref{homcomm}.  As we have noted earlier, for $\tR$-modules 
killed by $\fa$ it does not matter whether we think of them as $\tR$-modules 
and localize at $\tQ$ or think of  them as $R$-modules and localize at $Q$.   Hence, 
$H^d_{(\uy)}(\Ann_\tE \fa) )_Q \cong  \Ann_{H^d_{(\uy)}(\tE)_\tQ}\fa$. Moreover, we 
have $H^d_{(\uy)}(\tE)_\tQ \cong \E_\tR(\tR/\tQ)$ by part~(c) of Theorem~\ref{Ecm}. 
Therefore,
\[H^d_{(\uy)}(E)_Q = H^d_{(\uy)}(\Ann_\tE \fa) )_Q \cong 
\Ann_{\E_\tR(\tR/\tQ)}\fa \cong \E_R(R/Q).\] 

(d) Applying Corollary~\ref{HHom-exact} to the exact sequence $\tR/\fa \to \tR/\tP \to 0$ over $\tR$, 
we see that, after localization at one element of $\tW$ as needed, there is 
a natural injection $H^d_{(\uy)}(\Ann_\tE \tP)_\tQ \inj H^d_{(\uy)}(E)_\tQ$ for all $\tQ$ and $\uy$ 
satisfying \eqref{hulls-pound}. In terms of $A \cong Au = \Ann_E P \subseteq E$ and in terms of modules 
over $R_Q$, we have a natural injection $H^d_{(\uy)}(A)_Q \cong 
H^d_{(\uy)}(Au)_Q  = H^d_{(\uy)}(\Ann_E P)_Q \inj H^d_{(\uy)}(E)_Q$.
Hence, a socle generator for $H^d_{(\uy)}(Au)_Q \cong H^d_{(\uy)}(A)_Q$ over $A_Q$ maps to 
a socle generator for $H^d_{(\uy)}(E)_Q \cong \E_{R_Q}\big((R/Q)_Q\big)$ over $R_Q$. 
Corresponding to the socle generator $[1;\uy] \in H^d_{(\uy)}(A)_Q$, the element 
$[u;\uy] \in H^d_{(\uy)}(Au)_Q$ is a socle generator for  $H^d_{(\uy)}(Au)_Q$. 
Therefore, the element $[u;\uy] \in H^d_{(\uy)}(E)_Q$ is a socle generator 
for $H^d_{(\uy)}(E)_Q \cong \E_{R}(R/Q)$ over $R_Q$.
\end{proof}

\subsection{Generic behavior of injective hulls for \stwo quotients of \CM rings}\label{s2}
In this subsection, we  keep all of the conventions of Notation~\ref{notat}. In addition, 
we assume that $R_P$ is \stwo, where $R= \tR/\fa$ and $P =\tP/\fa$. 
We show that after localizing at one element of $\tW$,  we can obtain injective hulls of residue
class fields of $R$ at primes $Q := \tQ/\fa \in V(\tP/\fa)$ from the top nonvanishing local
cohomology module of a finitely generated  $R$-module.  
We need some preliminaries.

\begin{discussion}\label{S2canon-g} 
Let $\tR$ be a (not necessarily excellent or local) \CM ring with global 
canonical module $\tom$, let $\fa$ be an ideal of $\tR$ such that all minimal 
primes of $\fa$ have the same height, say $k$. Assume that $R: = \tR/\fa$ is \stwo. 
Let $\om := \Ext^k_{\tR}(R, \tom)$. When $R$ is Cohen-Macaulay, this is a canonical module for 
$R$.  However, when $R$ is \stwo and when all minimal  primes of $\fa$ have the same height 
$k$, this module has some of the same behavior as in the \CM case 
and, in some literature, is still referred to as a canonical module for $R$. 
In particular, without assuming that $R$ is \CM, we have the following 
(some of what we state below is in the literature in one form or another, 
cf.~\cite[Prop.~2]{Ao80}, \cite[Satz 6.12]{HK71} or \cite[Thm.~1.2]{Ao83} for example,
but we follow the statement with a brief, self-contained treatment):
\ben
\item For every finitely generated $R$-module $M$, we have 
$\ext^k_{\tR}(M, \tom) \cong \Hom_R(M, \om)$. 
Moreover, $\ext^k_{\tR}(\om, \tom) \cong \Hom_R(\om, \om) \cong R$, 
in which $R \cong \Hom_R(\om, \om)$ can be induced by the homothety map. 
\item $H^{\dim(R_Q)}_{Q}(\om_Q) \cong \E_R(R/Q) \cong \E_{R_Q}\big((R/Q)_Q\big)$ 
for all prime $Q$ of $R$. 
\item For all finitely generated $R$-module $M$ and for all prime $Q$ of $R$, we 
have $H^{\dim(R_Q)}_{Q}(M_Q) \cong 
\Hom_{R_Q}\Big(\Hom_R(M, \om)_Q,\,\E_{R_Q}\big((R/Q)_Q\big)\Big)$. 
In particular, we have $H^{\dim(R_Q)}_{Q}(R_Q) \cong 
\Hom_{R_Q}\Big(\om_Q, \, \E_{R_Q}\big((R/Q)_Q\big)\Big)$ 
for all prime $Q$ of $R$. 
\een

To prove (1), choose a maximal regular sequence $\uz = \vect z k$ in $\fa$. 
As $\uz \in \fa$ and $\uz$ is regular on $\tom$, we have 
$\om = \Ext^k_{\tR}(R, \tom) \cong \Hom_\tR\big(R, \tom/(\vect z k)\tom\big)$. 
But then, for any finitely generated $R$-module $M$, we get 
\begin{multline*} 
\Ext^{k}_\tR(M, \tom) \cong \Hom_\tR\big(M, \tom/(\uz)\tom)\big) 
\cong \Hom_\tR\big(M, \Hom_\tR(R, \tom/(\uz)\tom)\big) \\ 
\cong \Hom_\tR(M, \om) = \Hom_R(M, \om). 
\end{multline*}
In particular, $\Ext^{k}_\tR(\om, \tom) \cong \Hom_R(\om, \om)$. 
To show that the map $\theta \colon R \to \Hom_R(\om, \om)$ induced by homothety 
is an isomorphism, we need to prove that the kernel and cokernel are $0$.  If we localize 
at any prime $\fp: = \tfp/\fa \in \spec(R)$ of height $1$  or $0$, we see that $R_\fp$ is \CM and, 
hence, $\om_{\fp} = \Ext^{\dim(\tR_\tfp) - \dim(R_\fp)}_{\tR_\tfp}(R_\fp, \tom_\tfp)$ is a true canonical module 
of $R_\fp$, because $R$ is equidimensional. 
This shows that $\theta$ becomes an isomorphism whenever we localize at such a 
prime $\fp$.   Thus, the kernel and cokernel of the map $\theta$ are supported only at primes 
of height $2$ or greater.  Note that both $R$ and $\Hom_R(\om, \om)$ are \stwo. 
Now we can see that both $\Ker(\theta)$ and $\Coker(\theta)$ must be zero. 
If the kernel is not $0$, then localize at a minimal prime of its support, which must be in 
$\ass_R\big(\Ker(\theta)\big) \subseteq \ass_R(R)$, to get a contradiction. 
Hence $\theta$ is an injective map. Suppose that the cokernel is not $0$. 
After the localization at a minimal prime its support, the cokernal becomes depth $0$ while both 
$R$ and $\Hom_R(\om, \om)$ have depth at least $2$, which is a contradiction. 

Now that we have (1), we can prove (2) as follows. Let $Q = \tQ/\fa$ be any prime ideal of 
$R = \tR/\fa$, where $\fa \subseteq \tQ \in \spec(\tR)$. 
Then $k = \dim(\tR_\tQ) - \dim(R_Q)$. Therefore, 
\begin{align*}
H^{\dim(R_Q)}_{QR_Q}(\om_Q) 
&\cong \Hom_{\tR_\tQ}\big(\ext^k_{\tR}(\om, \tom)_\tQ,\, \E_{\tR_\tQ}(\tR_\tQ/\tQ_\tQ)\big) &&
 \text{(duality over $\tR_\tQ$)}\\
& \cong \Hom_{\tR_\tQ}\big((\tR/\fa)_\tQ,\, \E_{\tR_\tQ}(\tR_\tQ/\tQ_\tQ)\big) &&
\text{(part (1) here)}\\ & \cong \E_{R_Q}\big((R/Q)_Q\big) \cong \E_R(R/Q) .
\end{align*}

To see (3), let $M$ be an finitely generated $R$-module, and let $Q = \tQ/\fa$ be as above. Then, by local duality over $\tR_\tQ$, we get
\begin{align*}
H^{\dim(R_Q)}_{QR_Q}(M_Q) 
&\cong \Hom_{\tR_\tQ}\big(\ext^k_{\tR}(M, \tom)_\tQ,\, \E_{\tR_\tQ}(\tR_\tQ/\tQ_\tQ)\big) &&\text{(local duality)}\\
& \cong \Hom_{\tR_\tQ}\big(\Hom_R(M, \om)_\tQ,\, \E_{\tR_\tQ}(\tR_\tQ/\tQ_\tQ)\big) &&\text{(part~(1))}\\
& \cong \Hom_{R_Q}\big(\Hom_R(M, \om)_\tQ,\, \E_{R_Q}(R_Q/Q_Q)\big). 
\end{align*}
In particular, we obtain $H^{\dim(R_Q)}_{Q}(R_Q) \cong \Hom_{R_Q}\Big(\om_Q, \, \E_{R_Q}\big((R/Q)_Q\big)\Big)$. 
\end{discussion}

\begin{theorem}\label{main-g}
Let $\tR$ be a (not necessarily excellent or local) ring, let $\tP \in \spec(\tR)$ and let $\tom$ be 
a finitely generated $\tR$-module.  Assume that $\tR_\tP$ is \CM and that $\tom_\tP$ is a canonical 
module for $\tR_\tP$.  Fix an ideal $\fa \subseteq \tP$ and denote $R:= \tR/\fa$, $\tP/\fa =: P \in \spec(R)$. 
Assume that $R_P$ is \stwo. Let $\tth: = \dim(\tR_\tP)$, $h: = \dim(R_P)$, and $k:=\tth - h$. 
Set $\om := \Ext^k_{\tR}(R, \, \tom)$, $E:= H^{h}_{P}(\om)$ and $\tE := H^{\tth}_{\tP}(\tom)$. 
Also, let $\tW \inc \tR$ be a subset that maps onto $A \sm \{0_A\}$ under the natural map $\tR \to \tR/\tP$. 
Similarly, let $W \inc R$ be a subset that maps onto $A \sm \{0_A\}$ under the natural map $R \to R/P$. 
Then the following holds: 
\bena
\item After replacing $\tR$ and $R$ with $\tR_\tg$ and $R_\tg$ for some $\tg \in \tR \sm \tP$, 
we have $\ext^k_{\tR}(\om, \tom) \cong \Hom_R(\om, \om) \cong R$, 
in which the isomorphism $R \cong \Hom_R(\om, \om)$ can be induced by the homothety map. 
\item 
\twg, we have $H^{h}_{P}(\om) \cong \hom_\tR\big(\tR/\fa,\,H^{\tth}_{\tP}(\tom)\big) 
\cong \ann_{H^{\tth}_{\tP}(\tom)}\fa$, i.e., $E \cong \hom_\tR\big(\tR/\fa,\,\tE\big)$. 
Hence, \wg, we have $\ann_{H^{h}_{P}(\om)}P \cong A$. 
\item If $0 \to M' \to M \to M'' \to 0$ is a sequence of 
finitely generated $R$-modules that becomes exact after localization at $P$, then the 
induced sequence
$$0 \xra{\quad} \Hom_R(M'',\, E) \xra{\quad} \Hom_R(M,\, E) \xra{\quad}\Hom_R(M',\, E) \xra{\quad} 0$$ 
is \wg exact.  Hence, for every finitely generated $R$-module $M$ and for every $i \geq 1$, 
the module $\Ext^i_R\big(M, H^h_P(\om)\big)$ is \wg $0$. 
Therefore, the functor $\hom_R(\blank, \, E)$ is \wg exact on 
any given finite set of short exact sequences of finitely generated $R$-modules, and, hence, on
any given finite set of finite long exact sequences of finitely generated $R$-modules. The choice of $g$
depends on which finite set of finite exact sequences one chooses. 
\item Let $N: = \Hom_R(M, E)$ or $N := H^i_P(M)$, where $M$ is a finitely generated $R$-module.  Then  after 
localizing at one  element of $W$ \emph{that is independent of $n \in \N$}, the filtration of $N$ by the modules 
$\Ann_N P^n$,  which ascends as $n$ increases with $\bigcup_{n \in \N}\Ann_N P^n = N$, has 
factors that are finitely generated free modules over $A$.  
\item For every finitely generated $R$-module $M$, \wg, we have 
$H^{h}_{P}(M) \cong \Hom_{R}\big(\hom_R(M,\om), \, H^{h}_{P}(\om)\big)$. In particular, 
\wg, we get $H^{h}_{P}(R) \cong \Hom_{R}\big(\om, \, H^{h}_{P}(\om)\big)$. 
\een
\end{theorem}

\begin{proof} 
The proof is an application of Theorem~\ref{main} to $\tR$, together with Discussion~\ref{S2canon-g}.
As in the proof of Theorem~\ref{main}, it suffices to assume $\tW = \tR \sm \tP$ and $W = R \sm P$. 
We may use $\tg$ to denote an element of $\tW$ and use $g \in W$ to denote the natural image of $\tg$ 
under the natural map $\tR \to R$. In the course of the proof we may repeatedly, but finitely many times,  
localize at one element  $\tg \in \tW$. Each time, we make a change of terminology, and continue to use $\tR$ 
and $R$ to denote the resulting ring  $\tR_\tg$ and $R_\tg$.  Likewise, we use $\tP$ and $P$ for their extensions 
to $\tR_\tg$ and $R_\tg = R_g$, and for every module under consideration we use the same letter for the module 
after base change from $\tR$ or $R$ to $\tR_\tg$ or $R_g$ (finitely generated modules under consideration are 
replaced by their localizations: this is the same as base change from $\tR$ or $R$ to $\tR_\tg$ or $R_g$).
For an $R$-module $M$, we naturally identify $M_\tg$ with $M_g$. 

(a) Naturally, we may write $R_P = \tR_\tP/\fa_\tP$ and $\om_P = \Ext^k_{\tR_\tP}(R_P, \, \tom_\tP)$. Since $\tR_\tP$ is 
Cohen-Macaulay with $\tom_\tP$ being a canonical module, and since $R_P$ is \stwo, we apply 
Discussion~\ref{S2canon-g}(1) to obtain $\ext^k_{\tR_\tP}(\om_\tP, \tom_\tP) \cong \Hom_{R_P}(\om_P, \om_P) \cong R_P$, 
in which $R_P \cong \Hom_{R_P}(\om_P, \om_P)$ can be induced by the homothety map.
Thus, after replacing $\tR$ with its localization at one element of $\tW$ and, accordingly, replacing $R$ with its 
localization at the corresponding element of $W$, we obtain $\ext^k_{\tR}(\om, \tom) \cong \Hom_R(\om, \om) \cong R$, 
in which $R \cong \Hom_R(\om, \om)$ can be induced by the homothety map. 
So, \twg, $\ext^k_{\tR}(\om, \tom) \cong \tR/\fa$. 

(b) By Theorem~\ref{main}(e) applied to $\tR$, we see that, \twg, 
\begin{align*}
H_P^h(\om) &\cong \hom_\tR\big(\ext^{\tth-h}_\tR(\om,\,\tom),\,H_\tP^\tth(\tom)\big) && \text{(by \ref{main}(e))}\\
&\cong \hom_\tR\big(\tR/\fa,\,H_\tP^\tth(\tom)\big) \cong \ann_{H^{\tth}_{\tP}(\tom)}\fa &&\text{(by part~(a))}.
\end{align*}
Hence, \twg, we have $\ann_{H^{h}_{P}(\om)}\tP = \ann_{H^{\tth}_{\tP}(\tom)}\tP \cong A$ according to part~(b) 
of Theorem~\ref{main}. Over $R$, we see that, \wg, $\ann_{H^{h}_{P}(\om)}P \cong A$.

(c) Let $0 \to M' \to M \to M'' \to 0$ be a sequence of finitely generated $R$-modules that becomes exact after 
localization at $P$. Considered over $\tR$, this is a sequence of finitely generated $\tR$-modules that becomes 
exact after localization at $\tP$. By part~(a) of Theorem~\ref{main}, the induced sequence
$$0 \to \Hom_\tR(M'',\, \tE) \to \Hom_\tR(M,\, \tE) \to\Hom_\tR(M',\, \tE) \to 0$$ 
is \twg exact, in which $\tE := H^{\tth}_{\tP}(\tom)$. Moreover, part~(b) says that 
$E = \hom_\tR\big(\tR/\fa,\,\tE\big)$. Thus, the induced sequence
$$0 \xra{\quad} \Hom_R(M'',\, E) \xra{\quad} \Hom_R(M,\, E) \xra{\quad}\Hom_R(M',\, E) \xra{\quad} 0$$ 
is \wg exact.  

Next, for any finitely generated $R$-module $M$, there exists a short exact sequence $0 \to L \to R^n \to M \to 0$ 
of $R$-modules for some $n \in \N$, which forces $L$ to be finitely generated over $R$ as well. Then, as 
proved above, the induced sequence
$$0 \xra{\quad} \Hom_R(M,\, E) \xra{\quad} \Hom_R(R^n,\, E) \xra{\quad}\Hom_R(L,\, E) \xra{\quad} 0$$ 
is \wg exact, which forces $\ext_R^1(M,\,E)$ to vanish \wg. As for $i > 1$, we observe that
$\Ext^i_R\big(M, H^h_P(\om)\big) \cong \Ext^1_R\big(\Omega_R^{i-1}(M),\, H^h_P(\om)\big)$ is \wg $0$ 
because $\Omega_R^{i-1}(M)$, an $(i-1)^{\text{st}}$ syzygy of $M$ over $R$, finitely generated over
$R$. Therefore, the functor $\hom_R(\blank, \, E)$ is \wg exact on 
any given finite set of short exact sequences of finitely generated $R$-modules, and, hence, on
any given finite set of finite long exact sequences of finitely generated $R$-modules. The choice of $g \in W$, 
at which we localize, depends on which finite set of finite exact sequences one chooses.

(d) Let $M$ be a finitely generated $R$-module. In light of part~(b) above, we may identify $\Hom_R(M, E)$ 
(respectively, $H^i_P(M)$) with $\Hom_\tR(M, H^{\tth}_{\tP}(\tom))$ (respectively, $H^i_\tP(M)$).  Now the claim 
follows from parts~(d) and (f) of Theorem~\ref{main}.

(e)  As $\tR_\tP$ is \CM and $R_P$ is \stwo, we apply 
Discussion~\ref{S2canon-g}(1), to $\tR_\tP$ and $R_P$, and obtain $\ext^k_{\tR_\tP}(M_\tP, \tom_\tP) 
\cong \hom_{R_P}(M_P,\om_P)$ over $\tR_\tP$. As all the modules are finitely generated, we see that 
$\ext^k_{\tR}(M, \tom) \cong \hom_R(M,\om)$ \twg. 
By Theorem~\ref{main}(e) applied to $\tR$, we see that, \twg, 
\begin{align*}
H^{h}_{P}(M) 
&\cong \Hom_{\tR}\big(\ext^k_{\tR}(M, \tom),\, H^{\tth}_{\tP}(\tom)\big) &&\text{(by \ref{main}(e))}\\
& \cong \Hom_{\tR}\big(\hom_R(M,\om),\, H^{\tth}_{\tP}(\tom)\big) &&\text{(by above)}\\
& \cong \Hom_{R}\big(\hom_R(M,\om),\, H_P^h(\om)\big) &&\text{(by part~(b))}.
\end{align*}
Over $R$, this says that $H^{h}_{P}(M) \cong \Hom_{R}\big(\hom_R(M,\om), \, H^{h}_{P}(\om)\big)$, \wg. 
Now the claim for $H^{h}_{P}(R)$ is clear, and the proof is complete. 
\end{proof} 

\begin{theorem}\label{oR}  Suppose that $\tR$ is excellent or that conditions~\ref{E1}, \ref{E2} and \ref{E3} all hold. 
Let $\tP$ be prime in $\tR$ with $\fa \inc \tP$. Assume also that $\tR_\tP$  
is \CM, and  that $\tR$ has a \fg module $\tom$ such that $\tom_\tP$ is
a canonical module for $\tR_\tP$.  Let $R := \tR/\fa$, let $P:= \tP/\fa$, 
and assume as well that $R_P$ is \stwo.  
Let $\tP$ have height $\tth$ in $\tR$ and let $P$ have height $h$ in $R$, which imply that 
$\fa_\tP$ has height $k:=\tth - h$ in $\tR_\tP$.
Let $\om := \Ext^k_{\tR}(R, \, \tom)$.
Then we can localize at one  element of $\tR\sm \tP$ in such a way
 that all of  the following  statements hold:
\bena 
\item The ring $A$ is regular, the ring $\tR$ is Cohen-Macaulay, and the module $\tom$ 
 is a global canonical module for $\tR$. 
\item  $R$ is \stwo, locally equidimensional, and all minimal primes of $\fa$ are contained in $\tP$ 
 and have height $k$. 
\item  $H^h_P(\om) \cong \hom_\tR\big(\tR/\fa, H_\tP^\tth(\tom)\big)$ and, hence 
$\ann_{H^h_P(\om)}P \cong \hom_\tR\big(\tR/\tP, H_\tP^\tth(\tom)\big)$. Moreover, $\ann_{H^h_P(\om)}P \cong A$. 
Write $\ann_{H^h_P(\om)}P = Au$, for some $u \in H^h_P(\om)$, so that the image of $u$ in 
$H^h_P(\om)_P$ is a socle generator for $H^h_P(\om)_P \cong \E_{R_P}(\kappa_P)$ over $R_P$. \label{oR-Au}
\item For all primes $\tQ \in V(\tP)$, and for all $\uy:=\vect y d \in \tQ$ that maps to a \sop for $A_\tQ$, 
 we have  $H^d_{(\uy)}\big(H^h_P(\om)\big)_Q \cong \E_{R_Q}(\kappa_Q)
  \cong H^{d+h}_{Q}(\om)_Q$ over $R_Q$.  
\item Further assume that the sequence $\uy:=\vect y d \in \tQ$ maps to a regular \sop for $A_\tQ$. 
Then, with $u$ as in part~\ref{oR-Au} above, 
the element $[u; \uy] \in H^d_{(\uy)}\big(H^h_P(\om)\big)_Q$  
(see Notation~\ref{lcnot}) is a socle generator for the module $H^d_{(\uy)}\big(H^h_P(\om)\big)_Q 
\cong \E_{R_Q}(\kappa_Q)$ over $R_Q$. 
\een
\end{theorem}

\begin{proof} In light of Proposition~\ref{e12} and Remark~\ref{S2}, we may localize at one 
element of $\tW$ and assume that $A = \tR/\tP$ is regular, that $\tR$ is Cohen-Macaulay, 
that $\tom$ is a global canonical module for $\tR$,  
and that $R:= \tR/\fa$ is \stwo and, hence, locally equidimensional.  
Moreover, we can localize at one element of $\tW$ so that all minimal 
primes of $\fa$ in $\tR$ are contained in 
$\tP$, and so we may assume that all minimal primes of $\fa$ have the 
same height, which is necessarily $k$, which will be the 
height of $\fa$ in $\tR$. This verifies parts (a) and (b).

The isomorphisms in part~(c) concerning $H^h_P(\om)$ and $\ann_{H^h_P(\om)}P$ are verified in 
part~(b) of Theorem~\ref{main-g}. Moreover, the isomorphism $H^h_P(\om)_P \cong \E_{R_P}(\kappa_P)$ is shown in  Discussion~\ref{S2canon-g}(2). The rest of part~(c), concerning $u$, is clear.

In part~(d), note that part~(c) says that $H^h_P(\om) \cong \ann_\tE \fa$, where $\tE: = H_\tP^\tth(\tom)$.
Hence, the isomorphism $H^d_{(\uy)}\big(H^h_P(\om)\big)_Q \cong \E_{R_Q}(\kappa_Q)$ follows immediately from part~(c) of Theorem~\ref{hulls}.
On the other hand, the isomorphism $H^{\dim(R_Q)}_{Q}(\om_Q) \cong \E_{R_Q}(\kappa_Q)$ has been verified in Discussion~\ref{S2canon-g}(2), with $\dim(R_Q) = d+h$. 

Finally, part~(e) follows from part~(d) of Theorem~\ref{hulls}, in light of the isomorphism $H^h_P(\om) \cong \ann_\tE \fa$ as noted above.
\end{proof}

\section{Canonical modules via \'etale extensions}\label{canon} 

There are no restrictions on characteristic in this section.
We need the following result, whose history in the literature we describe below.

\begin{theorem}\label{canh} Let $\rmk$ be an excellent Cohen-Macaulay local ring.  Then the Henselization
$R\h$ of $R$ has a canonical module.   \end{theorem}

\begin{discussion}  Hinich proved \cite{Hin93} that an approximation ring (also called a ring with
approximation property) has a dualizing complex.  In the \CM case, the existence of a dualizing
complex is equivalent to the existence of a canonical module.  Rotthaus gave a more elementary
proof for the Cohen-Macaulay local case in \cite{Rott96}.  She phrases the result in terms
of rings with the complete approximation property, but since the treatment in \cite{Swan98}, following the ideas
of Popescu, \cite{Po85, Po86}, resolved any doubts about the general case of N\'eron-Popescu
desingularization, we know that every excellent Henselian ring is an approximation ring 
(in fact, a local ring is an approximation ring if and only if it is excellent and Henselian).
It is explained in \cite{Rott95} how this implies that every excellent Henselian ring has the 
complete approximation property as well.  We also note that the Henselization of an excellent
local ring is excellent \cite[Cor.~18.7.6]{EGAIV67}. \end{discussion}

We use Theorem~\ref{canh} to prove the following result.  First recall that a local map of local
 rings is called a {\it pointed} \'etale extension if it is a localization at a prime of a (finitely
 presented)  \'etale extension and the induced map of residue class fields is an isomorphsim.
 In the result below, $z$ is an indeterminate over $R$ and if $f \in R[z]$,  $f'$ denotes the
 derivative of $f$ with respect to $z$.

\begin{theorem}\label{etcanon} Let $R$ be an excellent Noetherian ring, with no restriction 
on the characteristic of $R$, and let $P$ be a prime ideal  of $R$ such that $R_P$ is \CM.  
Let $z$ be an indeterminate over $R$. Then after replacing $R$ by its localization
at one element in $R \sm P$,  there exists a flat extension
$\vR$ of the form $\big(R[z]/fR[z]\big)_g$, where $f$ is a monic polynomial in $R[z]$,  $g \in R[z]$, and 
there exists an element $r \in R$ such that $f(r) \in P$ with the properties listed below.
In describing these properties,  for
every prime $Q$ of $R$ such that $Q \supseteq P$,  we use symbol $\vQ$ to denote the prime ideal 
$\Big(\big(QR[z] + (z-r)R[z]\big)/fR[z]\Big)_g$ in $\big(R[z]/fR[z]\big)_g$. 
 
\benn 
\item $R$ is \CM. 
\item $f'(r)$ is a unit of $R$. 
\item $g(r)$ is a unit of $R$. Hence, $g \notin \vQ$ for any choice of $Q \in V(P)$.
\item $\vQ$ is a prime ideal of $\vR$ lying over $Q$ in $R$. 
\item $\vR$ is \CM and has a global canonical module $\om$. 
\item For all $Q \in V(P)$,  $R_Q \to \vR_{\vQ}$ is a pointed \'etale extension.  
\item $\vR/\vP \cong R/P$, and so under the natural map $\spec(\vR) \to \spec(R)$ induced by 
contraction, $V(\vP)$ maps homeomorphically to $V(P)$. 
\een
Moreover, if we replace $g$ by any multiple $g_1$ in $R_P[z]$ such that  $g_1(r) \notin PR_P$
 (or equivalently, $g_1 \in R_P[z] \sm \big(PR_P[z]+(z-r)R_P[z]\big)$), 
then after replacing  $R$ by its localization at one element of  $R\sm P$,  $g_1$ has coefficients in $R$, 
$g_1(r) \in  R$ is a unit, and all of the above properties $(1)\,$--$\,(7)$ hold for the extension 
$R \to \big(R[z]/fR[z]\big)_{g_1}$. 
\end{theorem}
\begin{proof}
By Theorem~\ref{canh}, we may choose
a canonical module $\om_1$ for the Henselization  $(R_P)\h$ of $R_P$, and represent it as
the cokernel of a matrix $\cM_1$.  The ring $(R_P)\h$ is a direct limit of pointed \'etale extensions of $R_P$.
Hence, $\cM_1$ will descend to a matrix $\cM_2$ over a suitable pointed \'etale extension of $R_P$, and the 
cokernel of $\cM_2$, call it $\om_2$,  will have the property that $(R_P)\h \otimes \om_2 \cong \om_1$.
It follows that $\om_2$ is a canonical module for this pointed \'etale extension of $R_P$.

By the structure theorem for pointed \'etale extensions of a local ring, the local \'etale extension
has the the from $\big(R_P[z]/(f)\big)_\fP$,  where $f = f(z)$ is a monic polynomial in $\fP \subseteq R_P[z]$,  
and $\fP$ is a prime ideal of $R_P[z]$  of the form  $(P + (z - r))R_P[z]$,  where  $r \in R_P$ represents 
a simple root  $\rho$ of the image of$f$ mod $P$ in $\kappa_P$.  The condition that $\rho$ be a simple root 
is equivalent to assuming t hat $f'(r) \notin PR_P$.  

Hence,  we may assume that $\om_2$ is a module over an extension ring
$R_P$  of the form   $\big(R_P[z]/fR_P[z]\big)_\fP$.  However, instead of localizing at all elements 
not in $\fP$,  we may descend to an algebra of the form $\big(R[z]/fR[z]\big)_g \cong R[z]_g/fR[z]_g$:  
for a suitable choice of $g$ $\in R[z] \sm \big(PR[z] + (z-r)R[z]\big)$,  $\om_2$ will
descend to a module $\om$ over $\big(R[z]/fR[z]\big)_g$ such that $\om_\fP$ is a canonical module
for $\Big(\big(R[z]/fR[z]\big)_g\Big)_{\vP} = \big(R_P[z]/fR_P[z]\big)_\fP$.    The condition that 
$g \notin PR[z] + (z-r)R[z]$ is equivalent to the assumption that $g(r) \notin P$. 

By replacing $R$ with its localization at an element
of $R \sm P$, we may additionally assume that $f(z)$ is monic over $R$ , that $g(z) \in R[z]$, that
$r \in R$, that $f'(r)$ is a unit of $R$ (since it is not in $P$) and that $g(r)$ is a unit of $R$.  We are now
in the situation described in the statement of the theorem, and we may make a preliminary choice of
$\vR$ to be $\big(R[z]/fR[z]\big)_g$.  
Note that the natural image of $\fP$, mod $f$, is the localization of $\vP$ at itself. 
However, $\om$ will not yet  necessarily be a global canonical module over $\vR$.  However, 
by Proposition~\ref{e12}, we can localize further at multiple of $g$ not in 
$PR[z] + (z-r)R[z]$ such that this is true, and we may localize
further at one element of $R$ not in $P$ so that for the new choice of $g$ we have that $g(r)$ is
a unit in $R$.  It is now straightforward to verify that the conditions $(1)\,$--$\,(7)$ all hold, and that, after replacing
$R$ by its localization at one element not in $P$, they
continue to hold if we localize further at polynomial in $R_P[z]$ that is not in  $PR_P[z] + (z-r)R_P[z]$:  we may
localize at one element of $R$ not in $P$ so that the new polynomial $g_1$ is in $R[z]$ and so that
$g_1(r)$ is a unit. 
\end{proof} 

\section{The purity exponent}\label{pureexp}  


\subsection{Purity}\label{pur} In this subsection, there are no assumptions on the characteristic of 
the ring and there are no finiteness assumptions on modules, unless such assumptions are 
explicitly stated. A map of $R$-modules $N \to M$ is called {\it pure} if for every $R$-module 
$V$, the induced map $V \otimes N \to V \otimes M$ is
injective.  In particular, $N \to M$ is injective.  This is a weakening of the condition that $N$ be 
a direct summand of $M$ over $R$.  In fact, if $0 \to N \to M \to C \to 0$ is exact and $C$ is finitely presented, 
then $N \to M$ is pure if and only if it the map is split as a map of $R$-modules; that is, $M$ is 
the internal direct sum of the image of $N$ and a submodule $N'$ so that the composite map 
$N \to M \surj M/N'$ is an isomorphism.
A direct limit of pure maps is pure.  Purity is preserved by arbitrary base change and, in particular 
by localization at any multiplicative system of $R$.  Moreover, $N \to M$ is pure if and only if for 
all maximal (respectively, all prime) ideals $P$ of $R$,  $N_P \to M_P$ is pure over $R_P$. For 
further information about purity, we refer the
reader to \cite[\S6]{HoR74}, \cite[\S5.(a)]{HoR76}, and \cite[Lemma 2.1]{HH95}. 

\subsection{Purity exponents: definition and basic facts}\label{pure-e} Throughout the remainder of this section, 
$R$ is a Noetherian ring of prime characteristic $p >0$, and $c$ is an element in $R$. 
In the sequel, $F_R^e$ or simply $F^e$ will denote the $e\,$th iterate of the Frobenius
endomoprhism on $R$, and $\e R$ will denote $R$ viewed as an $R$-algebra via the structural 
homomorphism $R \arwf{F^e} R$. An alternative notation for $\e R$ is  $F^e_*(R)$.  
We write $\e  c$ for the element corresponding to $c$ in $\e R$.
If $W \inc R$, we shall also write $\e\, W$ for the set $\{\e w: w \in W\}$.  
In particular, if $P$ is a prime ideal of $R$,  $\e P$ is the corresponding
prime ideal of $\e R$, and $\e P$ is the radical of the extended ideal $(P)\e R = \e \big(P^{[p^e]}\big)$.

We use the notation  $\theta^R_{e,c}$ (or $\theta_{e,c}$ if the ring $R$ is understood from
context) for the $R$-linear map  $R \to \e R$ 
such that $1 \mapsto \e c$.\footnote{\label{reduced-1/q}When $R$ is reduced, 
we may identify $\e c \in \e R$ with $c^{1/p^e} \in R^{1/p^e}$; hence, 
we may write $\theta_{e,c}$ as $\theta_{e,c}: R \to R^{1/p^e}$ such that $1 \mapsto c^{1/p^e}$.}
The \emph{purity exponent} $\fe_c$ of $c$ is defined to be the least nonnegative integer $e \in \N$ 
such that the map $\theta^R_{e'\!,c}:R \to \nw{e'} R$ is pure over $R$ for all $e' \ge e$, if such an exponent exists.
If no such exponent exists, then $\fe_c$ is defined to be $\infty$.    
Note that $R$ is F-pure if and only if $\fe_1 = 0$ if and only if $\fe_1 < \infty$, and also if and only if $\theta_{e,1}$ is pure for some (equivalently, all) $e \geq 1$.  
If $c'$ is a divisor of $c$ we have that $\fe_{c'} \leq \fe_{c}$, since if $c = c'c''$,  $\theta_{e,c}$ is the composition
$ (\e c'' \cdot \blank) \circ \theta_{e,c'}$.  Hence, the finiteness of $\fe_{c}$ for any $c \neq 0$ implies that $c$ is not
a zerodivisor and that the  ring is F-pure, and therefore reduced. Since these conditions are forced by the
finiteness of $\fe_c$, we do not assume them in general, which may be useful when we consider what happens as we localize
at $R$ a varying prime ideal.

Let $e \ge 1$. Once $\theta_{e,c}$ is pure, this also holds for $\theta_{e'\!,c}$ for all $e' \geq e$. 
It suffices to see this for $e' = e+1$.  But the map $\theta_{e,c}$ induces 
a $\nw{1}R$-pure map (consequently, an $R$-pure map) 
${}^1 \theta_{e,c}:\nw{1}R \to \nw{e+1}R$ such
that $\nw{1} 1_{\nw{1}R} \mapsto \nw{e+1}c$, and we may compose with the $R$-pure map $\theta_{1,1}$ to obtain that 
$\theta_{e+1,c} = {}^1\theta_{e,c} \circ \theta_{1,1}$ is pure. This gives us the following:

\begin{remark}\label{sA} For every $n \ge 0$, let $\sA_n(R)$, usually shortened to  $\sA_n$, 
denote the set $\{r \in R \,:\, \theta_{n, r} \text{ is not pure}\}$. Then $\sA_0$ consists of the non-units,
and $\sA_{n+1} \inc \sA_n$ for all $n \ge 1$.
Note that $\fe_c$ is finite if and only if $c \notin \bigcap_{n=1}^\infty \sA_n$, in which case $\fe_c = 0$
if $c \notin \sA_0 \cup \sA_{1}$ and otherwise $\fe_c$ is the least integer $e \in \N_+$ such that $c \notin \sA_e$.
We shall show shortly, in Theorem~\ref{sAchar}, that when $(R, \m, \kappa)$ is local, $\sA_n$ is an ideal. 
\end{remark}

\begin{definition} Let $c \in R$. For $\fp \in \spec(R)$, we use $\fe_c(\fp)$ to denote $\fe_{c/1}$ in the ring $R_\fp$. This defines a function 
$\fe_c : \spec(R) \to \N \cup \{\infty\}$ where $\fp \mapsto \fe_c(\fp)$.  \end{definition}

Hence, we have the following:

\begin{proposition}\label{purexploc} Let  $R$ be a Noetherian ring of prime characteristic $p$, and $c \in R$. 
\bena
\item $\fe_c(P) \leq \fe_c(Q) \leq \fe_c$ for all prime ideals $P \subseteq Q$ in $R$. 
\item $\fe_c = \sup_{\fp \in \Spec(R)} \fe_c(\fp) = \sup_{\fp \in \Max(R)} \fe_c(\fp)$. 
\een
\end{proposition} 

\begin{notation}\label{purecrit} Let $M$ be an $R$-module. 
For $e \in \N$, we write $q:=p^e$, $F^e_R(M) := \e R \otimes_R M$ and view $F^e_R(M)$ as an 
$R$-module via $r \cdot \sum_i \e r_i \otimes m_i = \sum_i \e (rr_i) \otimes m_i$ for all $r \in R$ and 
$\sum_i \e r_i \otimes m_i \in \e R \otimes_R M = F^e_R(M)$. 
Also, for $x \in M$, denote  $x^q_M:= 1 \otimes x \in \e R \otimes_R M = F^e_R(M)$. 
When the module $M$ is clear in the context, we may write $x^q _M$ as $x^q$. 
Moreover, we write $\E_R(M)$ for the injective hull over $R$ of the $R$-module $M$, which is unique up to 
non-unique isomorphism.    When $(R,\fm,\kappa)$ is local, we may use the notation $E_R$ or even $E$ 
to denote $\E_R(\kappa)$.  
\end{notation}

When $\fe_c$ is finite, it gives a ``tight closure style"  test exponent for membership:

\begin{proposition} \label{membership} Let $R$ be a Noetherian ring of prime characteristic $p >0$ and let $c \in R$ 
be such that $\fe_{c} < \infty$. Let $M$ be any $R$-module, not necessarily \fg, $N$ be an $R$-submodule of $M$, 
$u \in M$, and $e$ be an integer such that $e \ge \fe_c$. Then $u \in N$ if and only if
$cu^{p^{e}}_M \in N^{[p^{e}]}_M$. \end{proposition}

\begin{proof}  Since $\theta_{e,c}: R \to \nw{e} R$, determined by $1 \mapsto \nw{e}\,c$, 
is pure, the map remains injective if we apply $\blank \otimes_R (M/N)$ to obtain
a map $\theta': R \otimes_R (M/N) \to \nw{e} R  \otimes_R (M/N)$.  Now $cu^{p^{e}}_M \in N^{[p^{e}]}_M$
if and only if the class of $u$ in $M/N$ is contained in $\ker(\theta') = 0$ if and only if $u \in N$. 
\end{proof}  

\begin{remark}  The statement in Proposition~\ref{membership} is particularly useful when $R$ is very strongly F-regular (see subsection~\ref{vstfreg}), because in that case $\fe(c) < \infty$ for all $c \in R\0$, 
where $R\0 = R \sm \bigcup_{\fp \in \Min(R)}\fp$. 
Moreover, one of our main results,
Theorem~\ref{maine}, shows that every excellent strongly F-regular is very strongly F-regular. \end{remark}

\begin{theorem}\label{sAchar} Let notation be as in \ref{purecrit}, with $(R, \m, \kappa)$ local, and let $v$ be a socle generator of $E = \E_R(\kappa)$. Then for all $e$,  $\sA_e$ is the annihilator in $R$ of $v^q \in F^e_R(E)$. 
Thus, for $e \ge 1$, $\fe_c \le e$ if and only if $\theta_{e,c}$ is pure if and only if $0 \neq cv^q \in F^e_R(E)$. 
\end{theorem} 

\begin{proof}  By \cite[Lemma 2.1(e)]{HH95}, 
the map $\theta_{e,c}: R \to \e R$, with $1 \mapsto \e c$, is pure if and only
the induced map $\theta_{e,c} \otimes_R \id_E$ is injective, which holds if and only if
the element $1 \otimes v$ in $R \otimes_R E$ does not map to 0, i.e., $0 \neq \e c \otimes v \in \e R \otimes_R E$.  
In terms of Notation~\ref{purecrit}, we see that $\theta_{e,c}$ is pure if and only if $0 \neq cv^q \in F^e_R(E)$. The rest of the claims, concerning $\sA_e$ or $\fe_c$, are clear now, given that $\sA_{n+1} \inc \sA_n$ for all $n \ge 1$.
\end{proof}

\subsection{Flat regular extensions} 
We have the following:

\begin{theorem}\label{geomreg} Let $\phi: R \to S$ be a flat ring homomorphism and let $c\in R$.  
\bena
\item If $\phi: R \to S$ is faithfully flat (e.g., $(R, \fm) \to (S, \fn)$ is local), then $\fe_c \leq \fe_{\phi(c)}$. 
\item If $(S/\phi^{-1}(Q)S)_Q$ is regular for all $Q \in \Max(S)$ (e.g., $(R, \fm) \to (S, \fn)$ is local with regular closed fiber), then $\fe_{\phi(c)} \leq \fe_c$. \een
\end{theorem}
\begin{proof} In (a), if $S \to \e S$ with $1 \mapsto \e \phi(c)$ is pure over $S$, and so over $R$. Since $R \to S$ is 
faithfully flat, we have that $\phi: R \to S$ is pure over $R$, and so $R \to \e S$ with $1 \mapsto \e c$
is pure.  Since this map factors $R \to \e R \to \e S$, where the first map is $\theta_{e,c}$, it follows that
$\theta_{e,c}$ is pure over $R$.  

(b) For any $\fn \in \Max(S)$, let $\fm = \phi^{-1}(\fn)$ be its contraction to $R$. If $\fe_{\phi(c)}(\fn) \le \fe_c(\fm)$ 
for all $\fn \in \Max(S)$, then $\fe_{\phi(c)} = \sup_{\fn \in \Max(S)}\fe_{\phi(c)}(\fn) \leq \fe_c$. Thus, this reduces to 
the local case $\phi: (R, \fm) \to (S, \fn)$ with regular closed fiber. Now the result follows from \cite[Lemma~2.15]{HoY23}.
\end{proof}

\subsection{Strongly F-regular and very strongly F-regular rings}\label{vstfreg}
Let $R\0$ denote the complement of the union of all minimal primes of $R$, i.e., $R\0 = R \sm \bigcup_{\fp \in \Min(R)}\fp$. 
If $\fe_c$ is finite for all $c \in R\0$ then $R$ is called  {\it very strongly
F-regular} in the terminology of \cite{Hash10} and {\it F-pure regular} in the terminology of \cite{DaSm16}. We will use the
terminology ``very strongly F-regular."   If this condition holds for all local rings
of $R$ at maximal ideals (equivalently, prime ideals) then $R$ is called {\it strongly F-regular.}  This terminology
is proposed in \cite{Ho07}, and it is also used in \cite{Hash10}.  We also note that $R$ is strongly F-regular in this
sense if and only if one of the following two equivalent conditions holds:
\ben
\item For every $R$-module $M$ and every $R$-submodule $N \inc M$ (with no assumption about finite generation),
$N$ is tightly closed in $M$.
\item For every maximal ideal $\fm$ in $R$,  the submodule $0$ is tightly closed in the injective hull $\E_R(R/\fm)$.
\een
For further background on the theory of strongly F-regular rings, we refer the reader to
 \cite{Ab01, AL03, HH89, HH94b, HL02, LS99, PT18, SchSm10, Sm00, Tu12, Yao06}.

One of our main goals in the sequel is to prove that strongly F-regular excellent rings are very 
strongly F-regular. This is an obvious question, raised, for example, in \cite{DaSm16}.   
It is clear that a strongly F-regular ring is very strongly F-regular if $R$ is local. 
It is also known that if $R$ is F-finite or essentially of finite type over an excellent semilocal ring then strongly 
F-regular rings are very strongly F-regular (and under somewhat weaker hypotheses: see \cite{DEST25}).  We refer to \cite[\S2]{HoY23} for a thorough discussion of previously known results, many of which may be found  in \cite{Ho07}, \cite{Hash10}, \cite{DaSm16}, \cite{DEST25} as well as in \cite{HoY23}. Another of our main goals is
to prove that for every excellent Noetherian ring of prime characteristic $p >0$,  the strongly F-regular
locus is open. This was previously known only in special cases.  

\begin{discussion}{\bf Calculation of \boldmath $\sA_e$ from a canonical module.} Let $(R, \fp, \kappa)$ be \CM with canonical module
 $\om_R$,  $\uz =\vect z s$ be a \sop for $R$, and $v$ be a socle generator in the local cohomology  $E:= H^s_{(\uz)}(\om)$ corresponding
 to a socle element that is the image of $u$ in $\om_R/(\uz)\om_R$, where $u \in \om_R$. 
Then $$\sA_e =\dlim t \Big((\vct z {qt} s) F^e_R(\om_R):_R z^{qt-q}u^q\Big),$$ where $z=\prod_{j=1}^sz_j$ is the product of the $z_j$, and  
where $u^q$ is in $F^e_R(\om_R)$.  Note that, by Theorem~\ref{sAchar}, $\sA_e$ is $0:_{R} v^q = \ann_R(v^q)$, where $v^q \in F^e_R \big(H^s_{(\uz)}(\om)\big) \cong H^s_{(\uz)}\big(F^e_R(\om)\big)$.
\end{discussion}

\subsection{Semicontinuity} 
Our object is to prove the following result.

\begin{theorem}\label{mainsmct} Let $R$ be a homomorphic image of an excellent Cohen-Macaulay Noetherian 
ring of prime characteristic $p >0$ such that $R$ is \stwo and let $c \in R$. 
Then, for any given $e$ $\in \N \cup \{\infty\}$, the set $\{\fp \in \Spec(R): \fe_c(\fp) \leq e\}$ is Zariski open. 
In other words, the function $\fe_c : \spec(R) \to \N \cup \{\infty\}$ is upper semicontinuous. 
\end{theorem}

Before proving Theorem~\ref{mainsmct}, we record an important consequence:

\begin{theorem}\label{maine} If an excellent Noetherian ring $R$ of prime characteristic $p >0$ is strongly F-regular, then it is very strongly F-regular. 
\end{theorem}

\begin{proof}  Let $c \in R\0$. Since the ring is strongly F-regular, for every $P \in \Spec(R)$, the purity
exponent  $\fe_c(P)$ is finite for $c/1 \in R_P$.   By Theorem~\ref{mainsmct}, this defines a Zariski open neighborhood
$U_P:= \{Q \in \Spec(R): \fe_c(Q) \leq \fe_c(P)\}$ of $P$.  The open sets $U_P$, $P \in \Spec(R)$, 
cover $\Spec(R)$ and have a finite subcover, $U_{P_i}$, $1 \leq i \leq m$.  Let
$e$ be the maximum of the integers $\fe_{c}(P_i)$, $1 \leq i \leq m$.  Then $(\theta_{e,c})_P$ is pure 
over $R_P$ for all $P \in \Spec(R)$, and so $\theta_{e,c}$ is pure over $R$.  See subsections~\ref{pur} and \ref{pure-e}.  
\end{proof}

To prove Theorem~\ref{mainsmct}, we will need several preliminary results, including  
the following result on openness, based on an idea of Nagata. See \cite[Theorem~24.2]{Mat87} and 
\cite[Lemma~5.16.5]{StProj}  for a generalization to Noetherian topological spaces.
Note that there is no assumption about the characteristic of $R$.

\begin{theorem}\label{thm:open}  Let $R$ be any ring, and $U \inc \Spec(R)$.  Consider the conditions:
\beni
\item For $P,\, Q \in \Spec(R)$, if $P \inc Q$ and $Q \in U$, then $P \in U$. 
\item For every $P \in U$, $U \cap V(P)$ contains a non-empty open subset of $V(P)$. 
\een
Then
\bena
\item If $U$ is open then $U$ satisfies {\rm (i)} and {\rm (ii)}. 
\item Assume that $R$ is Noetherian. If $U$ satisfies {\rm (i)} and {\rm (ii)} then $U$ is open.  
\een
\end{theorem}

To prove Theorem~\ref{mainsmct}, it suffices to assume that $e \in \N$ and to show that the set 
$U_{c,e} := \{\fp \in \Spec(R): \fe_c(\fp) \leq e\}$ is open in $\spec(R)$. Proposition~\ref{purexploc}(a) says 
that, for $P,\, Q \in \Spec(R)$, if $P \inc Q$ and $Q \in U_{c,e}$, then $P \in U_{c,e}$. To apply 
Theorem~\ref{thm:open}, it remains to show that, for every $P \in U_{c,e}$, $U_{c,e} \cap V(P)$ contains 
a non-empty open subset of $V(P)$. Theorem~\ref{mainsmct} now follows at once from the 
Key Lemma (i.e., Lemma~\ref{mainloclemma}) that immediately follows. In proving the Key Lemma, 
we will need to replace $\tR$ by an \'etale extension, as in \S\ref{canon}. 
We will also need to view $R$ as $\tR/\fa$, as in Theorem~\ref{oR}. It may be helpful to the reader to
review the notation from \S\S\ref{hull}--\ref{canon}.

\begin{lemma}[Key Lemma]\label{mainloclemma}
Let $R$ be a homomorphic image of an excellent ring $\tR$ of prime characteristic $p >0$, say $R = \tR/\fa$, 
let $c \in R$, let $e \in \N$ be a nonnegative integer.  Let  $\tP \in \Spec(\tR)$ be such that $\fa \inc \tP$ and   
$\tR_\tP$ is \CM.  Denote $P := \tP/\fa$, and assume that $R_P$  is \stwo. Suppose that 
$\fe_c(P) \leq e$. Then there exists $g\in R\sm P$ such that $\fe_c(Q) \le e$ for all $Q \in D(g) \cap V(P)$. 
\end{lemma} 
 
\begin{proof} We may localize $\tR$ repeatedly (but only finitely many times) at elements $g \in \tR\sm \tP$, 
 and we shall do this finitely many times in the course of the proof.  Each time, we change notation and continue
 to use $\tR$, $\tP$, $R$, etc. In conquence, we my assume, for example, 
 that $\tR$ is Cohen-Macaulay, that $R$ is \stwo, and that
 $A:= R/P$ is regular. We want to reduce to the case where $\tR_\tP$ has a canonical module (and that  
 we may assume that it has the form $\tom_\tP$, where $\om$ is a global canonical module for $\tR$).  To this
 end, we use Theorem~\ref{etcanon} to replace $\tR$ by a suitable \'etale extension $\vtR$.  
 In doing this, we may need to localize
 at another element of $\tR \sm \tP$. Then $\vtR$ has a global canonical module $\vtom$ (we use this
 notation, but to be clear, at this point we do not have a module $\tom$ over $\tR$ that somehow gives rise to $\vtom$). 
 Let $\vR := \vtR\otimes_\tR R$, which is still \stwo, since this is an \'etale  extension of $R$.
 The purity exponent  $e$ of $c$ in $R_P$ is the same as the purity exponent of the image of $c$ in 
 $\vR_{\vP}$, with notation as in Theorem~\ref{etcanon},
 since pointed \'etale extensions are geometrically regular and we may apply 
 Theorem~\ref{geomreg}. Suppose that we know the theorem for $\vR$, so that $\fe_c(\vQ) \le e$ 
 for all $\vQ$ in an open neighborhood of $\vP$ in $V(\vP)$.
Since $R_Q \to \vR_{\vQ}$ is pointed \'etale and hence faithfully flat, we have 
$\fe_c(Q) \le \fe_c(\vQ)$ by Theorem~\ref{geomreg}(a). It follows from 
Theorem~\ref{etcanon}(7) that $\fe_c(Q) \le e$  
as well for all $Q$ in an open neighborhood of $P$ in $V(P)$. Hence, it suffices to verify the 
Key Lemma after the pointed \'etale  extension from $\tR$ to $\vtR$. 

Therefore, in the remainder of the proof we 
 may assume that $\tR$ has a canonical module $\tom$.  For any prime $\tQ$ of $\tR$ in $V(\fa)$, 
 we write $Q = \tQ/\fa$  and $\kappa_\tQ := \tR_\tQ/\tQ R_\tQ \cong R_Q/QR_Q =: \kappa_Q$. 
After localization at one element of $\tW := \tR\sm \tP$,  we may assume that all of the conclusions 
of Theorem~\ref{oR} hold, and we shall use the notation of that theorem. 
After such a localization, we still denote the rings as $\tR$ and $R$. 
Note that if $\fe_c(P) = 0$ and if there exists $g\in R\sm P$ such that $\fe_c(Q) \le 1$ for all $Q \in D(g) \cap V(P)$, 
then $\fe_c(Q) = 0$ for all $Q \in D(cg) \cap V(P) \neq \emptyset$. So it suffices to assume $e \ge 1$ next.
 
Let $\om = \Ext^k_\tR(R,\, \tom)$, where $k = \dim(\tR_\tP) - \dim(R_P)$, be defined as in Theorem~\ref{oR},
and let $E := H^h_P(\om)$, where $h = \dim(R_P)$, so that $E_P$ is an injective hull for $\kappa_P$ over $R_P$. 
As in Theorem~\ref{oR}, after localization at an element in $W$, 
we have $\ann_EP \cong A$. Write $\ann_EP = Au$, for some $u \in E$. 
The image of $u$ in $E_P$ generates the socle.  Note that in this theorem and its proof, $q = p^e$ is 
fixed, and the hypothesis that $\fe_c(P) \leq e$ tells us that $cu^q \neq 0$ in 
$F_{R_P}(E_P) \cong F^e_R(E)_P$. 
Up to radical, the ideal $PR_P$ can be generated by $h$ many elements in $R_P$. Thus, after localization 
at an element in $W$, we assume that $P$ is the radical of an ideal generated by $h$ many elements in $R$. 
Consequently, we may identify the following naturally isomorphic $R$-modules: 
$$
F^e_R(E) \cong F^e_R\big(H^h_P(\om)\big) \cong H^h_{P^{[q]}}\big(F^e_R(\om)\big) 
\cong H^h_P\big(F^e_R(\om)\big).
$$
Thus, we know that  $0 \neq cu^q \in H^h_P\big(F^e_R(\om)\big)_P$. 
Let $M := H^h_P\big(F^e_R(\om)\big)$. Here, we think of $F^e_R(\om)$ as simply a fixed, 
finitely generated $R$-module.  
By Theorem~\ref{main}(f) and Proposition~\ref{ff}(c), after localization at one element of $W$,  
we have that $M$ and $M/R(cu^q)$ are \ff, so that we may apply Corollary~\ref{Annihilator}   
to $R'\big(1_{R'} \otimes_R (cu^q)\big) \inc R' \otimes_RM$, with any choice of $R'$ that
is flat over $R$.  

Now fix an arbitrary $\tQ \in V(\tP)$ and let $\uy :=\vect y d \in \tR$  map to a regular \sop in the regular 
local ring $A_\tQ$.  By Theorem~\ref{oR}, we have 
$\E_{R_Q}(\kappa_Q) \cong H^d_{(\uy)}\big(H^h_{P}(\om)\big)_Q$ over $R_Q$ and, 
with notation as in \ref{lcnot}, we may take $v :=[u; \uy]$ as a socle generator of 
$H^d_{(\uy)}\big(H^h_{P}(\om)\big)_Q$. 

We may identify $F^e_{R_Q}\big(\E_{R_Q}(\kappa_Q)\big) \cong 
F^e_{R_Q}\Big(H^d_{(\uy)}\big(H^h_{P}(\om)\big)_Q\Big) \cong
H^d_{(\uy^q)}\big(F_R^e(E)\big)_Q$
and write $v^q = [u; \uy]^q = [u^q;\uy^q] \in H^d_{(\uy^q)}\big(F_R^e(E)\big)_Q 
\cong F^e_{R_Q}\big(\E_{R_Q}(\kappa_Q)\big)$.
Then $cv^q = [cu^q; \uy^q] \in H^d_{(\uy^q)}\big(F_R^e(E)\big)_Q$, where $cu^q \in F^e_R(E) = M$.  We apply 
Corollary~\ref{Annihilator} with $R' = R_Q$, and with $\uy^q$ replacing $\uy$ as in Remark~\ref{annrmk}, to 
show that the annihilator 
of $[cu^q; \uy^q]$ in $R_Q$ is $\big(\Ann_R (cu^q)\big)R_Q + (\uy^q)R_Q$.   Since $cu^q \in F^e_R(E)$ 
does not become $0$ after localization at $P$, we see that $\Ann_R(cu^q) \inc P$. Hence, we conclude that 
\begin{align*}
\ann_{R_Q}(cv^q) = \ann_{R_Q}([cu^q; \uy^q]) &= \big(\Ann_R (cu^q)\big)R_Q + (\uy^q)R_Q\\
&\subseteq P R_Q+ (\uy^q) R_Q = \big(P + (\uy^q) \big)R_Q \inc QR_Q.
\end{align*}
This means that $0 \neq cv^q \in H^d_{(\uy^q)}\big(F_R^e(E)\big)_Q \cong 
F^e_{R_Q}\big(\E_{R_Q}(\kappa_Q)\big)$.  
By Theorem~\ref{sAchar} applied to $R_Q$, we see 
that the purity exponent for the image of $c$ in $R_Q$ is at most $e$.
In summary, there exists $g \in R\sm P$ such that, for all $Q \in D(g) \cap V(P)$, we have $\fe_c(Q) \le e$. 
\end{proof}

\begin{remark}  The reason for assuming that $R$ is \stwo 
in the Key Lemma~\ref{mainloclemma} is that in this case, for every $P \in \spec(R)$ there 
exists a finitely generated $R$-module $\om$ such that the injective hull of the
residue field for $R_Q$ can be realized as the top local cohomology module of $\om_Q$ 
over $R_Q$, for all prime ideals $Q$ in an open neighborhood of $P$ in $V(P)$.  Then, 
when we apply the Frobenius functor $F^e$, we still have the top local
cohomology of a finitely generated module.  The proof works in a similar fashion whenever 
for every $P \in \spec(R)$ there exists a finitely generated $R$-module $\Psi$ such that the injective
hull of $\kappa_Q$ can be realized as the localization at $Q$ of a top local cohomology with support
in $Q$ of $\Psi$, for all prime ideals $Q$ in an open neighborhood of $P$ in $V(P)$. 
The module $\Psi$ then plays the role of $\omega$ in the argument.
The proof also depends on the fact that, after replacing $R$ with its localization at one element 
of $R \sm P$, the top local cohomology $H^{\dim(R_P)}_P\big(F^e_R(\om)\big)$ is generically \ff. 
We do not know  how to prove the needed facts about being generically \ff unless we can realize the 
injective hulls that arise as localizations of cohomology modules.
\end{remark}

Next, note the following consequences of Theorem~\ref{mainsmct}:
 
\begin{corollary}\label{pure-locus} Let $R$ be a Noetherian ring of prime characteristic $p >0$ such that $R$ 
is \stwo and is a homomorphic image of \CM excellent ring.  Then the F-pure locus is open in $R$. \end{corollary}

\begin{proof}  Take $c = 1_R \in R$, $e = 0 \in \N$, and apply Theorem~\ref{mainsmct}.
\end{proof}

\begin{theorem}\label{strong-locus} Let $R$ be an excellent Noetherian ring of prime characteristic $p >0$. 
Then the strongly F-regular locus  in $R$ is a Zariski open set. \end{theorem} 

\begin{proof} Let $P \in \Spec(R)$ be such that $R_P$ is strongly F-regular. It suffices to show
that every such $P$ has a Zariski open neighborhood in $\Spec(R)$ that is contained in the strongly F-regular locus.
We localize several times  at elements of $R \sm P$ while retaining the notation $R$ for the ring. 
Since the \CM locus is open in an excellent ring, we may localize at an element in $R \sm P$ and assume that $R$ 
is Cohen-Macaulay (hence $R$ is \stwo). By Corollary~\ref{pure-locus}, 
we may also assume that the ring $R$ is F-pure, hence reduced, after localizing at an element in $R \sm P$.   
By \cite[Thm.~4.16, p.~5479]{Sharp10}, if $R$ is excellent 
and F-pure and $c \in R\0$ is  such that $R_c$ is regular, then $c$ is a big test element.  
Choose such an element $c \in R\0$, which exists since $R$ is excellent and reduced.  
Since $R_P$ is strongly F-regular, the purity exponent $e := \fe_c(P)$ of $c$ in $R_P$ is finite. 
By Theorem~\ref{mainsmct}, we may replace $R$ by a localization at an element in $R\sm P$ such that 
the purity exponent of $c$ in $R_Q$ is bounded above by $e$ for all $Q \in \Spec(R)$. 
Therefore, by Proposition~\ref{purexploc}(b), we have that $\fe_c = \sup_{Q \in \Spec(R)}\fe_c(Q) = e$.
Let $M$ be any $R$-module, not necessarily \fg. It will suffice to show that $0^*_M = 0$ in order to show that $R$ is strongly F-regular (cf. the first paragraph of subsection~\ref{vstfreg}). 
Now, because $c$ is a big test element, for any $u \in 0^*_M$, we have $0 = cu^{[p^e]}_M \in F^e_R(M)$, 
which implies that $u = 0$ by Proposition~\ref{membership}. 
\end{proof}

\begin{remark}\label{Lyu-Ch8}
We note that using the results of an earlier version of this paper, S.~Lyu \cite{Lyu25} has proved
that Theorem~\ref{mainsmct}, Corollary~\ref{pure-locus} and Theorem~\ref{strong-locus} 
hold under substantially weaker hypotheses. For the first two results one only needs that the condition J$_0$ holds for
$R/P$ for all $P \in \Spec(R)$, i.e., that the regular locus in every $\Spec(R/P)$ contains a nonempty open subset (see \cite[Thm.~A.2.4]{Lyu25}).
For Theorem~\ref{strong-locus}, one needs in addition that for all $P \in \Spec(R)$ the fibers of the
map $R_P \to \wh{R_P}$ are regular (see \cite[Thm.~A.2.10]{Lyu25}).
\end{remark}

Finally, we record the following:

\begin{theorem}\label{maine-gen} Let $R$ be a Noetherian ring of prime characteristic $p >0$ such that for all $P \in \spec(R)$, the regular locus of $R/P$ contains a nonempty open subset of $\Spec(R/P)$. If $R$ is strongly F-regular, then it is very strongly F-regular. 
\end{theorem}

\begin{proof}
The proof of Theorem~\ref{maine} works verbatim, except that \cite[Thm.~A.2.4]{Lyu25} is quoted instead of Theorem~\ref{mainsmct}.
\end{proof}

\end{document}